\documentclass[11pt]{amsart}
\usepackage{comment}
\usepackage{amsmath,amsfonts,amssymb,amstext,amsthm,graphicx,fancyhdr,mathtools,epstopdf}
\usepackage{color}
\usepackage[top=1in, bottom=1in, left=1in, right=1in]{geometry}
\usepackage[titletoc,title]{appendix}
\usepackage{bbm}
\usepackage{bm}
\usepackage{dsfont} 
\usepackage[english]{babel}
\usepackage[T1]{fontenc}
\usepackage{latexsym}
\usepackage{mathrsfs}
\usepackage[scr=boondoxo]{mathalfa} 
\usepackage{graphics}

\usepackage{enumerate, paralist} 
\usepackage{color}
\usepackage{fullpage}

\usepackage[all,cmtip]{xy}

\DeclareFontFamily{OT1}{pzc}{}
\DeclareFontShape{OT1}{pzc}{m}{it}{<-> s * [1.10] pzcmi7t}{}
\DeclareMathAlphabet{\mathpzc}{OT1}{pzc}{m}{it}

\newtheorem{theorem}{Theorem}[section]
\newtheorem{lemma}[theorem]{Lemma}

\newtheorem{proposition}[theorem]{Proposition}
\newtheorem{remark}[theorem]{Remark}
\newtheorem{example}[theorem]{Example}

\newenvironment{Remark}{\begin{remark}\rm}{\end{remark}}

\newtheorem{assmp}{Assumption}

\newcommand{\ang}[1]{\ensuremath{ \left< #1 \right> }}

\newcommand{\E}{\mathbb{E}}
\newcommand{\N}{\mathbb{N}}
\newcommand{\R}{\mathbb{R}}
\newcommand{\Prob}{\mathbb{P}}
\newcommand{\Poly}{\mathscr{P}} 
\renewcommand{\P}{\mathcal{P}} 
\newcommand{\V}{\mathcal{V}} 

\newcommand{\calM}{\mathcal{M}} 

\newcommand{\calD}{\mathcal{D}}

\newcommand{\calI}{\mathds{I}}

\newcommand{\pab}[1][\alpha,\beta]{\mathrm{p}^{(#1)}}
\newcommand{\dpt}{\bm{d}_{\vat,\epsilon}}
\newcommand{\dph}{(d_\phi+1-\epsilon)}
\newcommand{\dz}{{\bm{d}}_{1,\epsilon}}
\newcommand{\h}{\hslash}
\newcommand{\Var}{\operatorname{Var}}
\newcommand{\Ent}{\operatorname{Ent}}
\newcommand{\Ran}{\operatorname{Ran}}
\newcommand{\Ker}{\operatorname{Ker}}

\newcommand{\Iph}{\widehat{\mathcal{I}}_{\phi}}

\newcommand{\Lv}{\lt^2(\vartheta_{\alpha})}
\newcommand{\lt}{{\rm{L}}}

\newcommand{\Rc}{\mathbf{R}}

\newcommand{\Em}{\mathrm{E}}

\newcommand{\calL}{\mathds{L}}
\newcommand{\calLd}{\mathbf{L}}
\newcommand{\calJ}{\mathds{J}} 
\newcommand{\calJd}{\mathbf{J}}

\newcommand{\Nu}{\mathcal{V}}

\newcommand{\m}{\mathfrak{m}}
\newcommand{\mo}{{\mu}}

\newcommand{\Bo}[1]{\mathcal{B}(#1)} 
\newcommand{\Bop}[2]{\mathcal{B}(#1,#2)} 
\newcommand{\lam}{{\bm{\lambda}_1}} 
\newcommand{\tet}{{\mathsf{r}_0}} 
\newcommand{\vat}{{\mathsf{r}_1}} 
\newcommand{\phm}{{\phi_\m^*}}
\renewcommand{\varphi}{{\phi_{\vat}^{\checkmark}}}

\newcommand{\C}{{\mathrm{C}}}

\newcommand{\Lcom}[2]{\lt^{#1}(#2)}

\newcommand{\mm}{{\bm{\delta}}}

\newcommand{\logSob}[1][\mo]{\bm{\lambda}_{\mathrm{logS}}^{\hspace{-0.06cm}(#1)}}

\newcommand{\J}{\mathds{Q}} 
\newcommand{\Q}[1][\mo]{\mathbf{Q}^{({#1})}}

\newcommand{\gab}[1][\lam,\mo]{{\beta_{#1}}}

\newcommand{\gabn}{{b}_{\alpha,\beta+n}}
\newcommand{\bpsi}{\mathscr{\beta}} 

\newcommand{\lga}{\lt^2(\gab)}
\newcommand{\Leb}{\mathrm{L}} 

\newcommand{\Mg}{\mathcal{M}_{\bpsi}}

\renewcommand{\leq}{\leqslant} 
\renewcommand{\geq}{\geqslant}
\newcommand{\norm}[1]{\left\lVert#1\right\rVert}

\numberwithin{equation}{section}

\author{P.~Cheridito}
\address{Department of Mathematics, ETH Zurich, Rämistrasse 101, 8092 Zurich, Switzerland.}
\email{patrick.cheridito@math.ethz.ch}

\author{P.~Patie}
\address{School of Operations Research and Information Engineering, Cornell University, Ithaca, NY 14853.}
\email{pp396@cornell.edu}

\author{A.~Srapionyan}
\address{Center for Applied Mathematics, Cornell University, Ithaca, NY 14853.}
\email{as3348@cornell.edu}

\author{A.~Vaidyanathan}
\address{Center for Applied Mathematics, Cornell University, Ithaca, NY 14853.}
\email{av395@cornell.edu}


\thanks{The authors are grateful to Paul Jenkins and Soumik Pal for fruitful discussions, and they
would like to thank the anonymous referees for valuable comments.}
\usepackage{hyperref}
\usepackage[noabbrev]{cleveref}
\crefname{assmp}{Assumption}{Assumptions}

\title{On non-local ergodic Jacobi semigroups: spectral theory, convergence-to-equilibrium and contractivity}
\begin{document}

\begin{abstract}
In this paper, we introduce and study non-local Jacobi operators, which generalize the classical (local) Jacobi operators. We show that these operators extend to generators of  ergodic Markov semigroups with  unique invariant probability measures and study their spectral and convergence properties. In particular, we derive a series expansion of the semigroup in terms of explicitly defined polynomials, which generalize the classical Jacobi orthogonal polynomials. In addition, we give a complete characterization of the spectrum of the non-self-adjoint generator and semigroup. We show that the variance decay of the semigroup is hypocoercive with explicit constants, which provides a natural generalization of the spectral gap estimate. After a random warm-up time, the semigroup also decays exponentially in entropy and is both hypercontractive and ultracontractive. Our proofs hinge on the development of commutation identities, known as intertwining relations, between local and non-local Jacobi operators and semigroups, with the local objects serving as reference points for transferring properties
from the local to the non-local case.
\end{abstract}

\maketitle

\section{Introduction}

In this paper we study the non-local Jacobi operators, given for suitable functions $f \colon [0,1] \to \R$, by
\begin{equation}
\label{eq:defJ}
\calJ f(x) = \calJd_\mo f(x) - h \diamond f'(x),
\end{equation}
where $\calJd_\mo$ is the classical Jacobi operator
\begin{equation*}
\calJd_\mo f(x)= x (1-x)f''(x)-\left(\lam x-\mo\right)f'(x),
\end{equation*}
$h \colon (1,\infty) \to [0,\infty)$ and $\lam, \mo \in (0, \infty)$ satisfy \Cref{assA} below,
and for two functions $f \colon D_f \to \R$ and $g \colon D_g \to \R$ with domains $D_f, D_g \subseteq [0, \infty)$,
we denote by $f \diamond g$ the product convolution given by
\begin{equation} \label{eq:defd}
f  \diamond g(x) = \int_0^{\infty} f(r) \bm{1}_{D_f}(r) g \left(\frac{x}{r} \right) \bm{1}_{D_g} \left(\frac{x}{r} \right) \frac{1}{r} dr.
\end{equation}
The classical Jacobi operator is a central object in the study of Markovian diffusions. For instance, it is a model candidate for testing functional inequalities such as the Sobolev and log-Sobolev inequalities; see for instance, Bakry \cite{Bakry-Jacobi} and Fontenas~\cite{fontenas:1998}. If $\mo = \frac{\lam}{2} = n$, an integer, there exists a homeomorphism between this particular Jacobi operator and the radial part of the Laplace--Beltrami operator on the $n$-sphere, revealing connections to diffusions on higher-dimensional manifolds that, in particular, lead to a curvature-dimension inequality as described in~Bakry et al.~\cite[Chapter 2.7]{Bakry_Book}. From the spectral theory viewpoint, the Markov semigroup $\Q[\mo] = (e^{t\calJd_\mo})_{t \geq 0}$ is diagonalizable with respect to an orthonormal, polynomial basis of $\Leb^2(\gab[\mo])$, where $\gab[\mo]$ denotes its unique invariant probability measure. As a consequence, it can be shown that the semigroup $\Q[\mo]$ converges to equilibrium in different ways, such as in variance or entropy, and is both hypercontractive as well as ultracontractive; see \Cref{sec:classical-Jacobi}, where we review essential facts about the classical Jacobi operator, semigroup and process. Classical
Jacobi processes have been used in applications such as population genetics under the name Wright--Fisher diffusion, see e.g., Ethier and Kurtz \cite[Chapter 10]{EthierKurtz}, Demni and Zani \cite{dem}, Griffiths et al.~\cite{griffiths:2010,griffiths2018}, Huillet \cite{huillet2007wright} or Pal~\cite{pal2013},
as well as in finance; see e.g., Delbaen and Shirikawa \cite{delbaen2002interest} or Gourieroux and Jasiak \cite{gourieroux2006}.

Due to the non-local part of $\calJ$ and its non-self-adjointness as a densely defined and closed operator in $\Leb^2(\bpsi)$ with $\bpsi$ denoting the invariant measure of the corresponding semigroup, a fact that is proved below, the traditional techniques that are used to study $\calJd_\mo$ seem out of reach. Nevertheless, our investigation of $\calJ$ yields generalizations of the
classical results mentioned above. An important ingredient of our approach is the notion of an intertwining relation, which is a type of commutation relationship for linear operators. For fixed $\lam$ and parameters $\widetilde{\mo}$, $\overline{\mo}$ to be specified below, we develop identities of the form
\begin{equation} \label{intJ}
\calJ \Lambda = \Lambda \calJd_{\widetilde{\mo}} \quad \mbox{and} \quad V \calJ = \calJd_{\overline{\mo}} V
\end{equation}
on the space of polynomials as well as
\begin{equation} \label{intQ}
\J_t \Lambda = \Lambda \Q[\widetilde{\mo}]_t \quad \mbox{and} \quad V\J_t = \Q[\overline{\mo}]_t V
\end{equation}
on $\Leb^2(\gab[\widetilde{\mo}])$ and $\Leb^2(\bpsi)$, respectively, where $\Lambda:\Leb^2(\gab[\widetilde{\mo}]) \to \Leb^2(\bpsi)$ and $V:\Leb^2(\bpsi) \to \Leb^2(\gab[\overline{\mo}])$ are bounded linear operators.
While \eqref{intJ} allows us to show that $\calJ$ generates an ergodic Markov semigroup $\J = (\J_t)_{t \geq 0}$
with unique invariant probability measure $\bpsi$, we use \eqref{intQ} to obtain
the spectral theory, convergence-to-equilibrium, hypercontractivity, and ultracontractivity estimates for $\J$.


The rest of the paper is organized as follows. The main results are stated in \Cref{sec:main-results}. All proofs are given in \Cref{sec:proofs}, and a specific family of non-local Jacobi semigroups is considered in \Cref{sec:examples}.
Known results on classical Jacobi operators, semigroups and processes are collected in \Cref{sec:classical-Jacobi}.



% ...



\section{Main results on non-local Jacobi operators and semigroups} \label{sec:main-results}

In this section we state our main results concerning the non-local operator $\calJ$ defined in \eqref{eq:defJ}.
Throughout the paper we make the following
\begin{assmp}
\label{assA}
The function $h \colon (1,\infty) \to [0,\infty)$ is assumed to be $0$ outside of $(1,\infty)$ and to satisfy
$\h = \int_1^\infty h(r)dr < \infty$. Moreover, $\Pi(dr) = -(e^rh(e^r))'dr$ is a finite non-negative Radon
measure on $(0,\infty)$, and if $h \not\equiv 0$,
\begin{equation*}
\lam > \bm{1}_{\{\mo < 1+\h\}} + \mo \quad \text{and} \quad \mo > \h,
\end{equation*}
while otherwise, $\lam > \mo > 0$.
\end{assmp}

\subsection{Preliminaries and existence of Markov semigroup}
%


Anticipating Theorem \ref{thm:main-inter-L2} below, we already mention that the c\`adl\`ag realization of the Markov semigroup $\J$  has downward jumps from $x$ to $e^{-r}x$, $r,x > 0$, which occur at a frequency given by the L\'evy kernel $\Pi(dr)/x$; see Lemma \ref{lem:rewriting-stuff}. Also note that for $h \not\equiv 0$, we have $\h >0$ and therefore, $\lam >1$. Next, we consider the convex twice differentiable and eventually increasing function $\Psi \colon [0,\infty) \to \R$ given by
\begin{equation}
\label{eq:defpsi}
\Psi(u) = u^2 + (\mo-\h-1)u + u \int_1^\infty (1-r^{-u})h(r)dr,
\end{equation}
which is easily seen to always have 0 as a root, and has a root $\mathsf{r} > 0$ if and only if $\mo < 1 + \h$. Set
\begin{equation}
\label{eq:def-root}
\tet = \mathsf{r}\bm{1}_{\{\mo < 1+\h\}} \quad \text{and} \quad \vat = 1-\tet,
\end{equation}
and define $\phi \colon [0,\infty) \to [0,\infty)$ by
\begin{equation}
\label{eq:def-phi}
\phi(u) = \frac{\Psi(u)}{u-\tet}.
\end{equation}
For instance, if $\tet = 0$, then
\begin{equation*}
\phi(u) = u + (\mo-1-\h) + \int_1^\infty (1-r^{-u})h(r)dr,
\end{equation*}
and we note that  $\calJ$ (resp.~$\phi$) is  uniquely determined by $\lam$, $\mo$ and $h$ (resp.~$\mo$ and $h$), so that for fixed $\lam$, there is a one-to-one correspondence between $\calJ$ and $\phi$. As we show in Lemma \ref{lem:phi-Psi1-claims} below, $\phi$ is a Bernstein function; that is, $\phi \colon [0,\infty) \to [0,\infty)$ is continuous, infinitely differentiable on $(0, \infty)$, and $(-1)^{n+1}\frac{d^n }{du^n}\phi(u)\geq 0$ for all $n = 1,2,\ldots$ and $u > 0$; see Bertoin \cite{Bertoin-s} or Schilling et al.~\cite{SchillingSongVondracek10} for a thorough exposition of Bernstein functions and subordinators.
As a Bernstein function, $\phi$ admits an analytic extension to the right half-plane $\{z \in \mathbb{C} : \Re(z) > 0\}$; see e.g., Patie and Savov~\cite[Chapter 4]{Patie-Savov-GeL}. We write $W_{\phi}$ for the unique solution, in the space of positive definite functions, to the functional equation
\begin{equation*}
W_{\phi}(z+1)=\phi(z)W_{\phi}(z), \quad \Re(z)>0,
\end{equation*}
with $W_{\phi}(1)=1$; see Patie and Savov~\cite{Patie-Savov-Bern} for a thorough account on this set of functions that generalize the gamma function, which appears as a special case if $\phi(z)=z$. In particular, for any
$n\in \N = \{0,1,2, \dots\}$,
\begin{equation}
\label{eq:product-Wphi}
W_{\phi}(n+1)=\prod_{k=1}^n \phi(k)
\end{equation}
with the convention $\prod_{k=1}^0 \phi(k) = 1$.


Let $C[0,1]$ be the Banach space of continuous functions $f \colon [0,1] \to \R$ equipped with the
sup-norm $|| \cdot ||_\infty$, and denote by $C^k[0,1]$ the subspace of $k$ times continuously
differentiable functions with $C^\infty[0,1] = \bigcap_{k=0}^\infty C^k[0,1]$. We call a one-parameter semigroup
$\J = (\J_t)_{t \geq 0}$ of linear operators on $C[0,1]$ a Markov semigroup if for all $f \in C[0,1]$ and $t \ge 0$,
$\J_t \bm{1}_{[0,1]}  = \bm{1}_{[0,1]}$, $\J_t f \geq 0$ if $f \geq 0$, $||\J_t f||_\infty \leq ||f||_\infty$
and $\lim_{t\to 0} ||\J_t f - f||_\infty = 0$. 
A probability measure $\bpsi$ on $[0,1]$ is invariant for $\J$ if for all $f \in C[0,1])$ and $t \ge 0$,
\begin{equation*}
\bpsi [ \J_t f ] = \bpsi[f] =\int_0^1 f(y)\bpsi(dy),
\end{equation*}
where the last equality serves as the definition of $\bpsi[f]$. It is then classical, see either Bakry et al.~\cite{Bakry_Book} or Da Prato \cite{da-prato:2006}, that a Markov semigroup on $C[0,1]$ with an invariant probability measure $\bpsi$
can be extended to the weighted Hilbert space
\begin{equation*}
\Leb^2(\bpsi) = \left\lbrace f \colon [0,1] \to \R \textrm{ measurable with } \bpsi[f^2] <\infty \right\rbrace.
\end{equation*}
Such a semigroup is said to be ergodic if, for every $f \in \Leb^2(\bpsi)$, $\lim_{T \to \infty} \frac{1}{T} \int_0^T \J_t f dt = \bpsi[f]$ in the $\Leb^2(\bpsi)$-norm.


Next, for complex numbers $x,a$ such that $a+x\neq -n, n\in\N$, we denote by $(a)_x$ the Pochhammer symbol
given by
\begin{equation} \label{eq:defPoc}
(a)_x = \frac{\Gamma(a+x)}{\Gamma(a)}.
\end{equation}
Writing $\mathscr{P}$ for the algebra of polynomials on $[0,1]$ and denoting $p_n(x) = x^n$, $x \in [0,1]$,
we formally define
\begin{equation}
\label{eq:mom-bpn}
\bpsi[p_n] = \frac{(\vat)_n}{(\lam)_n}\frac{W_{\phi}(n+1)}{n!} \quad \mbox{for } n \in \N,
\end{equation}
and note that in Lemma \ref{lem:phi-Psi1-claims} we show that $\vat \in (0,1]$. Recall that a sequence is said to be Stieltjes moment determinate if it is the moment sequence of a unique probability measure on $[0,\infty)$. Our first main result provides the existence of an ergodic Markov semigroup generated by the non-local Jacobi operator $\calJ$.

\begin{theorem} \label{thm:main-inter-L2} $\mbox{}$\\[-5mm]
\begin{enumerate}[\rm (i)]
\item \label{item-1:main-inter-L2}
$(\bpsi [ p_n ])_{n \geq 0}$ is a Stieltjes moment determinate sequence, and the corresponding probability measure  $\bpsi$ is an absolutely continuous
 measure with support $[0,1]$ and has a continuous density that is positive on $(0,1)$.
\item \label{item-2:main-inter-L2} The extension of $\calJ$ to an operator on $\Leb^2(\bpsi)$, still denoted by $\calJ$, is the infinitesmal generator, having $\Poly$ as a core, of an ergodic Markov semigroup $\J = (\J_t)_{t \geq 0}$ on $\Leb^2(\bpsi)$ whose unique invariant measure is $\bpsi$.
\end{enumerate}
\end{theorem}

The proof of \eqref{item-2:main-inter-L2} makes use of an intertwining relation stated in Proposition \ref{prop:inter_P}, which is an original approach to showing that the assumptions of the Hille--Yosida--Ray Theorem are fulfilled; see Lemma \ref{lem:inter} for more details. More generally, the idea of constructing a new Markov semigroup by intertwining with a known, reference Markov semigroup goes back to Dynkin \cite{dynkin:1965}, whose ideas were extended by Rogers and Pitman in \cite{rogers:1981},
leading to the characterization of Markov functions; that is, measurable maps that preserve the Markov property. More recently, Borodin and Olshanski \cite{thoma-cone} also used intertwining relations combined with a limiting argument to construct a Markov process on the Thoma cone.


We also point out that the invariant measure $\bpsi$ is a natural extension of the beta distribution, which is recovered if $\phi(u) = u$, as in this case we have $W_\phi(n+1)=n!$ in \eqref{eq:mom-bpn}. The requirement in \Cref{assA} that $\Pi(dr) = -(e^rh(e^r))'dr$ be a finite measure is necessary for the existence of an invariant probability measure for $\J$. Indeed, as we illustrate in our proof of \Cref{thm:main-inter-L2}, any candidate for such a measure must have moments given by \eqref{eq:mom-bpn}. If $\Pi(dr) = -(e^rh(e^r))'dr$ is not a finite measure, then estimates by Patie and Savov in \cite[Theorem 3.3]{Patie-Savov-Bern} imply that the analytical extension of \eqref{eq:mom-bpn} to $\{z \in \mathbb{C} : \Re(z) > \vat\}$ is not bounded along imaginary lines, a necessary condition for $\beta$ to be a probability measure.


%

\subsection{Spectral theory of the Markov semigroup and generator}

We proceed by developing the $\Leb^2(\bpsi)$-spectral theory for both, the semigroup $\J$ and the operator $\calJ$. Recalling that, for fixed $\lam$, there is a one-to-one correspondence between $\calJ$ and the Bernstein function $\phi$ given in \eqref{eq:def-phi}, we define, for $n \in \N$, the polynomial $\P_n^\phi \colon [0,1] \to \R$ as
\begin{equation}
\label{eq:Jpolpsi}
\P_n^\phi (x)= \sqrt{\C_n(\vat)} \sum_{k=0}^n \frac{(-1)^{n+k}}{(n-k)!} \frac{(\lam-1)_{n+k}}{(\lam-1)_{n\phantom{+k}}} \frac{(\vat)_n}{(\vat)_k} \frac{x^k}{W_\phi(k+1)}
\end{equation}
where $\C_n(\vat)$ is given by
\begin{equation*}
\C_n(\vat)=(2n+\lam-1)\frac{n!(\lam)_{n-1}}{(\vat)_n(\lam-\vat)_n}.
\end{equation*}
Note that for $h \equiv 0$, we get $\Psi(u) = u(u-(1-\mo))$ in \eqref{eq:defpsi}, and the polynomials $(\P_n^\phi)_{n \geq 0}$ boil down, up to a normalizing constant if $\mu>1$, to the classical Jacobi  polynomials $(\P_n^{(\mo)})_{n \geq 0}$ reviewed in \Cref{sec:classical-Jacobi}. Let us denote by $\Rc_n$ the Rodrigues operator
\begin{equation}
\label{eq:def-Rodrigues}
\Rc_n f(x) = \frac{1}{n!} \frac{d^n}{dx^n}(x^n f(x))
\end{equation}
and set
\begin{equation} \label{Delta}
\Delta = \lam-\vat - (\mo-1)\bm{1}_{\{\mo \geq 1+\h\}} - \h\bm{1}_{\{\mo < 1+\h\}}.
\end{equation}
We write $\bpsi(dx) = \bpsi(x)dx$ for the density given in \Cref{thm:main-inter-L2}.\eqref{item-1:main-inter-L2}, and define,
for every $n \in \N$, the function $\gab[\lam+n,\mu] \colon (0,1) \to [0,\infty)$ as
\begin{equation*}
\gab[\lam+n,\mu](x) = \frac{\Gamma(\lam+n)}{\Gamma(\mu)\Gamma(\lam+n-\mu)} x^{\mu-1}(1-x)^{\lam+n-\mu-1}.
\end{equation*}
In particular, the function
\begin{equation*}
\gab[\lam+n,\lam](x) = \frac{(\lam)_n}{\Gamma(n)} x^{\lam-1}(1-x)^{n-1}
\end{equation*}
will be useful for us in the sequel. Let us denote by
$\Leb^2[0,1]$ the usual Lebesgue space of square-integrable functions on $[0,1]$.


\begin{proposition}
\label{prop:co-eigenfunctions}
Set $\V_0^\phi \equiv 1$, and define, for $n = 1,2,\ldots$, $\V_n^\phi: (0,1) \to \R$ as
\begin{equation}
\label{eq:v-n-rodrigues}
\V_n^\phi(x) = \frac{1}{\bpsi(x)} \frac{(\lam-\vat)_n}{(\lam)_n}\sqrt{\C_n(\vat)} \ \Rc_n (\gab[\lam+n,\lam] \diamond \bpsi) (x) = \frac{1}{\bpsi(x)}w_n(x),
\end{equation}
where $\diamond$ is the product convolution operator defined in \eqref{eq:defd}.
Then, $w_n \in C^\infty(0,1)$. Moreover, if $\Delta > \frac{1}{2}$, then $w_n \in \Leb^2[0,1]$, and if
$\Delta \geq 2$, then $\V_n^\phi \in C^{\lceil \Delta \rceil - 2}(0,1)$.
\end{proposition}

\begin{Remark}
The definition in \eqref{eq:v-n-rodrigues} makes sense regardless of the differentiability of $\bpsi$, since $\gab[\lam+n,\lam] \in C^\infty(0,1)$ and $\Rc_n (\gab[\lam+n,\lam] \diamond \bpsi) = \Rc_n \gab[\lam+n,\lam] \diamond \bpsi$. However, the differentiability of $\V_n^\phi$ is limited by the smoothness of $\bpsi$, which is quantified by the index $\lceil \Delta \rceil - 2$. Note that for $h \equiv 0$, one has $\bpsi = \gab[\mo]$ and, by moment identification and determinacy, it is easily checked that \eqref{eq:v-n-rodrigues} boils down, up to a multiplicative constant,  to the Rodrigues representation of the classical Jacobi polynomials $\P_n^{(\mo)}$ given in \eqref{eq:Jac-Rod}. In this sense, $(\P_n^\phi)_{n \geq 0}$ and $(\V_n^\phi)_{n \geq 0}$  generalize $(\P_n^{(\mo)})_{n \geq 0}$ in different ways, related to different representations of these orthogonal polynomials.
\end{Remark}

We call two sequences $(f_n)_{n \geq 0}, (g_n)_{n \geq 0} \subseteq \Leb^2(\bpsi)$ biorthogonal if $\bpsi [ f_mg_n ] = 1$ for
$m=n$ and $\bpsi [f_mg_n] = 0$ otherwise, and then write $f_n \otimes g_n$ for the projection operator given by $f \mapsto \bpsi[f g_n] f_n$. Moreover, a sequence admitting a biorthogonal sequence will be called \emph{minimal} and a sequence that is both minimal and complete, in the sense that its linear span is dense in $\Leb^2(\bpsi)$,
will be called \emph{exact}. It is easy to show that a sequence $(f_n)_{n \geq 0}$ is minimal if and only if none of
its elements can be approximated by linear combinations of the others. If this is the case, then
a biorthogonal sequence is uniquely determined if and only if $(f_n)_{n \geq 0}$ is complete.
Next, a sequence $(f_n)_{n \geq 0} \subseteq \Leb^2(\bpsi)$ is said to be a Bessel sequence if there exists
a constant $B \ge 0$ such that, for all $f \in \Leb^2(\bpsi)$,
\begin{equation*}
\sum_{n=0}^\infty \bpsi [f_n f]^2 \leq B \: \bpsi[f^2].
\end{equation*}
The quantity $B$ is a Bessel bound of $(f_n)_{n \geq 0}$, and the smallest such $B$ is called the optimal Bessel bound of $(f_n)_{n \geq 0}$; see e.g., Christensen~\cite{Christensen2003a} for further information on these objects that play a central role in non-harmonic analysis.

We write $\sigma(\J_t)$ for the spectrum of the operator $\J_t$ in $\Leb^2(\bpsi)$ and $\sigma_p(\J_t)$ for its point spectrum, and similarly define $\sigma(\calJ)$ and $\sigma_p(\calJ)$. For an isolated eigenvalue $\varrho \in \sigma_p(\J_t)$ we write $\mathrm{M}_a(\varrho,\J_t)$ and $\mathrm{M}_g(\varrho,\J_t)$ for the algebraic and geometric multiplicity of $\varrho$, respectively. We also define, for $n \in \N$,
\begin{equation}
\label{eq:def-lambda-n}
\lambda_{n} = n(n-1)+\lam n = n^2+(\lam-1)n,
\end{equation}
noting that $\lambda_{1} = \lam$, which explains our choice of notation, and recall that $\sigma(\calJd_\mo) = \sigma_p(\calJd_\mo) = \{-\lambda_n; \ n \in \N\}$; see~\Cref{sec:classical-Jacobi}. Writing $\J_t^*$ for the $\Leb^2(\bpsi)$-adjoint of $\J_t$, we have the following spectral theorem for $\J$.




\begin{theorem}
\label{thm:spectral-representation} Let $t > 0$.
\begin{enumerate}[\rm(i)]
\item \label{item-1:thm:spectral-representation} Then,
\begin{equation*}
\J_t = \sum_{n=0}^{\infty} e^{-\lambda_n t}  \P_n^\phi \otimes \V_n^\phi, 
\end{equation*}
where the series converges in operator norm and $(\P_n^\phi)_{n \geq 0} \subseteq \Leb^2(\bpsi)$ is an exact Bessel sequence with optimal Bessel bound 1 and unique biorthogonal sequence $(\V_n^\phi)_{n \geq 0} \subseteq \Leb^2(\bpsi)$, which is also exact. Moreover, for all $n \in \N$, $\P_n^\phi$ ($\V_n^\phi$) is an eigenfunction of $\J_t$ ($\J_t^*$) with eigenvalue $e^{-\lambda_n t}$.
\item \label{item-2:thm:spectral-representation} {The operator $\J_t$ is compact; that is, the semigroup $\J$ is immediately compact.}
\item \label{item-3:thm:spectral-representation} The following spectral mapping theorem holds:
\begin{equation*}
\sigma(\J_t) \setminus \{0\} = \sigma_p(\J_t) = e^{t\sigma_p(\calJ)} = e^{t\sigma(\calJ)} = \left\lbrace e^{-\lambda_n t} : n \in \N \right\rbrace.
\end{equation*}
Furthermore, $\sigma(\J_t) = \sigma(\J_t^*)$ and, for any $n \in \N$,
\begin{equation*}
\mathrm{M}_a(e^{-\lambda_n t},\J_t) = \mathrm{M}_g(e^{-\lambda_n t},\J_t) = \mathrm{M}_a(e^{-\lambda_n t},\J_t^*) = \mathrm{M}_g(e^{-\lambda_n t},\J_t^*) = 1.
\end{equation*}
\item \label{item-4:thm:spectral-representation} The operator $\J_t$ is self-adjoint in $\Leb^2(\bpsi)$ if and only if $h \equiv 0$.
\end{enumerate}
\end{theorem}

The expansion in Theorem \ref{thm:spectral-representation}.\eqref{item-1:thm:spectral-representation} is not valid for $t =0$ since $(\P_n^\phi)_{n \geq 0}$ is a Bessel sequence but not a Riesz sequence, as it is not the image of an orthogonal sequence under a bounded linear operator having a bounded inverse; see Proposition \ref{prop:quasi-similar} below. The sequence of non-self-adjoint projections $\P_n^\phi \otimes \V_n^\phi$ is not uniformly bounded in $n$, see Remark \ref{rem:bounds}, and, in contrast to the self-adjoint case, the eigenfunctions of $\J_t$ and $\J_t^*$ do not form a Riesz basis of $\Leb^2(\bpsi)$. Finally, it follows from \Cref{thm:spectral-representation}.\eqref{item-4:thm:spectral-representation} that $\P_n^\phi \neq \V_n^\phi$ for all $n = 1,2,\ldots$.

\subsection{Convergence-to-equilibrium and contractivity properties}


We call a function $\Phi:I \to \R$, defined on an interval $I \subseteq \R$, admissible if
\begin{equation}
\label{eq:def-admissible}
\text{$\Phi \in C^4(I)$ with both $\Phi$ and $-1/\Phi''$ convex.}
\end{equation}
Given an admissible function $\Phi$ we write, for any $f:[0,1] \to I$ with $f, \Phi(f) \in \Leb^1(\bpsi)$,
\begin{equation}
\label{eq:def-Phi-entropy}
\Ent_\bpsi^\Phi(f) = \bpsi[\Phi(f)] - \Phi(\bpsi [f])
\end{equation}
for the so-called $\Phi$-entropy of $f$. An important special case is $\Phi(r) = r^2$ with $I = \R$, so that \eqref{eq:def-Phi-entropy} gives the variance $\Var_\bpsi(f)$ of $f \in \Leb^2(\bpsi)$. Recall that in the classical case $h \equiv 0$, we have the following equivalence between the Poincar\'e inequality for $\calJd_\mo$ and the spectral gap inequality for $\Q$,
\begin{equation*}
\lam = \inf_f \frac{-\gab[\mo][f \calJd_\mo f]}{\Var_{\gab[\mo]}(f)} > 0 \iff \Var_{\gab[\mo]}(\Q_t f ) \leq e^{-2\lam t} \Var_{\gab[\mo]}(f) \text{ for } f \in \Leb^2(\gab[\mo]) \text{ and } t \geq 0,
\end{equation*}
where the infimum is over all functions $f$ in the $\Leb^2$-domain of $\calJd_\mo$ such that
$\Var_{\gab[\mo]}(f) > 0$; see e.g., Bakry et al.~\cite[Chapter 4.2]{Bakry_Book}. The above variance decay is optimal in the sense that the decay rate does not hold for any constant strictly greater than $2\lam$. Another important instance of \eqref{eq:def-Phi-entropy} corresponds to $\Phi(r)=r\log r$ and $I = [0, \infty)$. It recovers the classical notion of entropy for a non-negative function, written simply as $\Ent_\bpsi(f)$. Here the classical equivalence is between the log-Sobolev inequality and entropy decay,
\begin{equation*}
\logSob[\mo] = \inf_f \frac{-4\gab[\mo] [f \calJd_\mo f]}{\Ent_{\gab[\mo]} (f^2)} > 0 \iff \Ent_{\gab[\mo]} (\Q[\mo]_tf) \leq e^{-\logSob[\mo] t} \Ent_{\gab[\mo]} (f) < \infty \text{ for } f \in \Leb^1_+(\bpsi) \text{ and } t \geq 0,
\end{equation*}
where the infimum is over all functions $f$ in the $\Leb^2$-domain of $\calJd_\mo$ such that $\Ent_{\gab[\mo]} (f^2) > 0$.
Note that the optimal entropy decay rate is obtained only for $\mo = \lam/2 > 1$, in which case, $\logSob[\mo] = 2\lam$, while otherwise $\logSob[\mo] < 2\lam$; see, e.g., Fontenas~\cite{fontenas:1998}. We review these notions for the classical Jacobi semigroup in \Cref{sec:classical-Jacobi}. For more details, we refer to Chafa\"i \cite{chafai}, An{\'e} et~al.~\cite{logSobolev-book} and the relevant sections of Bakry et al.~\cite{Bakry_Book}. However, due to the non-self-adjointness and non-local properties of $\calJ$, it seems challenging to develop an approach based on the Poincar\'e or log-Sobolev inequalities. For this reason, we take an alternative route to tackling convergence to equilibrium by using the concept of  interweaving relations recently introduced by Patie and Miclo in \cite[Section 3.5]{Gateway} and \cite{Patie-Miclo}.

Now, consider the function $\rho:[0,\infty) \to [0,\infty)$ given by
\begin{equation*}
\rho(u) = \sqrt{u+\frac{(\lam-1)^2}{4}}-\frac{\left|\lam-1\right|}{2}
\end{equation*}
and recall that for $h \not\equiv 0$, we have $\lam > 1$.  Note that $\rho$ is a Bernstein function, as it is obtained by translating and centering the well-known Bernstein function $u \mapsto \sqrt{u}$. In the literature $\rho$ is known as the Laplace exponent of the so-called relativistic $1/2$-stable subordinator; see Bakry \cite{Bakry-relativistic} or Bogdan et al.~\cite{Potential2009}. Recalling that any Bernstein function $\phi$ is analytic on the right half-plane, and writing $a_\phi=\sup\{u \ge 0 : \phi \textrm{ is analytic on } \{z \in \mathbb C : \Re(z)>-u\}\} \in [0,\infty]$, we denote
\begin{equation}
\label{eq:dphi}
d_{\phi} = \inf\{u \in [0,a_\phi] : \phi(-u)=0 \} \mbox{ with the convention }\inf \emptyset = a_{\phi}.
\end{equation}
Then $d_{\phi} \le a_{\phi}$, and if $a_{\phi} = \infty$, one has $\lim_{u \to - \infty} \phi(u) = - \infty$.
So, in any case, $d_{\phi} \in [0, \infty)$. For $\epsilon \in (0,d_\phi) \cup \{d_\phi\}$, we write
\begin{equation}
\label{eq:def-dtheta}
\dpt = \vat \bm{1}_{\{ \mo < 1+\h \}}+ \dph\bm{1}_{\{ \mo \geq 1+\h \}}
\end{equation}
noting that if $d_\phi=0$, then $\epsilon = 0$. Next, for any $\m \in (\bm{1}_{\{\mo < 1+\h\}} + \mo , \lam)$ and $\epsilon \in (0,d_\phi) \cup \{d_\phi\}$, we denote by $\tau$ a non-negative random variable,  whose existence is provided in the theorem below, with Laplace transform
\begin{equation}
\label{eq:def-T-Laplace}
\E\left[e^{-u \tau}\right] = \frac{(\dpt)_{\rho(u)}}{(\m)_{\rho(u)}} \frac{(\lam-\m)_{\rho(u)}}{(\lam-\dpt)_{\rho(u)}} , \quad u \geq 0,
\end{equation}
and write $\J_{t+\tau} = \int_0^\infty \J_{t+s} \Prob(\tau \in ds)$.

\begin{theorem}
\label{thm:convergence-equilibrium} Let $t \geq 0$. Then, for all $\m \in (\bm{1}_{\{\mo < 1+\h\}} + \mo ,\lam)$
and $\epsilon \in (0,d_\phi) \cup \{d_\phi\}$, the following hold:
\begin{enumerate}[\rm(i)]
\item \label{item-1:thm:convergence-equilibrium} For any $f \in \Leb^2(\bpsi)$,
\begin{equation*}
\Var_\bpsi(\J_t f) \leq \frac{\m(\lam-\dpt)}{\dpt(\lam-\m)} e^{-2\lam t} \Var_\bpsi(f)
\end{equation*}
and $\m(\lam-\dpt) > \dpt(\lam-\m)$.
\item \label{item-2:thm:convergence-equilibrium}
The function $\phi^{(\tau)} \colon u \mapsto - \log \E[e^{-u \tau}]$, defined in \eqref{eq:def-T-Laplace}, is a Bernstein function. Therefore,
$\tau$ is infinitely divisible and there exists a subordinator $\bm{\tau} = (\tau_t)_{t \geq 0}$ with $\tau_1 \stackrel{(d)}{=} \tau$. For any $f \in \Leb^1_+(\bpsi)$ with $\Ent_\bpsi(f) < \infty$, one has
\begin{equation*}
\Ent_\bpsi(\J_{t+\tau} f) \leq e^{- \logSob[\m] t} \Ent_\bpsi(f),
\end{equation*}
and if $\lam > 2(\bm{1}_{\{\mo < 1+\h \}} + \mo)$, then
\begin{equation*}
\Ent_\bpsi(\J_{t+\tau} f) \leq e^{- 2 \lam t} \Ent_\bpsi(f).
\end{equation*}
Furthermore, if $\bm{1}_{\{\mo < 1+\h \}} + \mo < \lam/2 \in \N$ and $\Phi \colon I \to \R$, $I \subseteq \R$, is an admissible function as in~\eqref{eq:def-admissible}, then, for any $f \colon [0,1]\to I$ such that $f \in \Leb^1(\bpsi)$ and $\Ent_\bpsi^\Phi(f) < \infty$,
\begin{equation*}
\Ent_\bpsi^\Phi(\J_{t+\tau} f) \leq e^{-(\lam-1)t} \Ent_\bpsi^\Phi(f).
\end{equation*}
\end{enumerate}
\end{theorem}

\begin{Remark}
Since $\frac{\m(\lam-\dpt)}{\dpt(\lam-\m)} > 1$, the estimate in \Cref{thm:convergence-equilibrium}.\eqref{item-1:thm:convergence-equilibrium} gives the hypocoercivity, in the sense of Villani \cite{Villani-09}, for non-local Jacobi semigroups. This notion continues to attract research interest, especially in the area of kinetic Fokker--Planck equations; see, e.g. Baudoin \cite{Baudoin}, Dolbeault et al.~\cite{Mouhot} or Mischler and Mouhout~\cite{Mischler2016}. We are able to identify the hypocoercive constants, namely the exponential decay rate as twice the spectral gap and the constant in front of the exponential, which is a measure of the deviation of the spectral projections from forming an orthogonal basis and is 1 in the case of an orthogonal basis. Note that in general, the hypocoercive constants may be difficult to identify and may have little to do with the spectrum. Similar results have been obtained by Patie and Savov \cite{Patie-Savov-GeL} as well as Achleitner et al.~\cite{BGK}. Our hypocoercive estimate is obtained via intertwining, which suggests that hypocoercivity may be studied purely from this viewpoint, an idea that is further investigated in the recent work \cite{PV-Hypo} by the second and fourth author.
\end{Remark}

\begin{Remark}
The second part of \Cref{thm:convergence-equilibrium} gives the exponential decay of $\J$ in entropy, but after an independent random warm-up time. Note that, for $\lam \leq 2(\bm{1}_{\{\mo < 1+\h \}} + \mo)$ the entropy decay rate is the same as for $\Q[\m]$ while under the mild assumption $\lam > 2(\bm{1}_{\{\mo < 1+\h \}} + \mo)$, we get the optimal rate $2 \lam$
irrespective of the precise value of $\mo$. The proof relies on developing so-called \emph{interweaving relations}, a concept which has been introduced and studied in the recent work \cite{Patie-Miclo} by Miclo and Patie, where the classical Jacobi semigroup $\Q[\m]$ serves as a reference object; see Proposition \ref{prop:cmir} below.
\end{Remark}

\begin{Remark}
The additional condition $\lam/2 \in \N$ for the $\Phi$-entropic convergence in \Cref{thm:convergence-equilibrium}.\eqref{item-2:thm:convergence-equilibrium} ensures that we can invoke the known result \eqref{eq:C-D-entropy} for the classical Jacobi semigroup $\Q[\lam/2]$. However, our approach allows us to immediately transfer any improvement of \eqref{eq:C-D-entropy} to the non-local Jacobi semigroup $\J$.
\end{Remark}

Next, we recall the famous equivalence between entropy decay and hypercontractivity due to Gross \cite{Gross}. For any $t \geq 0$ and $f \in \Leb^1_+(\gab[\m])$ such that $\Ent_{\gab[\m]}(f) < \infty$, one has
\begin{equation*}
\Ent_{\gab[\mo]} (\Q[\m]_tf) \leq e^{- \logSob[\m] t} \Ent_{\gab[\m]} (f) \iff ||\Q[\m]_t ||_{2 \rightarrow q} \leq 1,  \textrm{ where }  2\leq q \leq 1+e^{ \logSob[\m] t}
\end{equation*}
where we use the shorthand $||\cdot||_{p \rightarrow q} = ||\cdot||_{\Leb^p(\gab[\m]) \rightarrow \Leb^q(\gab[\m])}$ for $1 \leq p, q \leq \infty$. To state our next result we write, if $\lam - \m > 1$, $c_{\m} > 0$ for the Sobolev constant of $\calJd_\m$ of order $\frac{2(\lam-\m)}{(\lam-\m-1)}$, and recall that as a result of the Sobolev inequality for $\calJd_\m$, one gets $||\Q[\m]_t||_{1 \to \infty} \leq c_{\m} t^{-\frac{\lam-\m}{\lam-\m-1}}$ for $0 < t \leq 1$, which implies that $\Q[\m]$ is ultracontractive,
that is, $||\Q[\m]_t||_{1 \to \infty} < \infty$ for all $t > 0$; see~\Cref{sec:classical-Jacobi} for more details. We have the following concerning the contractivity of $\J$.

\begin{theorem}
\label{thm:contractivity}
For any $\m \in (\bm{1}_{\{\mo < 1+\h\}} + \mo , \lam)$ and $\epsilon \in (0,d_\phi) \cup \{d_\phi\}$, the following hold:
\begin{enumerate}[\rm(i)]
\item \label{item-1:thm:contractivity} For $t \geq 0$, we have the hypercontractivity estimate
\begin{equation*}
||\J_{t+\tau}||_{2 \rightarrow q} \leq 1 \: \text{ for all } q \text{ satisfying } 2\leq q \leq 1+e^{ \logSob[\m] t},
\end{equation*}
and furthermore, if $\lam > 2(\bm{1}_{\{\mo < 1+\h \}} + \mo)$, then
\begin{equation*}
||\J_{t+\tau}||_{2 \rightarrow q} \leq 1 \: \text{ for all } q \text{ satisfying }2\leq q \leq 1+e^{2\lam t}.
\end{equation*}
\item \label{item-2:thm:contractivity} If, in addition, $\lam-\m > 1$, then for $0 < t \leq 1$, we have the ultracontractivity estimate
\begin{equation*}
||\J_{t+\tau}||_{1 \to \infty} \leq c_{\m} t^{-\frac{\lam-\m}{\lam-\m-1}},
\end{equation*}
where for $\lam> 2$,  one can choose $\m=\frac{\lam}{2}$, yielding $c_{\frac{\lam}{2}}=\frac{4}{\lam(\lam- 2)}$.
\end{enumerate}
\end{theorem}

\subsection{Bochner subordination of the semigroup}

We write $\J^{\bm{\tau}} = (\J_t^{\bm{\tau}})_{t \geq 0}$ for the semigroup subordinated, in the sense of Bochner, with respect to the subordinator ${\bm{\tau}} = (\tau_t)_{t \geq 0}$ whose existence is guaranteed by \Cref{thm:convergence-equilibrium}.\eqref{item-2:thm:convergence-equilibrium}, that is,
\begin{equation*}
\J_t^{\bm{\tau}} = \int_0^\infty \J_s \hspace{1pt} \Prob(\tau_t \in ds),
\end{equation*}
so that $\J_1^{\bm{\tau}} = \J_\tau$. Note that $ \J^{\bm{\tau}}$ is also an ergodic Markov semigroup on $\Leb^2(\bpsi)$ with $\bpsi$ as an invariant measure, and its generator is given by $-\phi^{(\tau)}(-\calJ) = \log \J_\tau$; see Sato~\cite[Chapter 6]{Sato1999}.
For the subordinated semigroup we have the following.



\begin{theorem}
\label{cor:subordinated}
For any $\m \in (\bm{1}_{\{\mo < 1+\h\}} + \mo , \lam)$ and $\epsilon \in (0,d_\phi) \cup \{d_\phi\}$ the statement of \Cref{thm:spectral-representation} holds for $\J^{\bm{\tau}}$ and $t \ge 1$
if $(\lambda_n)_{n \geq 0}$ is replaced with
\[
\bigg(\log \frac{(\m)_n(\lam-\dpt)_n}{(\dpt)_n(\lam-\m)_n}\bigg)_{n \geq 0},
\]
and the assertions  of \Cref{thm:convergence-equilibrium}.\eqref{item-2:thm:convergence-equilibrium} and \Cref{thm:contractivity}.\eqref{item-1:thm:contractivity} hold for $\J^{\bm{\tau}}$ if $\lam$ is replaced with $\log \frac{\m(\lam-\dpt)}{\dpt(\lam-\m)}$ and $\tau$ with 1.
 Moreover, for all $\m$ and $\epsilon$ such that $1 < \lam - \m < (\m - \dpt)(\lam-\m-1)$, $\J_t^{\bm{\tau}} f (x) = \int_0^1 f(y) q_t^{(\bm{\tau})}(x,y) \bpsi(dy)$ for any $f\in \Leb^2(\bpsi)$ and  $t > 2$, where the heat kernel $q_t^{(\bm{\tau})}$ satisfies the estimate
\begin{equation*}
|q_t^{(\bm{\tau})}(x,y) - 1| \leq c_\m(\E[\tau^{-\frac{\lam-\m}{\lam-\m-1}}]+1) \left(\frac{\m(\lam-\dpt)}{\dpt(\lam-\m)}\right)^{\frac{1-2t}{2}} < \infty
\end{equation*}
for Lebesgue-almost all $(x,y) \in [0,1]^2$. As above, if $\lam> 2$,  one can choose $\m=\frac{\lam}{2}$, yielding $c_{\frac{\lam}{2}}=\frac{4}{\lam(\lam- 2)}$.
\end{theorem}

We point out that the Markov process realization of $\J$ (resp.~$\J^{\bm{\tau}}$) has "only negative (resp. non-symmetric two-sided) jumps and can easily be shown to be a polynomial process on $[0,1]$ in the sense of Cuchiero et al.~\cite{cuchiero2018}. Markov semigroups obtained by subordinating $\J$ with respect to any conservative subordinator $\widetilde{\bm{\tau}} = (\widetilde{\tau}_t)_{t \geq 0}$ with Laplace exponent $\phi^{(\widetilde{\tau})}$ (growing fast enough at infinity, e.g.~logarithmically) are also in this class, and we obtain the spectral expansion of the subordinated semigroup
from \Cref{thm:spectral-representation} by replacing $(\lambda_n)_{n \geq 0}$ with $(\phi^{(\widetilde{\tau})}(\lambda_n))_{n \geq 0}$. Note that in the aforementioned paper the authors investigate the martingale problem for general polynomial operators on the unit simplex, of which $\calJ$ and $-\phi^{(\tau)}(-\calJ)$ are specific instances. In particular, $\calJ$ is a L\'evy type operator with affine jumps of type 2 in the sense of \cite{cuchiero2018}. For such operators, existence and uniqueness for the martingale problem have been shown in \cite{cuchiero2018} under the weaker condition $\lam \geq \mo$. However, \Cref{assA} allows us to obtain the existence and uniqueness of an invariant probability measure. 

\section{Proofs} \label{sec:proofs}

\subsection{Preliminaries}
We start by proving some preliminary results that will be useful throughout the paper. We first give an alternative form of the operator $\calJ$, which will make some later proofs more transparent.
Recall that $\Pi$, given by $\Pi(dr) = -(e^rh(e^r))'dr$, is a finite non-negative Radon measure on $(0,\infty)$.

\begin{lemma}
\label{lem:rewriting-stuff}
One has $\int_0^\infty r\Pi(dr) = \h < \infty$, and the operator $\calJ$ defined in \eqref{eq:defJ} may,
for suitable functions $f \colon [0,1] \to \R$, be written as
\begin{equation}
\label{eq:alt-J}
\calJ f(x) = x(1-x)f''(x) - (\lam x-\mo+\h)f'(x) + \int_{0}^{\infty} \left(f(e^{-r}x)-f(x)+xrf'(x)\right)\frac{\Pi(dr)}{x}.
\end{equation}
\end{lemma}

\begin{proof}
Since $(e^rh(e^r))' \le 0$ and
\begin{equation*}
\int_0^\infty e^rh(e^r)dr = \int_1^\infty h(r)dr = \h < \infty,
\end{equation*}
one obtains that $\lim_{r \to \infty} e^r h(e^r) = 0$. Consequently, one has for all $y>0$,
\begin{equation*}
\overline{\Pi}(y) = \int_y^\infty \Pi(dr) = -\int_y^\infty (e^rh(e^r))' dr = e^yh(e^y) - \lim_{r \to \infty} e^r h(e^r) = e^yh(e^y).
\end{equation*}
Thus, by Tonelli's theorem and a change of variables,
\begin{equation*}
\int_0^\infty r\Pi(dr) = \int_0^\infty \overline{\Pi}(r)dr = \int_0^{\infty} e^r h(e^r) dr = \int_1^\infty h(r)dr = \h < \infty,
\end{equation*}
which yields
\begin{equation*}
\int_{0}^{\infty} \left(f(e^{-r}x)-f(x)+xrf'(x)\right)\frac{\Pi(dr)}{x} = \h f'(x) + \int_0^\infty \frac{f(e^{-r}x)-f(x)}{x} \Pi(dr).
\end{equation*}
Integration by parts and a change of variables give
\begin{align*}
\int_0^\infty \frac{f(e^{-r}x)-f(x)}{x} \Pi(dr) = -\int_0^\infty e^{-r} f'(e^{-r}x) \overline{\Pi}(r)dr = -\int_0^\infty f'(e^{-r}x)h(e^r)dr = -h \diamond f'(x),
\end{align*}
and the lemma follows.
\end{proof}

In the sequel we keep the notation {$\Pi(dr) = -(e^rh(e^r))'dr$}, $r > 0$ and $\overline{\Pi}(y) = e^y h(e^y)$, $y > 0$. Let $\varphi: [0,\infty) \to [0,\infty)$ be the function given by
\begin{equation}
\label{def:varphi}
\varphi(u) = \frac{u+\vat}{u+1}\phi(u+1).
\end{equation}
The following result collects some useful properties of the functions $\phi$ and $\varphi$
given in \eqref{eq:def-phi} and \eqref{def:varphi}, respectively.


\begin{lemma}
\label{lem:phi-Psi1-claims} $\mbox{}$\\[-5mm]
\begin{enumerate}[\rm(i)]
\item \label{item-0:lem:phi-Psi1-claims} $\phi$ is a Bernstein function satisfying $\lim_{u\to\infty} \phi(u)/u =1$.
\item \label{item-1:lem:phi-Psi1-claims}
$\vat$, given in \eqref{eq:def-root}, satisfies $\vat \in (0,1]$ with $\vat = 1$ if and only if $\mo \geq 1+\h$. Additionally, if $\mo \geq 1+\h$, then $\phi(0) = \mo-\h - 1$, while if $\mo < 1+\h$, then $\phi(0) = 0$.
\item \label{item-2:lem:phi-Psi1-claims}
Suppose $\mo < 1+\h$. Then $\varphi$ is a Bernstein function that is in correspondence with the non-local Jacobi operator $\calJ_\varphi$ with parameters $\lam$, $\mo_\varphi = 1+\mo$ and the non-negative function $h_\varphi(r)=r^{-1}\overline{\Pi}_\varphi(\log r), r>1$, where $\Pi_\varphi$ is the finite non-negative Radon measure given by
\begin{equation*}
\Pi_\varphi(dr) = e^{-r}\left(\Pi(dr)+\overline{\Pi}(r)dr\right), \ r > 0.
\end{equation*}
Furthermore, $\h_\varphi = \int_1^\infty h_\varphi(r)dr < \infty$ with $\mo_\varphi \geq 1+\h_\varphi$ and $\lam > \mo_\varphi$.
\end{enumerate}
\end{lemma}

\begin{proof}
First we rewrite \eqref{eq:defpsi} using integration by parts to get, for any $u\geq0$,
\begin{equation}
\label{eq:ll}
\Psi(u) = u^2+(\mo-\h-1)u+u\int_1^\infty (1-r^{-u})h(r)dr = u^2+(\mo-\h-1)u + \int_0^\infty (e^{-ur}-1+ur)\Pi(dr).
\end{equation}
Since, by Lemma \ref{lem:rewriting-stuff}, we have $\int_0^\infty r\Pi(dr) < \infty$, we recognize $\Psi$ as the Laplace exponent of a spectrally negative L\'evy process with a finite mean given by $\Psi'(0^+)=\mu-\h-1$. In particular, $\Psi$ is a convex, eventually increasing, twice differentiable function on $[0,\infty)$ that is zero at $0$. Therefore, it has a strictly positive root $\tet$ if and only if $\mo < 1+\h$. By the Wiener--Hopf factorization of L\'evy processes, see e.g.,~\cite[Chapter 6.4]{Kyprianou2014}, we get for $\Psi'(0^+) \geq 0$ (resp.~$\Psi'(0^+) < 0$) that $\Psi(u) = u\phi(u)$ (resp.~$\Psi(u) = (u-\tet)\phi(u)$) for a Bernstein function $\phi$. That $\lim_{u \to \infty} \Phi(u)/u = 1$ then follows from the well-known result that $\lim_{u \to \infty} u^{-2}\Psi(u) = 1$, which can be obtained by dominated convergence since $\Pi$ is a finite measure. This completes the proof of \eqref{item-0:lem:phi-Psi1-claims}.

Next, we show $\Psi(1) > 0$, which, by the convexity of $\Psi$ is equivalent to $\tet \in [0,1)$. Indeed, from \eqref{eq:ll} and an application of Tonelli's theorem, we get
\begin{equation*}
\Psi(1) = \mo-\h + \int_0^\infty (1-e^{-r})\overline{\Pi}(r)dr > 0
\end{equation*}
where we used the assumption $\mo > \h$. If $\mo \geq 1 + \h$, then $\tet = 0$, and we obtain from \eqref{eq:ll},
\begin{equation*}
\phi(u) = u+(\mo-\h-1)+ \int_0^\infty (e^{-ur}-1+ur)\Pi(dr),
\end{equation*}
which gives $\phi(0) = \mo - \h - 1$. On the other hand, if $\tet > 0$, then the fact that $\Psi(0)=-\tet\phi(0) = 0$ forces $\phi(0)=0$, which completes the proof of \eqref{item-1:lem:phi-Psi1-claims}.

To show \eqref{item-2:lem:phi-Psi1-claims}, we write $\Psi_1(u)=\frac{u}{u+1}\Psi(u+1)$. According to \cite[Proposition 2.2]{Chazal2014}, $\Psi_1$ is the Laplace exponent of a spectrally negative L\'evy process with Gaussian coefficient 1, mean $\mo_\varphi$ and L\'evy measure $\Pi_\varphi$. Observe that $\Psi_1'(0^+) = \Psi(1) > 0$ and
\begin{equation*}
\Psi_1(u) = \frac{u}{u+1}(u+1-\tet)\phi(u+1) = u \frac{u+\vat}{u+1}\phi(u+1) = u\varphi(u).
\end{equation*}
So the Wiener--Hopf factorization of $\Psi_1$ shows that $\varphi$ is a Bernstein function. Moreover, integration by parts of $\Pi_\varphi$ gives
\begin{equation*}
\h_\varphi = \int_0^\infty \overline{\Pi}_\varphi (r)dr = \int_0^\infty e^{-r}\overline{\Pi}(r)dr \leq \h < \infty,
\end{equation*}
where the boundary terms are easily seen to evaluate to 0. Finally, using the assumption $\mo > \h$, we get that $\mo_\varphi = 1+\mo-\h_\varphi+\h_\varphi \geq 1+\mo-\h+\h_\varphi > 1+\h_\varphi$, while the condition $\lam > \mo_\varphi$ follows from the assumption that $\lam > \bm{1}_{\{\mo < 1+\h\}} +\mo = 1+\mo = \mo_\varphi$.
\end{proof}

\subsection{Proof of \Cref{thm:main-inter-L2}.\eqref{item-1:main-inter-L2}}


Before we begin we provide an analytical result, which will allow us to show that the support of $\bpsi$ is $[0,1]$ and will also be used in subsequent proofs. We say that a linear operator $\Lambda$ is a Markov multiplicative kernel if $\Lambda f(x) = \E[f(xM)]$ for some random variable $M$. With $d_\phi$ as in \eqref{eq:dphi}, we denote, for any $\epsilon \in (0,d_\phi) \cup \{d_\phi\}$,
\begin{equation}
\label{eq:defn-d0}
\dz = \bm{1}_{\{ \mo < 1+\h \}} + \dph\bm{1}_{\{ \mo \geq 1+\h \}},
\end{equation}
recalling that for $d_\phi =0$, we have $\epsilon = 0$, so that at least $\dz \geq 1$. Note that $\dz=\dpt$ if $\vat=1$, explaining  the notation. By \cite[Lemma 10.3]{Patie-Savov-GeL}, the mapping
\begin{equation}
\label{eq:defphi-d0}
\phi_{\dz}(u)=\frac{u}{u+\dz-1}\phi(u), \quad u \ge 0,
\end{equation}
is a Bernstein function, and, by Proposition 4.4(1) in the same paper, we also have that, for
any $\m \in (\bm{1}_{\{\mo < 1+\h\}} + \mo, \lam)$,
\begin{equation}
\label{eq:def-phi-m}
\phm(u) = \frac{\phi(u)}{u+\m-1}, \quad u \ge 0,
\end{equation}
is a Bernstein function. We define the following linear operators acting on the space of polynomials $\Poly$, recalling
that for $n \in \N$, $p_n(x) = x^n$, $x \in [0,1]$.
\begin{eqnarray}\label{eq:mom_Ip-0}\label{eq:mom_Vp}\label{eq:mom_Up}
\qquad \Lambda_{\phi_{\dz}}p_n = \frac{(\dz)_n}{W_\phi(n+1)}p_n,  \:
 \mathrm{V}_{\phm} p_n = \frac{W_{\phi}(n+1)}{(\m)_n}p_n \: \textrm{ and } \:
 \mathrm{U}_\varphi p_n = \frac{\varphi(0)}{\varphi(n)} p_n
\end{eqnarray}
where $\mathrm{V}_{\phm}$ is defined for any $\m \in (\bm{1}_{\{\mo < 1+\h\}} + \mo, \lam)$ and $\varphi$ was defined in \eqref{def:varphi}. We write $\Bo{C[0,1]}$ for the unital Banach algebra of bounded linear operators on $C[0,1]$ and say that a linear operator between two Banach spaces is a quasi-affinity if it has trivial kernel and dense range.

\begin{lemma}
\label{lem:Markov-kernel}
The operators $\Lambda_{\phi_{\dz}}$, $\rm{V}_{\phm}$ and $\rm{U}_{\varphi}$ defined in \eqref{eq:mom_Ip-0} are Markov multiplicative kernels associated to random variables $X_{\phi_{\dz}}$, $X_{\phm}$ and $X_{\varphi}$, respectively, valued in $[0,1]$, and hence moment determinate. All 3 random variables have continuous densities, and all 3 operators belong to $\Bo{C[0,1]}$. Furthermore, $\Lambda_{\phi_{\dz}}$ is a quasi-affinity on $C[0,1]$ while $\rm{V}_{\phm}$ and $\rm{U}_{\varphi}$ have dense range in $C[0,1]$.
\end{lemma}

\begin{proof}
The claims regarding the operators $\Lambda_{\phi_{\dz}}$ and $\rm{V}_{\phm}$ and corresponding random variables have been proved in \cite{Patie-Savov-GeL}; see Theorem 5.2, Proposition 6.7(1) and Section 7.1 therein. Let $W \colon [0,\infty) \to [0,\infty)$ be the function characterized by its Laplace transform via
\begin{equation*}
\int_{0}^{\infty}e^{-u x}W(x)dx = \frac{1}{\Psi(u)}, \quad u > \tet,
\end{equation*}
and note that $W$ is increasing. Moreover, since $\Psi$ has a Gaussian component, $W$ is at least continuously differentiable; see e.g.,~\cite[Section 8.2]{Kyprianou2014}. The law of the random variable $X_\varphi$ is given by
\begin{equation*}
\Prob(X_\varphi \in dx) = \varphi(0)W'(-\log x)dx, \quad x \in [0,1].
\end{equation*}
So it clearly is supported on $[0,1]$ and has a continuous density. The claims concerning $\rm{U}_{\varphi}$ were shown in \cite[Lemma 4.2]{Intertwining-Krein}, where we note that $W(0)=0$ since $\Psi$ has a Gaussian component.
\end{proof}

Now, suppose $\mo \geq 1+\h$, so that, by Lemma \ref{lem:phi-Psi1-claims}, $\vat= 1$. Then, for all $n\in\N$, \eqref{eq:mom-bpn} reduces to
\begin{equation*}
\bpsi[p_n] = \frac{W_{\phi}(n+1)}{(\lam)_n}.
\end{equation*}
Since $\lam > \mo \geq 1$, we get that $\phi_{\lam}^*$ given in \eqref{eq:def-phi-m} is a Bernstein function. Indeed, if $\mo = 1$, we must have $\h = 0$, and the function $u \mapsto \frac{u}{u+\lam-1}$ is Bernstein since $\lam > 1$, see e.g.~\cite[Chapter 16]{SchillingSongVondracek10}, while if $\mo > 1$, Proposition 4.4(1) of \cite{SchillingSongVondracek10}
guarantees that $\phi_{\lam}^*$ is a Bernstein function. One straightforwardly checks that
\begin{equation*}
\bpsi[p_n] =  W_{\phi_{\lam}^*}(n+1)
\end{equation*}
for all $n\in \mathbb{N}$, and it follows from \cite{Berg} that $(\bpsi [p_n])_{n\geq 0}$ is indeed a Stieltjes moment determinate sequence of a probability measure $\bpsi$. Its absolute continuity follows from \cite[Proposition 2.4]{Patie-11-Factorization_e}.

Now suppose $\mo < 1+\h$, so that $\lam > 1+\mo > 1$, and observe that \eqref{eq:mom-bpn} factorizes as
\begin{equation*}
\bpsi [p_n] =  \frac{W_{\phi}(n+1)}{(\lam)_n} \frac{(\vat)_n}{n!},
\end{equation*}
where, by the above arguments, the first term in the product is a Stieltjes moment sequence, whereas the second term is the moment sequence of a beta distribution (see \eqref{eq:beta-moments}). Consequently, in this case, one also has that $(\bpsi [p_n])_{n \geq 0}$ is a Stieltjes moment sequence, and we temporarily postpone the proof of its moment determinacy and absolute continuity to after the proof of Lemma \ref{lem:fact}. We write $(\bpsi_\varphi [p_n])_{n \geq 0}$ for the sequence obtained from \eqref{eq:mom-bpn} by replacing $\phi$ with $\varphi$ defined in \eqref{def:varphi} and with the same $\lam$.

\begin{lemma}
\label{lem:fact}
For all $\epsilon \in (0,d_\phi) \cup \{d_\phi\}$ and $\m \in (\bm{1}_{\{\mo < 1+\h\}} + \mo, \lam)$,
we have the following factorization of operators on the space $\Poly$,
\begin{eqnarray}\label{eq:facto_op-1}\label{eq:facto_op-2}\label{eq:facto_op-3}
\bpsi \Lambda_{\phi_{\dz}}  = \gab[\dpt], \quad \gab[\m] \mathrm{V}_{\phm} = \bpsi  \quad \text{and} \quad \bpsi_\varphi \mathrm{U}_\varphi = \bpsi,
\end{eqnarray}
where the second identity holds if $\mo \geq 1+\h$ and the third if $\mo < 1+\h$.
\end{lemma}


\begin{Remark}
\label{rem:facto}
Once we have established the moment determinacy of $\bpsi$ for $\mo < 1+\h$, the operator factorizations in
Lemma \ref{lem:fact} extend to the space of bounded measurable functions. Indeed, \eqref{eq:facto_op-1} implies
\begin{equation*}
B_\phi \times X_{\phi_{\dz}} \stackrel{(d)}{=} B_{\dpt}, \quad B_{\m} \times X_{\phi^*_{\m}} \stackrel{(d)}{=} B_{\phi}
\quad \mbox{and} \quad B_{\varphi} \times X_{\varphi} \stackrel{(d)}{=} B_{\phi},
\end{equation*}
where the second identity holds if $\mo \geq 1+\h$ and the third if $\mo < 1+\h$;
$B_\phi$, $B_{\dpt}$, $B_{\m}$ and $B_{\varphi}$  are random variables with laws
$\bpsi$, $\gab[\dpt]$, $\beta_{\m}$ and $\beta_{\varphi}$, respectively, and $\times$ denotes the
product of independent random variables.
\end{Remark}

\begin{proof}
By \eqref{eq:mom_Ip-0}, we have for all $n \in \N$,
\begin{equation*}
\bpsi[\Lambda_{\phi_{\dz}} p_n] = \frac{(\dz)_n}{W_{\phi}(n+1)} \bpsi[p_n]=
\frac{(\dz)_n}{W_{\phi}(n+1)}  \frac{(\vat)_n}{(\lam)_n}\frac{W_{\phi}(n+1)}{n!} =\frac{(\dz)_n}{n!}  \frac{(\vat)_n}{(\lam)_n}.
\end{equation*}
By considering the cases $\vat=1$ and $\vat<1$ separately, we obtain the desired right-hand side, noting that $\gab[\dz]$ is well-defined since $\lam > d_\phi +1$, due to $ \lam > \mo = (\mo-\h) + \h$ and \cite[Proposition 4.4(1)]{Patie-Savov-GeL}.

Next, we note that by Lemma \ref{lem:phi-Psi1-claims}.\eqref{item-2:lem:phi-Psi1-claims},
$\mo \geq 1+\h$ if and only if $\vat = 1$. So, for all $n \in \N$,
\begin{equation*}
\gab[\m] [\mathrm{V}_{\phm} p_n] = \frac{W_{\phi}(n+1)}{(\m)_n} \gab[\m] [p_n] = \frac{W_{\phi}(n+1)}{(\m)_n} \frac{(\m)_n}{(\lam)_n} = \frac{W_{\phi}(n+1)}{(\lam)_n} = \bpsi[p_n],
\end{equation*}
which, by linearity, shows the second identity.

Finally, we know from Lemma \ref{lem:phi-Psi1-claims}.\eqref{item-2:lem:phi-Psi1-claims} that
$\mo_\varphi \geq 1+\h_\varphi$. Hence, $0$ is the only non-negative root of $u \mapsto u\varphi(u)$, and therefore,
\begin{equation*}
\bpsi_\varphi [p_n] = \frac{W_\varphi(n+1)}{(\lam)_n}.
\end{equation*}
Straightforward computations give that, for any $n \in \N$,
\begin{equation*}
W_\varphi(n+1) = \frac{(\vat+1)_{n}}{(n+1)!} \frac{W_\phi(n+2)}{\phi(1)} \quad \text{and} \quad \mathrm{U}_\varphi p_n(x) = \frac{\varphi(0)}{\varphi(n)} p_n(x) = \frac{\vat\phi(1)(n+1)}{(n+\vat)\phi(n+1)} p_n(x).
\end{equation*}
Putting these observations together yields
\begin{align*}
\bpsi_\varphi[\mathrm{U}_\varphi  p_n]
= \frac{1}{(\lam)_n} \frac{\vat (\vat+1)_{n}}{(n+\vat)}  \frac{(n+1)}{(n+1)!} \frac{W_{\phi}(n+2)}{\phi(n+1)}
=\frac{1}{(\lam)_n} (\vat)_n \frac{1}{n!} W_\phi(n+1) = \bpsi[p_n],
\end{align*}
where we used the recurrence relations for both, the gamma function and $W_\phi$; see e.g.~\eqref{eq:product-Wphi}.
This completes the proof.
\end{proof}

Now suppose that for $\mo < 1+\h$, the measure $\bpsi$ is moment indeterminate. Then, as the sequence $\left(\frac{(\dz)_n}{W_{\phi}(n+1)}\right)_{n \geq 0}$ is a non-vanishing Stieltjes moment sequence, it follows from \eqref{eq:facto_op-1} by invoking \cite[Lemma 2.2]{Berg} that also the beta distribution $\gab[\dpt]$ is moment indeterminate, which is a contradiction. Therefore, we conclude that in all cases, $\bpsi$ is moment determinate, and consequently we have the extended operator factorizations of Remark \ref{rem:facto}.

To obtain the absolute continuity of $\bpsi$ in the case $\mo < 1+\h$, we note that the factorization $\bpsi [p_n] = \frac{W_{\phi}(n+1)}{(\lam)_n}\frac{(\vat)_n}{n!}$ implies, by moment determinacy, that $\bpsi$ is the product convolution of two absolutely continuous measures, and hence, again absolutely continuous.

If we take $\epsilon = d_\phi$, then $\dpt = \vat$ and $\dz = 1$. In this case, we denote
$X_{\phi} = X_{\phi_{\dz}}$. For $\mo \ge 1 + \h$, we obtain from Remark \ref{rem:facto} that
\begin{equation} \label{idB}
B_\phi \times X_{\phi} \stackrel{(d)}{=} B_{\vat} \quad \mbox{and} \quad
B_{\m} \times X_{\phi^*_{\m}} \stackrel{(d)}{=} B_{\phi}.
\end{equation}
Since, by Lemma \ref{lem:Markov-kernel}, the law of $X_{\phi^*_{\m}}$ is supported in $[0,1]$
and has a continuous density, we obtain from the second identity in \eqref{idB} that
the support of $\beta$ is $[0,a]$ for a constant $a \le 1$, and $\beta$ has a continuous
density that is positive on $(0,a)$. But since the law $X_{\phi}$ is also supported in $[0,1]$,
we deduce from the first identity in \eqref{idB} that $a = 1$.

The case $\mo < 1+\h$ follows from analogous arguments with $\gab[\m]$ and $X_{\phm}$ replaced by $\bpsi_\varphi$ and $X_\varphi$, respectively, where we note that, since $\mo_\varphi \geq 1+h_\varphi$, the support of $\bpsi_\varphi$ is $[0,1]$ and $\bpsi_\varphi$ has a continuous density that is positive on $(0,1)$. This completes the proof of \Cref{thm:main-inter-L2}.\eqref{item-1:main-inter-L2}. \qed


\subsection{Proof of \Cref{thm:main-inter-L2}.\eqref{item-2:main-inter-L2}}

We start by proving the following more general intertwining that will be useful in subsequent proofs, recalling the definition of $\Lambda_{\phi_{\dz}}$ in \eqref{eq:mom_Ip-0}.

\begin{proposition}
\label{prop:inter_P}
With $\dpt$ and $\dz$ as in \eqref{eq:def-dtheta} and \eqref{eq:defn-d0}, respectively, we have, for any $\epsilon \in (0,d_\phi) \cup \{d_\phi\}$,
\begin{eqnarray}\label{eq:dz-inter}
\calJ \Lambda_{\phi_{\dz}}  = \Lambda_{\phi_{\dz}} \calJd_{\dpt} \quad \text{on} \quad \Poly.
\end{eqnarray}
\end{proposition}

\begin{Remark}
Note that $\lam$ is the common parameter of the Jacobi type operators in \eqref{eq:dz-inter} while the constant part of the affine drift as well as the non-local components are different. The commonality of $\lam$ is what ensures the isospectrality of these operators, as their spectrum depends only on $\lam$; see \Cref{thm:spectral-representation}.\eqref{item-2:thm:spectral-representation} and \eqref{Q_eigenv}.
\end{Remark}

We split the proof of Proposition \ref{prop:inter_P} into two lemmas. Among other things, our proof hinges on the interesting observation that intertwining relations are stable under perturbation with an operator that commutes with the intertwining operator, see Lemma \ref{lem:op-com} below. Let $\calLd_\mo$ be the operator defined as
\begin{equation}
\label{eq:defL}
\calLd_\mo f(x) = xf''(x) + \mo f'(x),
\end{equation}
write $\calI_h f(x) = - h \diamond f'(x)$, where $h$ satisfies \Cref{assA}, and set $\calL = \calLd_\mo + \calI_h$.

\begin{lemma} \label{lem:inter} With the notation of Proposition \ref{prop:inter_P}, one has
\begin{eqnarray}
\calL \Lambda_{\phi_{\dz}} = \Lambda_{\phi_{\dz}}\calLd_{\dpt} \quad \mbox{on } \Poly. 
\end{eqnarray}
\end{lemma}

\begin{proof}
Using $\h = \int_1^\infty h(r)dr$, one obtains for any $n\in \N$,
\begin{align*}
\calL p_n(x) &= n(n-1)p_{n-1}(x)+ \mo n p_{n-1}(x)-np_{n-1}(x)\int_1^\infty h(r) r^{-(n-1)} r^{-1}dr \nonumber \\
&= n^2 p_{n-1}(x) + (\mo-\h-1)np_{n-1}(x)+np_{n-1}(x)\int_1^\infty h(r) (1-r^{-n}) dr \nonumber \\
&= (n-\tet)\phi(n)p_{n-1}(x).
\end{align*}
Combining this with \eqref{eq:mom_Ip-0} gives
\begin{equation*}
\calL \Lambda_{\phi_{\dz}}p_n(x) = \frac{(\dz)_n}{W_{\phi}(n+1)} (n-\tet)\phi(n) p_{n-1}(x) = \frac{(\dz)_n}{W_{\phi}(n)} (n-\tet)p_{n-1}(x),
\end{equation*}
while on the other hand,
\begin{align*}
\Lambda_{\phi_{\dz}} \calLd_{\dpt} p_n(x) = n(n+\dpt-1)\frac{(\dz)_{n-1}}{W_{\phi}(n)}  p_{n-1}(x) = (n-\tet)\frac{(\dz)_n}{W_{\phi}(n)}p_{n-1}(x)
\end{align*}
where the second equality follows by considering the cases $\vat=1$ and $\vat < 1$ separately. Now, the lemma
follows from the linearity of the involved operators.
\end{proof}

For a Borel measure $\eta$ on $[0,1]$, define $\Lambda_{\eta}f(x)=\int_{0}^{1}f(xy)\eta(dy)$, and denote by $\mathbf{D}_n$ the operator given by $\mathbf{D}_n f(x) = x^n \frac{d^n}{dx^n} f(x)$

\begin{lemma} \label{lem:op-com}
Let $\eta$ be a Borel measure on $[0,1]$ satisfying
$\int_{0}^{1}y^n \eta(dy)<\infty$ for all $n \in \mathbb{N}$.
Then
\begin{equation*}
\mathbf{D}_{n} \Lambda_{\eta} f = \Lambda_{\eta} \mathbf{D}_{n} f
\end{equation*}
for all $n\in \N$ and $f \in C^\infty[0,1]$.
\end{lemma}


\begin{proof}
It follows from the assumptions that
\[
\mathbf{D}_{n} \Lambda_{\eta} f(x)  = x^n \int_0^1 y^n f^{(n)}(xy) \eta (dy) = \Lambda_{\eta} \mathbf{D}_{n} f(x).
\]
\end{proof}

\begin{proof}[Proof of Proposition \ref{prop:inter_P}]
It is now an easy exercise to complete the proof of Proposition \ref{prop:inter_P}. Let us write
\begin{equation*}
\mathbf{A}= \mathbf{D}_{2}+ \lam \mathbf{D}_{1}.
\end{equation*}
Then, for any $f \in \Poly$, we get from Lemmas \ref{lem:inter} and \ref{lem:op-com}
and the linearity of the involved operators,
\begin{equation*}
\calJ \Lambda_{\phi_{\dz}} f = \left( \calL -\mathbf{A}\right) \Lambda_{\phi_{\dz}}f = \Lambda_{\phi_{\dz}}\left( \calLd_{\dpt} -\mathbf{A}\right) f
= \Lambda_{\phi_{\dz}}  \calJd_{\dpt}  f.
\end{equation*}

\end{proof}



Having established the necessary intertwining relation, we are now able to show that $\calJ$ extends to the generator of a Markov semigroup.

\begin{lemma}
\label{lem:feller}
The operator $(\calJ ,\Poly)$ is closable in $C[0,1]$, and its closure is the infinitesimal generator of a Markov semigroup $\J = (\J_t)_{t \geq 0}$ on $C[0,1]$.
\end{lemma}
\begin{proof}
We want to invoke the Hille--Yosida--Ray theorem for Markov generators, see \cite[Theorem 1.30]{Levy-Matters-III}, which requires that $\Poly$ and, for some $q>0$, $(q-\calJ)(\Poly)$ are dense in $C[0,1]$ and in addition, $\calJ$ satisfies the positive maximum principle on $\Poly$.

The density of $\Poly$ in $C[0,1]$ follows from the Stone--Weierstrass theorem.
To show that $(q-\calJ)(\Poly)$ is dense in $C[0,1]$ for some $q >0$, we set $\epsilon = d_\phi$. Then
$\dz = 1$, in which case, we write $\Lambda_{\phi} = \Lambda_{\phi_{\dz}}$.
By Lemma \ref{lem:Markov-kernel}, $\Lambda_{\phi}$ is injective and bounded on $C[0,1]$. So, the inverse $\Lambda_{\phi}^{-1}$ is a linear operator on $\Lambda_{\phi}(\Poly)$. But since $\Lambda_\phi$ is a Markov multiplicative kernel, we get by injectivity that $\Lambda_\phi(\Poly) = \Poly$ and $\Lambda_\phi^{-1}(\Poly) = \Poly$.
Putting these observations together, we deduce from the first intertwining in Proposition \ref{prop:inter_P} that
\begin{equation*}
\calJ =\Lambda_{\phi}\calJd_{\vat} \Lambda_{\phi}^{-1} \quad \text{on} \quad \Poly.
\end{equation*}
Hence, for any $q >0$,
\begin{equation}
\label{eq:qJ-qJ}
(q - \calJ)(\Poly) = (q - \Lambda_\phi \calJd_{\vat} \Lambda_\phi^{-1})(\Poly) = \Lambda_\phi(q-\calJd_{\vat})\Lambda_\phi^{-1}(\Poly) = \Lambda_\phi(q-\calJd_{\vat})(\Poly),
\end{equation}
where we used the trivial commutation of $\Lambda_\phi$ with $q$. For $h \not\equiv 0$, \Cref{assA} guarantees that $\lam > \vat$ since $\lam > 1$ and $\vat = 1-\tet \in (0,1]$. Therefore, $\Poly$ belongs to the domain of $\calJd_{\vat}$, which is explicitly described in~\eqref{eq:C-domain}, and as $\Poly$ is an invariant subspace of the classical Jacobi semigroup $\Q[\vat]$, we get that $\Poly$ is a core of $\calJd_{\vat}$, see \cite[Lemma 1.34]{Levy-Matters-III}. Hence, we obtain from the
reverse direction of the Hille--Yosida--Ray theorem that $(q-\calJd_{\vat})(\Poly)$ is dense in $C[0,1]$ for any $q > 0$. Since the image of a dense subset under a bounded operator with dense range is also dense in the codomain,
$(q - \calJ)(\Poly) = \Lambda_\phi(q-\calJd_{\vat})(\Poly)$ is dense in $C[0,1]$ for any $q > 0$.

Now, consider $f \in \Poly$ and $x_0 \in [0,1]$ such that $f(x_0)=\sup_{x\in [0,1]} f(x)$.
If $x_0 = 0$, one has $f'(x_0) \le 0$ and therefore,
\[
\calJ f(x_0) =  \mo f'(0) - f'(0) \int_1^{\infty} h(r)\frac{1}{r} dr \le f'(0) (\mo - \h) \le 0.
\]
If $x_0 \in (0,1]$, we use Lemma \ref{lem:rewriting-stuff} to write
\[
\calJ f(x_0) = x_0 (1-x_0)f''(x_0)-\left(\lam x_0-\mo\right)f'(x_0)+\int_{0}^{\infty}\left(f(e^{-r}x_0)-f(x_0)\right) \frac{\Pi(dr)}{x_0},
\]
and we observe that
\begin{equation*}
\int_{0}^{\infty} \left(f(e^{-r}x_0)-f(x_0)\right) \frac{\Pi(dr)}{x_0} \leq 0.
\end{equation*}
If $x \in (0,1)$, we must have $f''(x_0)\leq 0$ and $f'(x_0)=0$, from which one obtains $\calJ f(x_0)\leq 0$. On the other hand, if $x_0=1$, then $f'(1)\geq 0$ and so, $\calJ f(1)\leq  -\left(\lam-\mo \right)f'(1) \leq 0$ since $\lam > \mo$.
This shows that $\calJ$ satisfies the positive maximum principle on $\Poly$, and it follows that $\calJ$ extends to the generator of a Feller semigroup $\J = (\J_t)_{t \geq 0}$ in the sense of \cite[Theorem 1.30]{Levy-Matters-III}. The fact that $\J$ is conservative, i.e.~$\J_t \bm{1}_{[0,1]} = \bm{1}_{[0,1]}$, follows from
\begin{equation*}
\J_t \bm{1}_{[0,1]} - \bm{1}_{[0,1]} = \int_0^t \J_s \calJ \bm{1}_{[0,1]} \: ds = 0,
\end{equation*}
which is a consequence of $\calJ \bm{1}_{[0,1]} = 0$;
see e.g.~\cite[Lemma 1.26]{Levy-Matters-III}.
 \end{proof}

\begin{proof}[Proof of \Cref{thm:main-inter-L2}.\eqref{item-2:main-inter-L2}]

To complete the proof it suffices to establish the claims concerning the invariant measure. For $f \in \Poly$ we have,
\[
\bpsi[\calJ \Lambda_{\phi}f] = \bpsi[\Lambda_{\phi}\calJd_{\vat} f] =\gab[\vat] [\calJd_{\vat} f] = 0,
\]
where we have used Proposition \ref{prop:inter_P} with $\epsilon = d_\phi$, Lemma \ref{lem:fact} and the fact that $\gab[\vat]$ is the invariant measure of $\Q[\vat]$. This shows that $\bpsi \calJ \equiv 0$ on the dense subset $\Lambda_{\phi}(\Poly)=\Poly$ of $C[0,1]$, which implies that $\bpsi$ is an invariant measure of $\J$; see for instance \cite[Section 1.4.1]{Bakry_Book}. To show uniqueness, we note that any other invariant measure $\widetilde{\bpsi}$ of $\J$ must have all positive moments finite and satisfy
\begin{equation*}
\widetilde{\bpsi} [ \Lambda_{\phi} \calJd_{\vat} f ] =\widetilde{\bpsi} [ \calJ \Lambda_\phi f ] = 0
\end{equation*}
for any $f \in \Poly$, where we again used that $\Lambda_\phi(\Poly) = \Poly$. By uniqueness of the invariant measure of $\calJd_{\vat}$, we obtain the factorization $\widetilde{\bpsi} \Lambda_{\phi} = \gab[\vat]$ on $\Poly$, and the moment determinacy of $\bpsi$ then forces $\widetilde{\bpsi} = \bpsi$. Finally the extension of $\J$ to a Markov semigroup on $\Leb^2(\bpsi)$ is classical, see for instance the remarks before the theorem, and it is well-known that if $\J$ has a unique invariant measure, it is an ergodic Markov semigroup; see e.g.~\cite[Theorem 5.16]{da-prato:2006}.
\end{proof}

\subsection{Proof of Proposition \ref{prop:co-eigenfunctions}}

To prove Proposition \ref{prop:co-eigenfunctions}, we first show two auxiliary results, the first of which provides
a characterization of the functions $w_n$ appearing in \eqref{eq:v-n-rodrigues}.
We recall that the Mellin transform of a finite measure $\nu$, resp. an integrable function $f$, on $\R_+$ is given by
\begin{equation*}
\calM_{\nu}(z) = \nu[p_{z-1}] = \int_0^\infty x^{z-1}\nu(dx), \quad \text{resp.}~\calM_{f}(z)=\int_0^\infty x^{z-1}f(x)dx,
\end{equation*}
which is valid for at least $z \in 1+i\R$. We denote by $\Em_{p,q}$ (resp.~$\Em_{p,q}'$), with $p < q$ reals, the linear space of functions $f \in C^\infty(\R_+)$ such that there exist $c, c' > 0$ for which, for all $k \in \N$,
\begin{equation*}
\lim_{x\to 0} \left| x^{k+1-p-c} \frac{d^k}{dx^k}f(x)\right| = 0 \quad \text{and} \quad \lim_{x\to \infty} \left| x^{k+1+c'-q} \frac{d^k}{dx^k}f(x)\right| = 0
\end{equation*}
(resp.~the linear space of continuous linear functionals on $\Em_{p,q}$ endowed with a structure of a countably multinormed space as described in \cite[p.~231]{Misra-Lavoine}). For any $n \in \N$ and $x\in [0,1]$, we denote
\begin{equation*}
\pab[\vat]_n(x) = \gab[\vat](x) \P_n^{(\vat)}(x) = \frac{(\lam-\vat)_n}{(\lam)_n} \sqrt{\C_n(\vat)} \Rc_{n} \gab[\lam+n,\vat](x)
\end{equation*}
where $\Rc_{n}$ denotes the Rodrigues operator defined in \eqref{eq:def-Rodrigues} and the last identity follows from \eqref{eq:Jac-Rod}. For any complex number  $a$ we recall that the Pochhammer notation $(a)_z$ to any $z \in \C$ such that $-(z+a) \notin \N$ is given in \eqref{eq:defPoc}, and, for the remainder of the proofs, we shall write $\left\langle \cdot{,}\cdot \right\rangle_{\bpsi}$ for the $\Leb^2(\bpsi)$-inner product, adopting the same notation for other weighted Hilbert spaces.

We know from Lemma \ref{lem:Markov-kernel} that $X_{\phi}$ has a continuous density $\iota$ on $[0,1]$.
The corresponding Markov multiplicative kernel is given by $\Lambda_{\phi} f(x) = \int_0^1 f(xy) \iota(y)dy$.
We write $\iota^*(y)=\iota(1/y)1/y$ and denote by $\hat{\Lambda}_{\phi}$ the operator given by
$\widehat{\Lambda}_{\phi} f(x) = \int_1^{1/x}f(xy)\iota^*(y)dy$.

\begin{proposition}
\label{prop:Convolution}
For any $n \in \N$, the Mellin convolution equation
\begin{equation}
\label{eq:equation_w_n1}
\widehat{\Lambda}_{\phi} w (x) = \pab[\vat]_n(x)
\end{equation}
has a unique solution, in the sense of distributions, given by
\begin{equation}
\label{eq:w_n}
w_n(x) = \frac{(\lam-\vat)_n}{(\lam)_n} \sqrt{\C_n(\vat)} \ \Rc_n (\gab[\lam+n,\lam] \diamond \bpsi) (x) \in \Em = \bigcup_{q>\tet}\Em_{\tet,q}.
\end{equation}
Its Mellin transform is given, for any $z \in \C$ with $\Re(z)>\tet$, by
\begin{equation}
\label{eq:Mellin_w_n}
\calM_{w_n}(z) = \frac{1}{n!} \frac{(\lam-\vat)_n}{(\lam)_n} \sqrt{\C_n(\vat)}  \frac{\Gamma(z)}{\Gamma(z-n)} \calM_{\gab[\lam+n,\lam]}(z) \calM_\bpsi(z). 
\end{equation}
\end{proposition}

\begin{proof}
The proof is an adaptation of the proof of  \cite[Lemma 8.5]{Patie-Savov-GeL} to the current setting. Since the mapping $z \mapsto \calM_{\iota}(z)=\calM_{\iota^*}(1-z)$ is analytic on $\Re(z)>0$ and  $|\calM_{\iota}(z)| \leq \calM_{\iota}(\Re(z)) <\infty$, for any $\Re(z)>0$, see for instance \cite[Proposition 6.8]{Patie-Savov-GeL}, we deduce from \cite[Theorem 11.10.1]{Misra-Lavoine} that $\iota \in \Em'_{0,q}$, for every $q>0$ and $\iota^*  \in \Em'_{p,1}$ for every $p<1$.
So, for $w \in \Em'_{0,q}, \:q>0$ and with $0<\Re(z)<q$, $p_z(x)=x^{z} \in \Em_{0,q}$, we have 	
\begin{align*}
&\mathcal{M}_{\widehat{\Lambda}_{\phi} w}(z)  = \langle \widehat{\Lambda}_{\phi} w,p_{z-1}\rangle_{\Em'_{0,q},\Em_{0,q}}
= \int_0^1 \left(\int_1^{1/x} w(xy) \iota \bigg(\frac{1}{y}\bigg) \frac{1}{y} dy \right) x^{z-1} dx\\
&= \int_0^1 \left(\int_x^{1} w(r) \iota \bigg(\frac{x}{r}\bigg) \frac{1}{r} dr \right) x^{z-1} dx =
\int_0^1 w(r) \bigg( \int_0^r x^{z-1} \iota \bigg(\frac{x}{r}\bigg) \frac{1}{r} dx \bigg) dr\\
&= \int_0^1 w(r) \bigg(\int_0^1 (ur)^{z-1} \iota(u) du \bigg) dr
= \int_0^1 u^{z-1} \iota(u) du \int_0^1 r^{z-1} w(r) dr\\
& = \mathcal{M}_{\iota}(z) \mathcal{M}_w(z).
\end{align*}
On the other hand, for any $n \in \N$, we get, from \cite[11.7.7]{Misra-Lavoine} and a simple computation,
\begin{align*}
\calM_{\pab[\vat]_n}(z) = \frac{1}{n!} (\lam-\vat)_n \sqrt{\C_n(\vat)} \frac{\Gamma(z)}{\Gamma(z-n)}  \frac{(\vat)_{z-1}}{(\lam)_{z+n-1}}.
\end{align*}
So we deduce that the Mellin transform of a solution $w$ to \eqref{eq:equation_w_n1} takes the form	
\begin{align*}
\mathcal{M}_{w}(z) =  \frac{\calM_{\pab[\vat]_n}(z)}{\calM_{\iota}(z)} &= \frac{1}{n!} \sqrt{\C_n(\vat)} \frac{\Gamma(z)}{\Gamma(z-n)}  (\lam-\vat)_n \frac{(\vat)_{z-1}}{(\lam)_{z+n-1}} \frac{W_\phi(z)}{\Gamma(z)} \\
&= \frac{1}{n!} \sqrt{\C_n(\vat)}   \frac{\Gamma(z)}{\Gamma(z-n)}  \frac{(\lam-\vat)_n}{(\lam)_n} \frac{(\lam)_{z-1}}{(\lam+n)_{z-1}} \calM_\bpsi(z) \\
&= \frac{1}{n!} \frac{(\lam-\vat)_n}{(\lam)_n} \sqrt{\C_n(\vat)}  \frac{\Gamma(z)}{\Gamma(z-n)}  \calM_{\gab[\lam+n,\lam]}(z) \calM_\bpsi(z).
\end{align*}
Since for $\Re(z)>\tet$, $z\mapsto \Mg(z)$ is analytic with $|\Mg(z)|\leq \Mg(\Re(z)) <\infty$, we deduce from \cite[Theorem 11.10.1]{Misra-Lavoine} that $\bpsi \in \Em'_{\tet,q}$, for any $q>\tet$. Hence, by means of
\cite[11.7.7]{Misra-Lavoine}, we have that $w \in \Em'_{\tet,q}$ with $w = w_n$ is a solution to \eqref{eq:equation_w_n1}, and the uniqueness of the solution follows from the uniqueness of Mellin transforms in the distributional sense.
\end{proof}

\begin{lemma}
\label{lem:Mellin-bpsi}
For $a > \tet$ and $b \in \R$, we have the estimate
\begin{equation*}
\left| \calM_\bpsi(a+ib) \right| \leq C|b|^{-\Delta},
\end{equation*}
which holds uniformly on bounded $a$-intervals and for $|b|$ large enough,
where $C > 0$ is a constant depending on $\phi$ and the considered $a$-interval, and $\Delta$ is given in \eqref{Delta}. 
\end{lemma}

\begin{proof}
By uniqueness of $W_\phi$ in the space of positive-definite functions, the Mellin transform of $\bpsi$ is given by
\begin{equation*}
\calM_\bpsi(z) =  \frac{(\vat)_{z-1}}{(\lam)_{z-1}}\frac{W_\phi(z)}{\Gamma(z)}
\end{equation*}
where $z = a+ib$, with $a > \tet \geq 0$. Invoking \cite[Equation (6.20)]{Patie-Savov-Bern}, we get the following estimate, which holds uniformly on bounded $a$-intervals and for $|b|$ large enough,
\begin{equation}
\label{eq:W/Gamma-estimate}
\left| \frac{W_\phi(a+ib)}{\Gamma(a+ib)}\right| \leq C_\phi |b|^{\phi(0)+\overline{\nu}(0)}
\end{equation}
with $C_\phi > 0$ a constant depending on $\phi$, and where, for any $y > 0$, $\overline{\nu}(y) = \int_y^\infty \nu(ds)$ with $\nu$ denoting the L\'evy measure of $\phi$. We know form Lemma \ref{lem:phi-Psi1-claims}.\eqref{item-1:lem:phi-Psi1-claims}
that $\phi(0) = \mo-\h - 1$ if $\mo \ge 1+\h$ and $\phi(0) = 0$ otherwise.
Moreover, if $\mo \ge 1+\h$, we obtain from  \eqref{eq:def-phi} that $\nu(dy) = \overline{\Pi}(y)dy$.
Thus to utilize the estimate in \eqref{eq:W/Gamma-estimate} we need to identify $\overline{\nu}(0)$ in the case $\mo < 1+\h$. To do that, let us write $\Psi(u) = (u-\tet)\phi(u) = (u-\tet)\phi_{\tet}(u-\tet)$, where $\phi_{\tet}(u) = \phi(u+\tet)$. From the fact that $\Psi(\tet)=0$, we conclude that $\Psi(u+\tet) = u\phi_{\tet}(u)$ is itself a function of the form \eqref{eq:ll}, which gives $\nu_{\tet}(dy) = \overline{\Pi}_{\tet}(y)dr, y>0$, where $\nu_{\tet}$ is the L\'evy measure of $\phi_{\tet}$ and
$\Pi_{\tet}$ the L\'evy measure of $\Psi(u+\tet)$ obtained via \eqref{eq:ll}. As $\phi_{\tet}$ is a Bernstein function it is given, for $u \geq -\tet$, by
\begin{equation*}
\phi_{\tet}(u) = \kappa+u+u\int_0^\infty e^{-uy} \overline{\nu}_{\tet}(y)dy
\end{equation*}
for some $\kappa \geq \tet$. Thus, for $u \geq 0$,
\begin{align*}
\phi(u) = \phi_{\tet}(u-\tet) &= \kappa+(u-\tet)+(u-\tet)\int_0^\infty e^{-(u-\tet)y} \overline{\nu}_{\tet}(y)dy \\
&= (\kappa-\tet) + u + u \int_0^\infty e^{-uy} e^{\tet y}\overline{\nu}_{\tet}(y)dy -\tet \int_0^\infty e^{-uy}e^{\tet y} \overline{\nu}_{\tet}(y)dy \\
&=  (\kappa-\tet) + u + u \int_0^\infty e^{-uy} e^{\tet y}\overline{\nu}_{\tet}(y)dy -\tet u \int_0^\infty e^{-uy} \int_0^y e^{\tet s} \overline{\nu}_{\tet}(s)ds dy \\
&= (\kappa-\tet) + u + u \int_0^\infty e^{-uy} \left( e^{\tet r}\overline{\nu}_{\tet}(y) - \tet\int_0^y e^{\tet s} \overline{\nu}_{\tet}(s)ds\right)dy.
\end{align*}
The third equality follows from Tonelli's theorem, justified as all integrands therein are non-negative, and using $e^{-uy} = \int_y^\infty ue^{-us} ds$. Thus we deduce
\begin{equation*}
\overline{\nu}(y) = e^{\tet y}\overline{\nu}_\tet(y) -\tet\int_0^y e^{\tet s} \overline{\nu}_\tet(s)ds = \int_y^\infty e^{\tet s}\nu_\tet(ds),
\end{equation*}
where the latter follows by some straightforward integration by parts and shows that $\nu$ is indeed the L\'evy measure of $\phi$. Next, an application of \cite[Proposition 4.1(9)]{Patie-Savov-GeL} together with another integration by parts yields $\int_0^\infty e^{-\tet y} \overline{\Pi}(y)dy \leq \int_0^\infty \overline{\Pi}(y)dy = \h$. Putting the different pieces together, we get $\overline{\nu}(0) = \overline{\nu}_\tet(0) \leq \h$, so that in all cases $\overline{\nu}(0) \leq \h$. Therefore, we can deduce from \eqref{eq:W/Gamma-estimate} that
\begin{equation}
\label{eq:W/Gamma-estimate-2}
\left| \frac{W_\phi(a+ib)}{\Gamma(a+ib)}\right| \leq C_\phi |b|^{\phi(0)+\h},
\end{equation}
which, as before, holds uniformly on bounded $a$-intervals and for $|b|$ large enough. Next, we recall the following classical estimate for the gamma function
\begin{equation}
\label{eq:classical-Gamma}
\lim_{|b| \to \infty} C_a|b|^{\frac{1}{2}-a} e^{\frac{\pi}{2}|b|} \left|\Gamma(a+ib)\right| =1
\end{equation}
where $C_a > 0$ is a constant continuously depending on $a$. Combining this estimate with the one in \eqref{eq:W/Gamma-estimate-2} we thus get, uniformly on bounded $a$-intervals and for $|b|$ large enough,
\begin{equation*}
\left|\calM_\bpsi(z)\right| \leq C|b|^{-\lam+\vat + \phi(0) + \h}
\end{equation*}
for a constant $C > 0$. Since $C$ is a function of $C_\phi$ and the constants in the estimate for the gamma function, it follows that it only depends on $\phi$ and $a$-interval on which the estimate holds. Finally, the fact that $\Delta = \lam-\vat-\phi(0)- \h$ follows from Lemma \ref{lem:phi-Psi1-claims}.\eqref{item-1:lem:phi-Psi1-claims}.
%
\end{proof}

\begin{proof}[Proof of Proposition \ref{prop:co-eigenfunctions}]
Note that $\Rc_n \gab[\lam+n,\lam] \in C^\infty(0,1)$ and trivially, $\bpsi \in \Leb^1[0,1]$. Then, well-known properties of convolution give $\Rc_n\left( \gab[\lam+n,\lam] \diamond \bpsi\right) = \Rc_n \gab[\lam+n,\lam] \diamond \bpsi$, and that $w_n$ is a well-defined $C^\infty(0,1)$-function. To show that $\Delta > \frac{1}{2}$ implies $w_n \in \Leb^2[0,1]$, we note that the classical estimate for the gamma function given in \eqref{eq:classical-Gamma} yields that, for $z = a+ib$ with $a > n$ fixed,
\begin{equation*}
\lim_{|b| \to \infty} \left| \frac{\Gamma(z)}{\Gamma(z-n)} \calM_{\gab[\lam+n,\lam]}(z) \right| = \lim_{|b| \to \infty}  (\lam)_n  \left|  \frac{\Gamma(z)}{\Gamma(z-n)}\frac{\Gamma(z+\lam-1)}{\Gamma(z+\lam+n-1)} \right| = C
\end{equation*}
where $C$ is a positive constant depending only on $a$, $\lam$, and $n$. Thus, we get from \eqref{eq:w_n} that $\calM_{w_n}$ has the same rate of decay along imaginary lines as $\calM_{\bpsi}$. So Lemma \ref{lem:Mellin-bpsi} together with Parseval's identity for Mellin transforms shows that $w_n \in \Leb^2[0,1]$. Finally, since $w_n \in C^\infty(0,1)$, the differentiability of $\V_n^\phi$ is determined by the differentiability of $\bpsi$. Invoking Lemma \ref{lem:Mellin-bpsi}, we get for $a > \tet$ and $|b|$ large enough that
\begin{equation*}
\left| (a+ib)^n \calM_{\bpsi}(a+ib) \right| \leq C |b|^{n-\Delta}
\end{equation*}
uniformly on bounded $a$-intervals, where $C > 0$ is a constant.
A classical Mellin inversion argument then gives $\bpsi \in C^{\lceil \Delta \rceil - 2}(0,1)$
if $\Delta \ge 2$.
\end{proof}

%

\subsection{Proof of \Cref{thm:spectral-representation}}

To prove this result we need to develop further intertwinings for $\calJ$ and lift these to the level of semigroups. We write $\calJ_\varphi$ for the non-local Jacobi operator with parameters $\lam$, $\mo_\varphi$ and $h_\varphi$, as in Lemma \ref{lem:phi-Psi1-claims}, which is in one-to-one correspondence with the Bernstein function $\varphi$ defined in \eqref{def:varphi}.

\begin{lemma}
\label{prop:inter_P-2}
For any $\m \in (\bm{1}_{\{\mo < 1+\h\}}+\mo,\lam)$, the following identities hold on $\Poly$:
\begin{eqnarray}\label{Jm_intertwin}\label{Jm_intertwin2}
\calJd_\m {\rm{V}}_{\phm}  = \mathrm{V}_{\phm}\calJ \quad \text{and} \quad \calJ_\varphi \rm{U}_\varphi = \rm{U}_\varphi \calJ
\end{eqnarray}
in the cases $\mo \geq 1+\h$ and $\mo < 1+\h$, respectively.
\end{lemma}

\begin{proof}
It suffices to prove that $\calLd_\m {\rm{V}}_{\phm}  = {\rm{V}}_{\phm} \calL$ and $\calL_\varphi \rm{U}_\varphi  = \rm{U}_\varphi \calL$ hold on $\Poly$, where we write $\calL_\varphi = \calLd_{\mo_\varphi} +\calI_{h_\varphi}$ and refer to \eqref{eq:defL} and the subsequent discussion for the definitions, as then the same arguments as in the proof of Proposition \ref{prop:inter_P} will go through. In the case $\mo \geq 1+\h$, we have, for any $n\in \N$ and using the recurrence relation of the gamma function,
\begin{align*}
\calLd_\m {\rm{V}}_{\phm}p_n(x) &= \frac{W_{\phi}(n+1)}{(\m)_n}\calLd_\m p_n(x) = \frac{W_{\phi}(n+1)}{(\m)_n}n(n+\m-1) p_{n-1}(x) = \frac{W_{\phi}(n+1)}{(\m)_{n-1}} n p_{n-1}(x).
\end{align*}
On the other hand, since $W_\phi(n+1)=\phi(n)W_\phi(n)$ and $\vat = 1$,
\begin{equation*}
{\rm{V}}_{\phm}\calL  p_n(x) = \frac{W_{\phi}(n)}{(\m)_{n-1}} n\phi(n) p_{n-1}(x)
= \frac{W_{\phi}(n+1)}{(\m)_{n-1}} n  p_{n-1}(x),
\end{equation*}
which proves the claim in this case. Finally,
\begin{equation*}
\calL_\varphi \mathrm{U}_\varphi p_n(x)  = \frac{\varphi(0)}{\varphi(n)} \calL_\varphi  p_n(x)= \frac{\varphi(0)}{\varphi(n)} n\varphi(n) p_{n-1}(x) = \varphi(0)np_{n-1}(x),
\end{equation*}
while on the other hand, using the definition of $\varphi$ in \eqref{def:varphi},
\begin{equation*}
\mathrm{U}_\varphi \calL p_n(x)  = (n-\tet)\phi(n) \mathrm{U}_\varphi	p_{n-1}(x) =  (n-\tet)\phi(n) \frac{\varphi(0)}{\varphi(n-1)} p_{n-1}(x)= \varphi(0) np_{n-1}(x),
\end{equation*}
which,  by linearity, completes the proof of the lemma.
\end{proof}

The following result lifts the intertwinings of the Propositions \ref{prop:inter_P} and \ref{prop:inter_P-2} to the level of semigroups. We here write $\J=\J^\phi = (\J_t^\phi)_{t \geq 0}$ to emphasize the one-to-one correspondence, for fixed $\lam$, between $\phi$ and $\J$.

\begin{proposition}
\label{prop:intertwining}
For all $\epsilon \in (0,d_\phi) \cup \{d_\phi\}$ and $\m \in (\bm{1}_{\{\mo < 1+\h\}}+\mo,\lam)$,
the following identities hold for all $t \geq 0$ on the appropriate $\Leb^2$-spaces,
\begin{eqnarray}\label{eq:inter-right-d0}\label{eq:inter-left-m}\label{eq:inter-left-varphi}
\J_t^\phi \Lambda_{\phi_{\dz}} = \Lambda_{\phi_{\dz}} \Q[\dpt]_t, \quad \Q[\m]_t \mathrm{V}_{\phm} = \mathrm{V}_{\phm} \J_t^\phi \quad \text{and} \quad \J_t^\varphi {\rm{U}}_\varphi = {\rm{U}}_\varphi \J_t^\phi
\end{eqnarray}
with the latter two holding for $\mo \geq 1+\h$ and $\mo < 1+\h$, respectively.
\end{proposition}

We need an auxiliary result concerning the corresponding intertwining operators, which extends their boundedness from $C[0,1]$ to the corresponding weighted Hilbert spaces. For two Banach spaces $B$ and $\widetilde{B}$, we denote by $\Bop{B}{\widetilde{B}}$ the space of bounded linear operators from $B$ to $\widetilde{B}$.

\begin{lemma}
\label{lem:Lambda-L2-bounded}
For all $\epsilon \in (0,d_\phi) \cup \{d_\phi\}$, $\m \in (\bm{1}_{\{\mo < 1+\h\}} + \mo, \lam)$ and
$p \in \{1,\ldots,\infty\}$, the operators $\Lambda_{\phi_{\dz}}$, $\mathrm{V}_{\phm}$ and $\rm{U}_\varphi$ belong to $\Bop{\Leb^p(\gab[\dpt])}{\Leb^p(\bpsi)}$, $\Bop{\Leb^p(\bpsi)}{\Leb^p(\gab[\m])}$ and $\Bop{\Leb^p(\bpsi)}{\Leb^p(\bpsi_\varphi)}$, respectively, and have operator norm 1.
\end{lemma}

\begin{proof}
Let $f \in \Poly$ and $p < \infty$. Then, applying Jensen's inequality to the Markov multiplicative kernel $\Lambda_{\phi_{\dz}}$ together with Lemma \ref{lem:fact} gives
\begin{equation*}
\bpsi \left[ \left(\Lambda_{\phi_{\dz}} f\right)^p\right] = \int_0^1 \left(\Lambda_{\phi_{\dz}} f(x)\right)^p \bpsi(dx) \leq \int_0^1 \Lambda_{\phi_{\dz}} f^p(x) \bpsi(dx) = \bpsi[\Lambda_{\phi_{\dz}} f^p] = \gab[\dpt] [ f^p]
\end{equation*}
where we used that $f^p \in \Poly$. Since $\gab[\dpt]$ is a probability measure on the compact set $[0,1]$, it follows that $\Poly$ is a dense subset of $\Leb^p(\gab[\dpt])$; see e.g.~\cite[Corollary 22.10]{Driver03}. So by density, we conclude that $\Lambda_{\phi_{\dz}}$ is in $\Bop{\Leb^p(\gab[\dpt])}{\Leb^p(\bpsi)}$ with operator norm less than or equal to 1. Equality then follows from $\Lambda_{\phi_{\dz}} \bm{1}_{[0,1]} = \bm{1}_{[0,1]}$. The case $p = \infty$ is a straightforward consequence of $\Lambda_{\phi_{\dz}}$ being a Markov multiplicative kernel, and the claims regarding the other operators are deduced similarly from Lemma \ref{lem:fact}.
\end{proof}



%

Next, since $\calJ$ and $\calJd_{\dpt}$ are generators of $C[0,1]$-Markov semigroups, it follows that their resolvent operators, given for $q > 0$, by
\begin{equation*}
\mathds{R}_q = (q-\calJ)^{-1}, \quad \text{and} \quad {\rm{R}}_q = (q-\calJd_{\dpt})^{-1}
\end{equation*}
are bounded, linear operators on $C[0,1]$.  We write ${\rm{R}}_q^\m$ (resp.~$\mathds{R}_q^\varphi$) for the resolvent corresponding to $\mathbf{J}_\m$ (resp.~$\calJ_\varphi$).

\begin{lemma}
\label{prop:resolvent-intertwining}
For all $q > 0$, $\epsilon \in (0,d_\phi) \cup \{d_\phi\}$ and $\m \in (\bm{1}_{\{\mo < 1+\h\}}+\mo,\lam)$,
we have the following identities on $\Poly$:
\begin{equation}
\label{eq:resolvent-intertwining}
\mathds{R}_q \Lambda_{\phi_{\dz}} = \Lambda_{\phi_{\dz}} \mathrm{R}_q, \quad \mathrm{V}_{\phm} \mathds{R}_q = \mathrm{R}_q^\m \mathrm{V}_{\phm} \quad \text{and} \quad {\rm{U}_\varphi} \mathds{R}_q =  \mathds{R}_q^\varphi {\rm{U}}_\varphi,
\end{equation}
where the second one holds for $\mo \ge 1 +\h$ and the last one for $\mo < 1 + \h$.
\end{lemma}

\begin{proof}
We shall only provide the proof of the first claim, which relies on the intertwining in Proposition \ref{prop:inter_P}, as the other claims follow by invoking Proposition \ref{prop:inter_P-2} and involve the same arguments, mutatis mutandis. First, suppose that $\mathds{R}_q(\Poly) \subseteq \Poly$ and $\mathrm{R}_q(\Poly) \subseteq \Poly$ and let $f \in \Poly$ so that there exists $g \in \Poly$ such that $(q-\calJd_{\dpt})g = f$. Applying $\Lambda_{\phi_{\dz}}$ to both sides of this equality gives
\begin{equation*}
\Lambda_{\phi_{\dz}} f = \Lambda_{\phi_{\dz}}(q-\calJd_{\dpt})g = (\Lambda_{\phi_{\dz}} q - \Lambda_{\phi_{\dz}}\calJd_{\dpt})g = (q\Lambda_{\phi_{\dz}} - \calJ \Lambda_{\phi_{\dz}})g = (q-\calJ)\Lambda_{\phi_{\dz}} g
\end{equation*}
where in the third equality we have used Proposition \ref{prop:inter_P}, which is justified as $g \in \Poly$. This equality may be rewritten as $\mathds{R}_q \Lambda_{\phi_{\dz}} f = \Lambda_{\phi_{\dz}} g$ and consequently, for any $f \in \Poly$, we get
\begin{equation*}
\mathds{R}_q \Lambda_{\phi_{\dz}} f = \Lambda_{\phi_{\dz}} g = \Lambda_{\phi_{\dz}} \mathrm{R}_q  f.
\end{equation*}
Thus it remains to show the inclusions $\mathds{R}_q(\Poly) \subseteq \Poly$ and $\mathrm{R}_q(\Poly) \subseteq \Poly$ for which we recall, from the proof of Proposition \ref{prop:inter_P}, that $\calJ = \calL-\mathbf{A}$ with $\calL p_n = (n-\tet)\phi(n)p_{n-1}$, for any $n \geq 1$. A straightforward computation gives that $\mathbf{A} p_n = (\mathbf{D}_2 + \lam \mathbf{D}_1)p_n = (n(n-1)+\lam n)p_n$ and hence
\begin{equation*}
(q-\calJ)p_n = (q+n(n-1)+\lam n)p_n-(n-\tet)\phi(n)p_{n-1},
\end{equation*}
from which it follows, by the injectivity of $\mathds{R}_q$ on $\Poly \subseteq C[0,1]$, that
\begin{equation*}
\mathds{R}_q\left((q+n(n-1)+\lam n)p_n-(n-\tet)\phi(n)p_{n-1}\right) = p_n.
\end{equation*}
Rearranging the above yields the equation
\begin{equation}
\label{eq:resolvent-recursive}
 \mathds{R}_q p_n = \frac{1}{(q+n(n-1)+\lam n)} p_n + \frac{(n-\tet)\phi(n)}{(q+n(n-1)+\lam n)} \mathds{R}_q p_{n-1},
\end{equation}
which is justified as, for any $q > 0$, both roots of the quadratic equation $n^2+(\lam-1)n+q = 0$ are always negative. Note that $\mathds{R}_q p_0 = q^{-1}$ so by iteratively using the equality in \eqref{eq:resolvent-recursive}, we conclude that, for any $n \in \N$, $\mathds{R}_q p_n \in \Poly$, and by linearity $\mathds{R}_q(\Poly) \subseteq \Poly$ follows. Similar arguments applied to $\mathrm{R}_q$ then allow us to also conclude that $\mathrm{R}_q(\Poly) \subseteq \Poly$, which completes the proof.
\end{proof}

\begin{proof}[Proof of Proposition \ref{prop:intertwining}]
We are now able to complete the proof of Proposition \ref{prop:intertwining}. As was shown in the proof of
Proposition \ref{prop:resolvent-intertwining} above and using the notation therein, $\mathds{R}_q(\Poly) \subseteq \Poly$ and $\mathrm{R}_q(\Poly) \subseteq \Poly$, so that on $\Poly \subseteq C[0,1]$ we have
\begin{equation*}
\mathds{R}_q^2 \Lambda_{\phi_{\dz}} = \mathds{R}_q\mathds{R}_q  \Lambda_{\phi_{\dz}}  = \mathds{R}_q  \Lambda_{\phi_{\dz}}  \mathrm{R}_q =  \Lambda_{\phi_{\dz}}  \mathrm{R}_q \mathrm{R}_q = \Lambda_{\phi_{\dz}}  \mathrm{R}_q^2,
\end{equation*}
and, by induction, for any $n \in \N$,
\begin{equation*}
\mathds{R}_q^n  \Lambda_{\phi_{\dz}}  = \Lambda_{\phi_{\dz}}  \mathrm{R}_q^n.
\end{equation*}
In particular, for any $f \in \Poly$ and $t > 0$,
\begin{equation*}
(n/t)\mathds{R}_{n/t}^n  \Lambda_{\phi_{\dz}}  f =  \Lambda_{\phi_{\dz}}  (n/t)\mathrm{R}_{n/t}^n f.
\end{equation*}
Now, taking the strong limit in $C[0,1]$ as $n\to\infty$ of the above yields, by the exponential formula \cite[Theorem 8.3]{pazy:1983} and the continuity of the involved operators guaranteed by Lemma \ref{lem:Markov-kernel}, for any $f \in \Poly$ and $t \geq 0$,
\begin{equation}
\label{eq:inter-semi-P}
\J_t  \Lambda_{\phi_{\dz}} f =  \Lambda_{\phi_{\dz}} \Q[\dpt]_t f,
\end{equation}
where $(\Q[\dpt]_t)_{t \geq 0}$ is the classical Jacobi semigroup on $C[0,1]$ with parameters $\lam$ and $\dpt$. By density of $\Poly$ in $\Leb^2(\gab[\dz])$ and since Lemma \ref{lem:Lambda-L2-bounded} with $p=2$ gives $\Lambda_{\phi_{\dz}} \in \Bop{\Leb^2(\gab[\dpt])}{\Leb^2(\bpsi)}$, it follows that the identity in \eqref{eq:inter-semi-P} extends to $\Leb^2(\gab[\dz])$, which completes the proof of the first identity. The other two identities follow from similar arguments and so the proof is omitted.
\end{proof}

For $\lam > s \geq 1$, we define, for $n \in \N$, the quantity $\mathfrak{c}_n(s)$ as
\begin{equation}
\label{eq:defn-cn}
\mathfrak{c}_n(s) = \frac{(s)_n}{n!} \sqrt{\frac{\C_n(s)}{\C_n(1)}} = \sqrt{\frac{(s)_n}{n!}\frac{(\lam-1)_n}{(\lam-s)_n}}
\end{equation}
where the first equality comes from some straightforward algebra given the definition of $\C_n(s)$ in \eqref{eq:Cn-lam-mo}. Note that, with $s=1$ we get $\mathfrak{c}_n(1) = 1$, for all $n$. We shall need the following result.

\begin{lemma}
\label{lem:seq-cn}
For any $\lam > s > r \geq 1$ the mapping $n \mapsto \frac{\mathfrak{c}_n(s)}{\mathfrak{c}_n(r)}$ is strictly increasing on $\N$ with
\begin{equation}
\label{eq:asymptotic-cn}
\lim_{n \to \infty} \frac{\mathfrak{c}_n(s)}{n^{s-1}} = \sqrt{\frac{\Gamma(\lam-s)}{\Gamma(s)\Gamma(\lam-1)}}.
\end{equation}
\end{lemma}

\begin{proof}
Using the definition in \eqref{eq:defn-cn} we get that
\begin{equation*}
\frac{\mathfrak{c}_n^2(s)}{\mathfrak{c}_n^2(r)} = \prod_{j=0}^{n-1} \frac{(s+j)(\lam-r+j)}{(r+j)(\lam-s+j)}.
\end{equation*}
Since $s > r$ each term in the product is strictly greater than 1 and together with Stirling's formula for the gamma function this completes the proof.
\end{proof}
%

%

Now, we write $\Lambda^*_{\phi} \colon \Leb^2(\beta) \to \Leb^2(\beta_{\vat})$, $\mathrm{V}_{\phm}^* : \Leb^2(\gab[\m]) \to \Leb^2(\bpsi)$ and ${\rm{U}}_\varphi^* : \Leb^2(\bpsi_\varphi) \to \Leb^2(\bpsi)$ for the Hilbertian adjoints of the operators
$\Lambda_{\phi}$, $\mathrm{V}_{\phm}$ and ${\rm{U}}_\varphi$, respectively.

\begin{proposition}
\label{prop:riesz}
Let $\epsilon \in (0,d_\phi) \cup \{d_\phi\}$ and $\m \in (\bm{1}_{\{\mo < 1+\h\}} + \mo ,\lam)$. Then,
for all $n \in \N$,
\begin{equation}
\label{eq:alt-V-n-0}
\P_n^\phi  = \Lambda_{\phi} \P_n^{(\vat)},
\end{equation}
and with $\dz$ as in \eqref{eq:defn-d0}, the sequence $\left(\mathfrak{c}_n(\dz\right)\P_n^\phi)_{n \geq 0}$ is a complete Bessel sequence in $\Leb^2(\bpsi)$ with Bessel bound 1. Furthermore, for any $n \in \N$, we have,
\begin{equation}
\label{eq:alt-V-n}
\V_n^\phi  = \mathfrak{c}_n(\m) \mathrm{V}_{\phm}^*\P_n^{(\m)} \quad \mbox{if } \mo \geq 1+\h,
\end{equation}
while
\begin{equation}
\label{eq:alt-V-n-2}
\V_n^\phi  = \frac{\mathfrak{c}_n(\m)}{\mathfrak{c}_n(\vat)} {\rm{U}}_\varphi^* \mathrm{V}_{\phm}^*\P_n^{(\m)}
\quad \mbox{if } \mo < 1+\h,
\end{equation}
and $(\V_n^\phi)_{n \geq 0}$ is the unique biorthogonal sequence to $(\P_n^\phi)_{n \geq 0}$ in $\Leb^2(\bpsi)$, which is equivalent to $\V_n^\phi$ being the unique $\Leb^2(\bpsi)$-solution to $\Lambda_\phi^* g = \P_n^{(\vat)}$ for any $n \in \N$. In all cases $\left(\frac{\mathfrak{c}_n(\vat)}{\mathfrak{c}_n(\m)} \V_n^\phi\right)_{n \geq 0}$ is a complete Bessel sequence in $\Leb^2(\bpsi)$ with Bessel bound 1.
\end{proposition}
\begin{remark}\label{rem:bounds}
Note that Proposition \ref{prop:riesz} yields $\Leb^2(\bpsi)$-norm bounds for the functions $\P_n^\phi$ and $\V_n^\phi$. Indeed, writing $||\cdot||_\bpsi$ for the $\Leb^2(\bpsi)$-norm we get, from the boundedness claims of Lemma \ref{lem:Lambda-L2-bounded}, for all $\epsilon \in (0,d_\phi) \cup \{d_\phi\}$ and $\m \in (\bm{1}_{\{\mo < 1+\h\}} + \mo,\lam)$,
\begin{equation*}
||\P_n^\phi||_\bpsi \leq \frac{1}{\mathfrak{c}_n(\dz)}\leq C n^{1-\dz} \quad \text{and} \quad ||\V_n^\phi||_\bpsi \leq \frac{\mathfrak{c}_n(\m)}{\mathfrak{c}_n(\vat)}\leq C n^{\m-\vat}
\end{equation*}
where $C \ge 0$ and we used Lemma \ref{lem:seq-cn} for the two estimates.
We show in the proof below that
\begin{equation*}
\frac{\mathfrak{c}_n(\m)}{\mathfrak{c_n}(\dz)\mathfrak{c}_n(\vat)} = \frac{\mathfrak{c}_n(\m)}{\mathfrak{c}_n(\dpt)},
\end{equation*}
and since $\m > \dpt$, invoking again Lemma \ref{lem:seq-cn}, we have that the above ratio grows with $n$.
\end{remark}

\begin{proof}
By \eqref{eq:Jpolpsi}, \eqref{eq:mom_Ip-0}, \eqref{J_pol} and linearity of $\Lambda_{\phi}$, one obtains
\begin{equation}\label{eq:LambdaP}
\Lambda_\phi \P_n^{(\vat)}(x) = \sqrt{\C_n(\vat)} \sum_{k=0}^n \frac{(-1)^{n+k}}{(n-k)!}  \frac{(\lam-1)_{n+k}}{(\lam-1)_{n\phantom{+k}}} \frac{(\vat)_n}{(\vat)_k} \frac{k!}{W_\phi(k+1)} \frac{x^k}{k!} = \P_n^\phi(x).
\end{equation}
Recall that the sequence $(\P_n^{(\vat)})_{n \geq 0}$ forms an orthonormal basis of $\Leb^2(\gab[\vat])$ and thus, as the image under a bounded operator of an orthonormal basis, we get that $(\P_n^\phi)_{n \geq 0}$ is a Bessel sequence in $\Leb^2(\bpsi)$ with Bessel bound given by the operator norm of $\Lambda_\phi$, which by Lemma \ref{lem:Lambda-L2-bounded}, is 1. If $\vat <1$, we have $\mathfrak{c}_n(\dz) = \mathfrak{c}_n(1) =1$, so that the first claim is proved in this case. In the case $\vat = 1$ we suppose, without loss of generality, that $d_\phi > 0$ and $\epsilon \in (0,d_\phi)$. Then $\P_n^\phi$ reduces to
\begin{equation*}
\P_n^\phi(x) = \sqrt{\C_n(1)} \sum_{k=0}^n \frac{(-1)^{n+k}}{(n-k)!} \frac{(\lam-1)_{n+k}}{(\lam-1)_{n\phantom{+k}}} \frac{n!}{k!} \frac{x^k}{W_{\phi}(k+1)}
\end{equation*}
and one gets from \eqref{eq:Jpolpsi}, \eqref{eq:mom_Ip-0}, \eqref{J_pol} and linearity of $\Lambda_{\phi_{\dz}}$
\begin{align*}
\Lambda_{\phi_{\dz}} \P_n^{(\dz)} (x)
&= \sqrt{\C_n(\dz)}  \sum_{k=0}^n \frac{(-1)^{n+k}}{(n-k)!}  \frac{(\lam-1)_{n+k}}{(\lam-1)_{n\phantom{+k}}} \frac{(\dz)_n}{W_\phi(k+1)} \frac{x^k}{k!} = \mathfrak{c}_n(\dz)\P^{\phi}_n(x).
\end{align*}
By Lemma \ref{lem:Lambda-L2-bounded}, $\Lambda_{\phi_{\dz}}$ belongs to $\Bop{\Leb^2(\gab[\dz])}{\Leb^2(\bpsi)}$ with operator norm 1 and thus, by similar arguments as above, we deduce that $(\mathfrak{c}_n(\dz)\P^{\phi}_n)_{n \geq 0}$ is a Bessel sequence in $\Leb^2(\bpsi)$ with Bessel bound $1$.

We continue with the claims regarding $\V_n^\phi$, starting with the case $\vat=1$. A simple calculation
shows that for $f \in \Leb^2(\gab[\m])$,
\begin{equation*}
\mathrm{V}_{\phm}^* f(x) = \frac{1}{\bpsi(x)} \widehat{\mathrm{V}}_{\phm} (f \gab[\m])(x),
\end{equation*}
where $\widehat{\mathrm{V}}_{\phm} f(x) = \int_{1}^{1/x}f(xy)\nu_\m^*(y)dy$ with $\nu_\m^*(y) = \nu_\m(1/y)/y$ and $\nu_\m$ the density of the random variable $X_{\phm}$, whose existence is provided in Lemma \ref{lem:Markov-kernel}. Thus it suffices to show that, for all $n \in \N$,
\begin{equation*}
w_n(x) = \mathfrak{c}_n(\m)\widehat{\mathrm{V}}_{\phm} (\P_n^{(\m)} \gab[\m])(x) = \mathfrak{c}_n(\m) \widehat{\mathrm{V}}_{\phm} \pab[\m]_n(x).
\end{equation*}
To do that, we note that for $\Re(z) > \tet$,
\begin{align*}
\calM_{ \widehat{\mathrm{V}}_{\phm} \pab[\m]_n}(z)
&= \calM_{\nu_\m}(z) \calM_{\pab[\m]_n}(z) \\
&=   \frac{1}{n!} \frac{(\lam-\m)_n}{(\lam)_n}\sqrt{\C_n(\m)} \frac{W_\phi(z)}{(\m)_{z-1}}\frac{\Gamma(z)}{\Gamma(z-n)}    \calM_{\gab[\lam+n,\m]}(z)  \\
&=\frac{1}{n!} \frac{(\lam-\m)_n}{(\lam)_n}\sqrt{\C_n(\m)} \frac{\Gamma(z)}{\Gamma(z-n)} \calM_{\gab[\lam+n,\lam]}(z) \calM_{\bpsi}(z)
\end{align*}
and
\begin{align*}
\mathfrak{c}_n(\m) \frac{(\lam-\m)_n}{(\lam)_n} \frac{\sqrt{\C_n(\m)}}{n!} =  \frac{(\lam-\m)_n}{(\lam)_n}  \frac{(\m)_n}{n!} \frac{\C_n(\m)}{n!\sqrt{\C_n(1)}} = \frac{\sqrt{\C_n(1)}}{n!} \frac{(\lam-1)_n}{(\lam)_n}.
\end{align*}
So invoking the uniqueness claim of Proposition \ref{prop:Convolution} yields the representation \eqref{eq:alt-V-n} for $\vat = 1$. The case $\vat < 1$ follows by similar arguments, albeit with more tedious algebra, and its proof is omitted.

Next, we note that for $\vat =1$,
\begin{align*}
\mathrm{V}_{\phm} \P^{\phi}_n(x) &= \sqrt{\C_n(1)} \sum_{k=0}^n \frac{(-1)^{n+k}}{(n-k)!} \frac{(\lam-1)_{n+k}}{(\lam-1)_{n\phantom{+k}}} \frac{n!}{k!}  \frac{W_\phi(k+1)}{(\m)_k} \frac{x^k}{W_\phi(k+1)} = \mathfrak{c}^{-1}_n(\m) \P_n^{(\m)}(x).
\end{align*}
As $(\P_n^{(\m)})_{n\geq 0}$ is an orthonormal sequence in $\Leb^2(\gab[\m])$, we have for all $n,p \in \N$,
\begin{equation*}
\delta_{np} =\langle \P_n^{(\m)}, \P_p^{(\m)}\rangle_{\gab[\m]} = \mathfrak{c}_n(\m) \langle \mathrm{V}_{\phm}  \P^\phi_n, \P_p^{(\m)}\rangle_{\gab[\m]} = \mathfrak{c}_n(\m) \langle \P^\phi_n, \mathrm{V}_{\phm}^*\P_p^{(\m)}\rangle_{\bpsi},
\end{equation*}
and thus we get that $(\V_n^\phi)_{n \geq 0}$ is a biorthogonal sequence in $\Leb^2(\bpsi)$ of $(\P_n^\phi)_{n \geq 0}$.
It follows from Lemma \ref{lem:Lambda-L2-bounded} that ${\mathrm{V}_{\phm}^*}$ belongs to
$\Bop{\Leb^2(\gab[\m])}{\Leb^2(\bpsi)}$ and has operator norm $1$. So, since
$(\P_n^{(\m)})_{n \geq 0}$ forms an orthonormal basis of ${\rm{L}}^2(\gab[\m])$, one obtains that
$(\mathfrak{c}_n^{-1}(\m)\V_n^\phi)_{n \geq 0}$ is a Bessel sequence in $\Leb^2(\bpsi)$ with Bessel bound 1.

To show uniqueness, we first observe that any sequence $(g_n)_{n \geq 0} \in \Leb^2(\bpsi)$ biorthogonal to $(\P_n^\phi)_{n \geq 0}$ must satisfy
\begin{equation*}
\delta_{np} = \langle \P_n^\phi, g_p \rangle_\bpsi =
\langle \Lambda_\phi \P_n^{(\vat)}, g_p \rangle_{\gab[\vat]} =
\langle \P_n^{(\vat)}, \Lambda_\phi^* g_p \rangle_{\gab[\vat]},
\end{equation*}
which means that $(\Lambda_\phi^* g_n)_{n \geq 0}$ is biorthogonal to $(\P_n^{(\vat)})_{n \geq 0}$. However, since $(\P_n^{(\vat)})_{n \geq 0}$ is an orthonormal basis for $\Leb^2(\gab[\mo])$, one must have
$\Lambda_\phi^* g_p =  \P_n^{(\vat)} = \Lambda_\phi^* \V_n^\phi$,
and so, $\Lambda_\phi^* \left(\V_n^\phi-g_n\right)=0$.
Since by Lemma \ref{lem:Markov-kernel}, $\mathrm{Ran}(\Lambda_\phi)$ is dense in $\Leb^2(\bpsi)$, one has $\mathrm{Ker}(\Lambda_\phi^*) = \{0\}$, and we conclude that $(\V_n^\phi)_{n \geq 0}$ is the unique sequence in $\Leb^2(\bpsi)$ biorthogonal to $(\P_n^\phi)_{n \geq 0}$.

Finally, assume now that $\vat < 1$. Then, using the definition of $\varphi$ in \eqref{def:varphi} and
Lemma \ref{lem:phi-Psi1-claims}.\eqref{item-2:lem:phi-Psi1-claims}, we get that
\begin{align*}
\P_n^\varphi(x) = \sqrt{\C_n(1)} \sum_{k=0}^n \frac{(-1)^{n+k}}{(n-k)!} \frac{(\lam-1)_{n+k}}{(\lam-1)_{n\phantom{+k}}} \frac{n!(k+1)}{(\vat+1)_{k}} \frac{\phi(1) x^k}{W_\phi(k+2)}.
\end{align*}
On the other hand, since $\mathrm{U}_\varphi p_n = \frac{\varphi(0)}{\varphi(n)}p_n$, see \eqref{eq:mom_Up}, simple algebra yields that
\begin{equation}
\label{eq:U-P}
\mathrm{U}_\varphi \P_n^\phi(x) = \frac{(\vat)_n}{n!}\sqrt{\C_n(\vat)} \sum_{k=0}^n \frac{(-1)^{n+k}}{(n-k)!}   \frac{(\lam-1)_{n+k}}{(\lam-1)_{n\phantom{+k}}} \frac{n!(k+1)}{(\vat+1)_k}\frac{\phi(1) x^k}{W_\phi(k+2)} = \mathfrak{c}_n(\vat) \P_n^\varphi(x).
\end{equation}
We know that, since $\lam > \m > 1+\mo = \mo_\varphi$, $(\V_n^\varphi)_{n\geq 0} = (\mathfrak{c}_n(\m) {\rm{V}}_{\phm}^* \P_n^{(\m)})_{n\geq 0}$ is the unique sequence biorthogonal to $(\P_n^\varphi)_{n \geq 0}$. Combining this with \eqref{eq:U-P} gives
\[
\delta_{np} = \ang{\P_n^\varphi, \V_n^\varphi }_{\beta} =
\ang{ \mathrm{U}_\varphi \P_n^\phi, \frac{\mathfrak{c}_n(\m)}{\mathfrak{c}_n(\vat)} {\rm{V}}_{\phm}^* \P_n^{(\m)}}_{\beta}
= \ang{ \P_n^\phi, \frac{\mathfrak{c}_n(\m)}{\mathfrak{c}_n(\vat)} \mathrm{U}_\varphi^* {\rm{V}}_{\phm}^* \P_n^{(\m)}}_{\beta},
\]
which shows that $(\V_n^\phi)_{n \geq 0}$ is biorthogonal  to $(\P_n^{\phi})_{n \ge 0}$.
Uniqueness follows as above.

Finally, the completeness of $(\V_n^\phi)_{n \geq 0}$ is a consequence of the fact that it is, in all cases and by Lemmas \ref{lem:Markov-kernel} and \ref{lem:Lambda-L2-bounded}, the image under a bounded linear operator with dense range of the sequence $\left(\frac{\mathfrak{c}_n(\m)}{\mathfrak{c}_n(\vat)} \P_n^{(\m)}\right)_{n \geq 0}$, which is itself is complete.
\end{proof}

\begin{proof}[Proof of \Cref{thm:spectral-representation}]
We tackle the different claims of \Cref{thm:spectral-representation} sequentially. Setting $\epsilon = d_\phi$ in \eqref{eq:def-dtheta} we get, by the first intertwining in Proposition \ref{prop:intertwining}, \eqref{eq:alt-V-n-0}
and the spectral expansion of the self-adjoint semigroup $\Q[\vat]$ in \eqref{Q_sem}, that for any $f \in \Leb^2(\gab[\vat])$ and $t \geq 0$,
\begin{align*}
\J_t \Lambda_\phi f = \Lambda_\phi \Q[\vat]_t f = \sum_{n=0}^\infty e^{-\lambda_n t} \langle f, \P_n^{(\vat)} \rangle_{\gab[\vat]} \P_n^\phi = \sum_{n=0}^\infty e^{-\lambda_n t} \langle \Lambda_\phi f, \V_n^\phi \rangle_{\bpsi} \P_n^\phi
\end{align*}
where the second identity is justified by $(\langle f, \P_n^{(\vat)}\rangle_{\gab[\vat]})_{n \geq 0} \in \ell^2(\N)$ and the fact that $(\P_n^\phi)_{n \geq 0}$ is a Bessel sequence in $\Leb^2(\bpsi)$, see \cite[Theorem 3.1.3]{Christensen2003a}, and the last identity uses the fact that, by Proposition \ref{prop:riesz}, $\V_n^\phi$ is the unique $\Leb^2(\bpsi)$-solution to the equation $\Lambda_\phi^* \V_n^\phi = \P_n^{(\vat)}$. It follows that $\P_n^\phi$ is an eigenfunction for $\J_t$ with eigenvalue $e^{-\lambda_n t}$. Taking the adjoint of the first identity in Proposition \ref{prop:intertwining}
and using the self-adjointness of $\Q[\vat]_t$ on $\Leb^2(\gab[\vat])$ yields $\Lambda_\phi^* \J_t^* = \Q[\vat]_t \Lambda_\phi^*$ and thus, for any $n \in \N$ and $t \geq 0$,
\begin{equation*}
\Lambda_\phi^* {\J_t}^* \V_n^\phi = \Q[\vat]_t \Lambda_\phi^* \V_n^\phi = \Q[\vat]_t \P_n^{(\vat)} = e^{-\lambda_n t} \P_n^{(\vat)} = e^{-\lambda_n t}\Lambda_\phi^* \V_n^\phi.
\end{equation*}
Since $\mathrm{Ker}(\Lambda_\phi^*) = \{0\}$, this shows that ${\J_t}^* \V_n^\phi = e^{-\lambda_n t} \V_n^\phi$.

Next, let $S_t$ be the linear operator on $\Leb^2(\bpsi)$ defined by
\begin{equation*}
S_t f = \sum_{n=0}^\infty \langle \J_t f, \V_n^\phi \rangle_{\bpsi} \P_n^\phi
\end{equation*}
so that, by the above observations,
\begin{equation*}
S_t f = \sum_{n=0}^\infty \langle f, \J_t^* \V_n^\phi \rangle_{\bpsi} \P_n^\phi = \sum_{n=0}^\infty  e^{-\lambda_n t} \langle f, \V_n^\phi \rangle_{\bpsi} \P_n^\phi.
\end{equation*}
For convenience, we set ${\bm{\V}}_n^\phi = \frac{\mathfrak{c}_n(\vat)}{\mathfrak{c}_n(\m)} \V_n^\phi$, $n \in \N$. Then, for any $t > 0$ and $f \in \Leb^2(\bpsi)$ we have, for $C >0$ a constant independent of $n$,
\begin{equation*}
\sum_{n=0}^\infty e^{-2\lambda_n t} \left|\left\langle f, \V_n^\phi \right\rangle_{\bpsi}\right|^2 = \sum_{n=0}^\infty e^{-2 \lambda_n t} \frac{\mathfrak{c}_n^2(\m)}{\mathfrak{c}_n^2(\vat)} \left|\left\langle f, {\bm{\V}}_n^\phi\right\rangle_{\bpsi}\right|^2 \leq C \sum_{n=0}^\infty \left|\left\langle f, {\bm{\V}}_n^\phi \right\rangle_{\bpsi}\right|^2 \leq C \bpsi [f^2] < \infty
\end{equation*}
where the first inequality follows from the asymptotic in \eqref{eq:asymptotic-cn} combined with the decay of the sequence $(e^{-2\lambda_n t})_{n \geq 0}$, $t > 0$, and the second inequality follows from the Bessel property of $({\bm{\V}}_n^\phi )_{n \geq 0}$ guaranteed by Proposition \ref{prop:riesz}. Hence we deduce that $\left(e^{-\lambda_n t} \langle f, \V_n^\phi \rangle_{\bpsi}\right)_{n \geq 0} \in \ell^2(\N)$ and, as $(\P_n^\phi)_{n \geq 0}$ is a Bessel sequence, it follows that $S_t$ defines a bounded linear operator on $\Leb^2(\bpsi)$ for any $t > 0$, again by~\cite[Theorem 3.1.3]{Christensen2003a}.  However, $S_t = \J_t$ on $\mathrm{Ran}(\Lambda_\phi)$, a dense subset of $\Leb^2(\bpsi)$. Therefore, by the bounded linear extension theorem, we have $S_t = \J_t$ on $\Leb^2(\bpsi)$ for any $t > 0$. Note that, by similar Bessel sequence arguments as above, for any $N \geq 1$,
\begin{equation*}
\norm{\J_tf -\sum_{n=0}^N e^{-\lambda_n t} \left\langle f, \V_n^\phi \right\rangle_{\bpsi}\P_n^\phi}_\bpsi^2 \leq \bpsi [f^2] \sup_{n \geq N+1} e^{-2 \lambda_n t} \frac{\mathfrak{c}_n^2(\m)}{\mathfrak{c}_n^2(\vat)}.
\end{equation*}
Since the supremum on the right-hand side is decreasing in $N$ for any $t > 0$, we get
\begin{equation*}
\J_t = \lim_{N \to \infty} \sum_{n=0}^N e^{-\lambda_n t} \P_n^\phi \otimes \V_n^\phi
\end{equation*}
in the operator norm topology, where each $\sum_{n=0}^N e^{-\lambda_n t} \P_n^\phi \otimes \V_n^\phi$ is of finite rank. This completes the proof of \eqref{item-1:thm:spectral-representation} of \Cref{thm:spectral-representation} and also shows that $\J_t$ is a compact operator for any $t > 0$, which is claim \eqref{item-2:thm:spectral-representation}.

Next, the intertwining identity \eqref{eq:inter-right-d0} and the completeness of $(\P_n^\phi)_{n \geq 0}$ and $(\V_n^\phi)_{n \geq 0}$ enable us to invoke \cite[Proposition 11.4]{Patie-Savov-GeL} to obtain the equalities for algebraic and geometric multiplicities in \eqref{item-3:thm:spectral-representation}, and also to conclude that
\begin{equation*}
\sigma_p(\J_t) = \sigma_p(\J_t^*) = \sigma_p(\Q[\vat]_t) = \{e^{-\lambda_n t} : \ n \in \N\}.
\end{equation*}
Since $\J_t$ is compact, $\J_t^*$ is so too, and thus $\sigma(\J_t) \setminus \{0\} = \sigma_p(\J_t)$ as well as
$\sigma(\J_t^*) \setminus \{0\} = \sigma_p(\J_t^*)$. To establish the remaining equalities we use the immediate compactness of $\J$ to invoke \cite[Corollary 3.12]{EngelNagel} and obtain $\sigma(\J_t) \setminus \{0\} = e^{t\sigma(\calJ)}$, while we also have from \cite[Theorem 3.7]{EngelNagel} that, $\sigma_p(\J_t) \setminus \{0\} = e^{t\sigma_p(\calJ)}$. Putting all of these together completes the proof of \eqref{item-3:thm:spectral-representation}.

It remains to prove the last item concerning the self-adjointness of $\J$. Clearly if $h \equiv 0$ then $\J$ is self-adjoint, as in this case $\bpsi$ reduces to $\gab[\mo]$ and $\J$ reduces to the classical Jacobi semigroup $\Q[\mo]$, which is self-adjoint on $\Leb^2(\gab[\mo])$. Now suppose that $\J$ is self-adjoint on $\Leb^2(\bpsi)$, that is $\J_t = \J_t^*$ for all $t \geq 0$. By differentiating in $t$ the identity, for any $n, m \in \N$,
\begin{equation*}
\left\langle \J_t p_n, p_m \right\rangle_{\bpsi} = \left\langle p_n, \J_t p_m \right\rangle_{\bpsi}
\end{equation*}
we deduce, by a simple application of Fubini's Theorem using the finiteness of the measure $\bpsi$, that
\begin{equation}
\label{eq:symmetry-calJ}
\left\langle \calJ p_n, p_m \right\rangle_{\bpsi} = \left\langle p_n, \calJ p_m \right\rangle_{\bpsi}.
\end{equation}
Note that \eqref{eq:symmetry-calJ} holds trivially if either $n=0$ or $m=0$, or if $n=m$, so we may suppose that $n \neq m$; all together we take, without loss of generality, $n > m > 0$. Now, for any $n \geq 1$, a straightforward calculation shows that
\begin{equation}
\label{eq:calJ-pn}
\calJ p_n(x) = \Psi(n)p_{n-1}(x)-\lambda_n p_n(x)
\end{equation}
where we recall from \eqref{eq:def-phi} that $\Psi(n) = (n-\tet)\phi(n)$ and from \eqref{eq:def-lambda-n} that $\lambda_n = n^2+(\lam-1)n$. Using \eqref{eq:calJ-pn} on both sides of \eqref{eq:symmetry-calJ} and rearranging gives
\begin{equation}
\label{eq:n-m-eqn}
\left(\lambda_n-\lambda_m\right) \bpsi p_{n+m} = \left(\Psi(n)-\Psi(m)\right) \bpsi p_{n+m-1}.
\end{equation}
By \eqref{eq:mom-bpn} and the recurrence relations for $W_\phi$ and the gamma function, the ratio $\bpsi [p_{n+m}] / \bpsi [p_{n+m-1}]$ evaluates to
\begin{equation*}
\frac{\bpsi [p_{n+m}]}{\bpsi [p_{n+m-1}]} = \frac{(n+m+\tet)}{(n+m+\lam-1)}\frac{\phi(n+m)}{(n+m)} = \frac{\Psi(n+m)}{\lambda_{n+m}},
\end{equation*}
so that substituting into \eqref{eq:n-m-eqn} shows that the following must be satisfied
\begin{equation}
\label{eq:Psi-n-m}
\Psi(n+m)\left(\lambda_n-\lambda_m\right) = \lambda_{n+m}\left(\Psi(n)-\Psi(m)\right).
\end{equation}
Next, we write $\Psi$ as
\begin{equation*}
\Psi(n) = n^2 + (\mo-\h-1)n + n \int_1^\infty (1-r^{-n})h(r)dr = n^2 + (\mo-1)n + n\int_1^\infty r^{-n}h(r)dr
\end{equation*}
where the first equality is simply the definition of $\Psi$ in \eqref{eq:defpsi} and the second follows from the assumption that $\h = \int_1^\infty h(r) dr < \infty$. Let us write $G(n) = n^2+(\mo-1)n$ and $H(n) = n\int_1^\infty r^{-n}h(r)dr$. By direct verification we get
\begin{align*}
G(n+m)\left(\lambda_n-\lambda_m\right) &= (n-m)\left[(n+m)^3+(\lam + \mo-2)(n+m)^2 (\lam-1)(\mo-1)(n+m)\right] \\
&= \lambda_{n+m}\left(G(n)-G(m)\right),
\end{align*}
so that \eqref{eq:Psi-n-m} is equivalent to
\begin{equation}
\label{eq:H-n-m}
H(n+m)\left(\lambda_n-\lambda_m\right) = \lambda_{n+m}\left(H(n)-H(m)\right).
\end{equation}
Observe that
\begin{equation*}
H(n+m)\left(\lambda_n-\lambda_m\right) = (n-m)(n+m)\left(n+m+\lam-1\right) \int_1^\infty r^{-(n+m)}h(r)dr,
\end{equation*}
while
\begin{equation*}
\lambda_{n+m}\left(H(n)-H(m)\right) = (n+m)(n+m+\lam-1)\left(n\int_1^\infty r^{-n}h(r)dr  - m\int_1^\infty r^{-m}h(r)dr \right).
\end{equation*}
Hence canceling $(n+m)(n+m+\lam-1)$ on both sides of \eqref{eq:H-n-m}, then dividing by $nm$ and rearranging the resulting equation yields
\begin{equation*}
\int_1^\infty r^{-m} h(r)dr = \int_1^\infty r^{-n}h(r)dr + \left(\frac{1}{n}-\frac{1}{m}\right) \int_1^\infty r^{-(n+m)} h(r)dr.
\end{equation*}
Applying the dominated convergence theorem when taking the limit as $n \to \infty$ of the right-hand side we find that, for all $m > 0$ with $m \neq n$,
\begin{equation*}
\int_1^\infty r^{-m} h(r)dr = 0,
\end{equation*}
which implies that $h \equiv 0$. This completes the proof of claim \eqref{item-4:thm:spectral-representation}.\end{proof}

To conclude this section we give a result concerning the intertwining operators in Proposition \ref{prop:intertwining} which illustrates that, except in the self-adjoint case of $h \equiv 0$ and $\mo \leq 1$, none of these operators admit bounded inverses. This latter fact combined with the relation \eqref{eq:LambdaP} imply that $(\P_n^\phi)_{n\geq0}$ is a not a Riesz sequence in $\Leb^2(\bpsi)$, as it is not the image of an orthogonal sequence by an invertible bounded operator, see \cite{Christensen2003a}. Recall that a quasi-affinity is a linear operator between two Banach spaces with trivial kernel and dense range.


\begin{proposition}
\label{prop:quasi-similar}
Let $\epsilon \in (0,d_\phi) \cup \{d_\phi\}$ and $\m \in (\bm{1}_{\{\mo < 1+\h\}} + \mo ,\lam)$.
\begin{enumerate}[\rm(i)]
\item \label{item-1:prop:quasi-similar} Then, the operators $\Lambda_{\phi_{\dz}}:\Leb^2(\gab[\dz]) \to \Leb^2(\bpsi)$, $\mathrm{V}_{\phm}:\Leb^2(\bpsi)\to\Leb^2(\gab[\m])$ and $\mathrm{U}_\varphi:\Leb^2(\bpsi_\varphi) \to \Leb^2(\bpsi)$ are quasi-affinities.
\item \label{item-2:prop:quasi-similar}
If $\dz = 1$, the operator $\Lambda_{\phi_{\dz}}$ admits a bounded inverse if and only if $h \equiv 0$ and $\mo \leq 1$. In all cases, $\mathrm{V}_{\phm}$ and $\mathrm{U}_\varphi$ do not admit bounded inverses.
\end{enumerate}
\end{proposition}

\begin{proof}
Since polynomials belong to the $\Leb^2$-range of the operators $\Lambda_{\phi_{\dz}}$, $\mathrm{V}_{\phm}$, and $\mathrm{U}_\varphi$, we get, by moment determinacy, that each of these has dense range in their respective codomains. For the remaining claims we proceed sequentially by considering each operator individually, starting with $\Lambda_{\phi_{\dz}}$. Proposition \ref{prop:riesz} gives that, for any $n \in \N$,
\begin{equation*}
\P_n^\phi = \frac{1}{\mathfrak{c}_n(\dz)}\Lambda_{\phi_{\dz}} \P_n^{(\dz)},
\end{equation*}
and also that $(\P_n^\phi)_{n \geq 0}$ and $(\V_n^\phi)_{n \geq 0}$ are biorthogonal. Consequently,
\begin{equation*}
\delta_{np} = \left \langle \P_n^\phi, \V_p^\phi \right\rangle_\bpsi = \left \langle \frac{1}{\mathfrak{c}_n(\dz)}\Lambda_{\phi_{\dz}} \P_n^{(\dz)}, \V_p^\phi \right\rangle_\bpsi = \frac{1}{\mathfrak{c}_n(\dz)} \left\langle \P_n^{(\dz)}, \Lambda_{\phi_{\dz}}^* \V_p^\phi \right\rangle_{\gab[\dz]}.
\end{equation*}
However, as $(\P_n^{(\dz)})_{n \geq 0}$ forms an orthonormal basis for $\Leb^2(\gab[\dz])$ it must be its own unique biorthogonal sequence, which forces
\begin{equation*}
\frac{1}{\mathfrak{c}_n(\dz)}\Lambda_{\phi_{\dz}}^* \V_n^\phi =  \P_n^{(\dz)}
\end{equation*}
for all $n \in \N$. Thus we conclude that $\Poly \subseteq \Ran(\Lambda_{\phi_{\dz}}^*)$, so that by moment determinacy of $(\gab[\dz])$, we get that $\Ker(\Lambda_{\phi_{\dz}}) = \{0\}$. Next, by straightforward computation we have, for any $n \in \N$,
\begin{equation}
\label{eq:ratio-1}
\norm{p_n}_{\gab[\dz]}^{-2}\norm{\Lambda_{\phi_{\dz}} p_n}_\bpsi^2 =  \frac{W_\phi(2n+1)}{W_\phi^2(n+1)}\frac{(\dz)_n^2}{(\dz)_{2n}} = \frac{W_{\phi_{\dz}}(2n+1)}{W_{\phi_{\dz}}^2(n+1)}\frac{(n!)^2}{(2n)!}
\end{equation}
where the second equality follows by using the definition of $\phi_{\dz}$, see \eqref{eq:defphi-d0}, together with the recurrence relation for $W_{\phi_{\dz}}$. Now, the same arguments as in the proof of \cite[Theorem 7.1(2)]{Patie-Savov-GeL} may be applied, see e.g.~Section 7.3 therein, to get that the ratio in \eqref{eq:ratio-1} tends to 0 as $n \to \infty$ if and only if $\phi_{\dz}(0) = 0$ and $\Pi \equiv  0 \iff h \equiv 0$. This is because, with the notation of the aforementioned paper, the expression for $\frac{\psi(u)}{u^2}$ is equal to $\frac{\phi_{\dz}(u)}{u}$ in our notation, and we have $\sigma^2 = 1$ from $\lim_{u \to \infty} \frac{\phi_{\dz}(u)}{u} = 1$. From the definition of $\phi_{\dz}$ in \eqref{eq:defphi-d0} we find that, if $\dz = 1$, then $\phi_{\dz}(0) = \phi(0) = 0$ and from Lemma \ref{lem:phi-Psi1-claims}.\eqref{item-2:lem:phi-Psi1-claims} we get that $\phi(0) = \mo-1-\h$ if $\mo \geq 1+\h$ while $\phi(0)$ is always zero when $\mo < 1+\h$, which shows that if $\dz =1$ then $\phi(0)=0 \iff \mo \leq 1$. On the other hand, from \eqref{eq:defphi-d0}, it is clear that if $\dz > 1$ then always $\phi_{\dz}(0)=0$. This completes the proof of the claims regarding $\Lambda_{\phi_{\dz}}$.
Next, by Proposition \ref{prop:riesz}, $\V_n^\phi \in \Ran(\mathrm{V}_{\phm}^*)$, for each $n \in \N$, and as proved in Proposition \ref{prop:intertwining}, the sequence $(\V_n^\phi)_{n \geq 0}$ is complete. Thus $\Ran(\mathrm{V}_{\phm}^*)$ is dense in $\Leb^2(\gab[\m])$, or equivalently $\Ker(\mathrm{V}_{\phm}) = \{0\}$. By direct calculation we get that,
\begin{equation}
\label{eq:ratio-2}
\norm{p_n}_{\bpsi}^{-2} \norm{\mathrm{V}_{\phm} p_n}_{\gab[\m]}^2  = \frac{W_\phi^2(n+1)}{W_\phi(2n+1)} \frac{(\m)_{2n}}{(\m)_n^2} = \prod_{k=1}^n \frac{\phm(k)}{\phm(k+n)}
\end{equation}
where $\phm$ was defined in \eqref{eq:def-phi-m}. Now the fact that $\lim_{u \to \infty} \frac{\phi(u)}{u} = 1$ allows us to deduce $\lim_{u \to \infty} \phm(u) = 1$ and, as noted earlier, $\phm$ is a Bernstein function and hence non-decreasing. As the case $\phm \equiv 1$ is excluded by the assumption on $\m$, we get that, as $n \to \infty$, the ratio in \eqref{eq:ratio-2} tends to 0. Next, by taking the adjoint of \eqref{eq:inter-left-varphi} we get
\begin{equation*}
\mathrm{U}_\varphi^* {\J_t^\varphi}^* =  {\J_t^\phi}^* \mathrm{U}_\varphi^*,
\end{equation*}
and using this identity we get that $\mathrm{U}_\varphi^* \V_n^\phi$ is an eigenfunction for ${\J_t^\varphi}^*$ associated to the eigenvalue $e^{-\lambda_n t}$. Then, \Cref{thm:spectral-representation}.\eqref{item-3:thm:spectral-representation} forces $\mathrm{U}_\varphi^* \V_n^\phi = \V_n^\varphi$, and the latter is a complete sequence, whence $\Ker(\mathrm{U}_\varphi) = \{0\}$. Finally, another straightforward calculation gives that
\begin{equation*}
\norm{p_n}_{\bpsi_\phi}^{-2} \norm{\mathrm{U}_{\varphi} p_n}_{\bpsi_\varphi}^2 = \frac{\varphi^2(0)}{\varphi^2(n)} \frac{\phi(2n+1)}{\vat \phi(1)} \frac{2n+\vat}{2n+1} = \varphi(0) \frac{2n+\vat}{\left(n+\vat\right)^2} \left(\frac{n+!}{\phi(n+1)}\right)^2 \frac{\phi(2n+1)}{2n+1},
\end{equation*}
and using the fact that $\lim_{u \to \infty} \frac{\phi(u)}{u} = 1$ we conclude that the right-hand side tends to 0 as $n \to \infty$.
\end{proof}

\subsection{Proof of \Cref{thm:convergence-equilibrium}\eqref{item-1:thm:convergence-equilibrium}}

\Cref{thm:spectral-representation} gives, for any $f \in \Leb^2(\bpsi)$ and $t > 0$,
\begin{equation*}
\J_t f = \sum_{n=0}^{\infty} e^{-\lambda_n t} \langle f, \V_n^\phi  \rangle_{\bpsi} \P_n^\phi
\end{equation*}
so that, since $\lambda_0=0$ and $\P_0^\phi \equiv 1 \equiv \V_0^\phi$,
\begin{equation}
\label{eq:J-t-minus}
\J_tf - \bpsi f = \sum_{n=1}^\infty e^{-\lambda_n t} \langle f, \V_n^\phi  \rangle_{\bpsi} \P_n^\phi.
\end{equation}
Next, we note that
\begin{align}
\label{eq:sup-t}
\sup_{n \geq 1} e^{-2n\lam t}\frac{\mathfrak{c}_n^2(\m)}{\mathfrak{c}_n^2(\dpt)} \leq e^{- 2\lam t}\frac{\mathfrak{c}_1^2(\m)}{\mathfrak{c}_1^2(\dpt)} \iff 2\lam t \geq \log\left(\frac{(\m+1)(\lam-\dpt+1)}{(\dpt+1)(\lam-\m+1)}\right)
\end{align}
since
\begin{equation*}
e^{-2(n-1)\lam t}\frac{\mathfrak{c}_n^2(\m)}{\mathfrak{c}_n^2(\dpt)}\frac{\mathfrak{c}_1^2(\dpt)}{\mathfrak{c}_1^2(\m)} = \prod_{j=1}^{n-1} e^{-2\lam t} \frac{(\m+j)(\lam-\dpt+j)}{(\dpt+j)(\lam-\m+j)},
\end{equation*}
and $\m > \dpt$, which is trivial when $\vat < 1$, as then $\m > 1 > \dpt = \vat$, while if $\vat =1$ we have $\m-1> d_\phi > \dpt-1$ from \cite[Proposition 4.4(1)]{Patie-Savov-GeL}. Now, we claim that the following computation is valid, writing $||\cdot||_\bpsi$ again for the $\Leb^2(\bpsi)$-norm and $\bm{\V}_n^\phi = \frac{\mathfrak{c}_n(\vat)}{\mathfrak{c}_n(\m)}\V_n^\phi$,
\begin{align*}
\norm{\J_t f - \bpsi f}^2
&\leq \sum_{n=1}^\infty \frac{1}{\mathfrak{c}_n^2(\dz)} \left| \langle \J_t f, \V_n^\phi \rangle_{\bpsi} \right|^2 = \sum_{n=1}^\infty e^{-2 \lambda_n t}\frac{\mathfrak{c}_n^2(\m)}{\mathfrak{c}_n^2(\dpt)} \left| \left\langle f, \bm{\V}_n^\phi  \right\rangle_{\bpsi} \right|^2 \\
&\leq \frac{\m(\lam-\dpt)}{\dpt(\lam-\m)} e^{-2\lam t} \sum_{n=1}^\infty \left| \left\langle f, \bm{\V}_n^\phi  \right\rangle_{\bpsi} \right|^2 = \frac{\m(\lam-\dpt)}{\dpt(\lam-\m)} e^{-2\lam t} \sum_{n=1}^\infty \left| \left\langle f-\bpsi f, \bm{\V}_n^\phi  \right\rangle_{\bpsi} \right|^2 \\
&\leq \frac{\m(\lam-\dpt)}{\dpt(\lam-\m)} e^{-2\lam t} \norm{f-\bpsi f}_\bpsi^2.
\end{align*}
To justify this we start by observing that the first inequality follows from \eqref{eq:J-t-minus} together with $(\mathfrak{c}_n(\dz)\P_n^\phi)_{n \geq 0}$ being a Bessel sequence with Bessel bound 1, which was proved in \Cref{prop:riesz}. Next we use the fact that $\V_n^\phi$ is an eigenfunction for $\J_t^*$ associated to the eigenvalue $e^{-\lambda_n t}$, and then the identity
\begin{equation*}
\mathfrak{c}_n(\vat)\mathfrak{c}_n(\dz) = \mathfrak{c}_n(\dpt),
\end{equation*}
which follows by considering the cases $\vat = 1$ and $\vat < 1$ separately. Indeed, when $\vat = 1$ then $\dpt = \dz$ and $\mathfrak{c}_n^2(\vat) = 1$, while otherwise $\dz = 1$ so that $\dpt = \vat$ and $\mathfrak{c}_n^2(\dz) = 1$. The second inequality follows from \eqref{eq:sup-t} and then we use the biorthogonality of $(\P_n^\phi)_{n \geq 0}$ and $(\V_n^\phi)_{n \geq 0}$, given by Proposition \ref{prop:riesz}, which implies that for any $c \in \R$, $\langle c\bm{1}_{[0,1]}, \bm{\V}_n^\phi  \rangle_{\bpsi} = 0$ if $n \neq 0$. The last inequality follows from the fact that $(\bm{\V}_n^\phi  )_{n \geq 0}$ is a Bessel sequence with Bessel bound 1, again due to Proposition \ref{prop:riesz}. Next, if $0 \leq 2 \lam t < \log\left(\frac{(1+\m)(1+\lam-\dpt)}{(1+\dpt)(1+\lam-\m)}\right)$ and since $\m > \dpt$, we get
\begin{equation*}
\frac{\m(\lam-\dpt)}{\dpt(\lam-\m)} e^{-2\lam t} \geq \frac{\m}{\m+1} \frac{\dpt+1}{\dpt}   \frac{\lam-\dpt}{\lam-\dpt+1} \frac{\lam-\m+1}{\lam-\m} \geq 1,
\end{equation*}
so that the contractivity of the semigroup $\J$ yields, for $f \in \Leb^2(\bpsi)$ and any $t > 0$,
\begin{equation*}
\norm{\J_t f - \bpsi f}_\bpsi^2 \leq e^{-2\lam t} \norm{f-\bpsi f}_\bpsi^2.
\end{equation*}
Finally, since $\bpsi$ is an invariant probability measure,
\begin{equation*}
||\J_t f - \bpsi f||_\bpsi^2 = \bpsi[(\J_tf-\bpsi f)^2] = \bpsi [ (\J_tf)^2 ] - 2 \bpsi [f] \bpsi [ \J_t f] + (\bpsi [f])^2 = \bpsi [(\J_tf)^2] - (\bpsi [f])^2 = \Var_\bpsi(\J_tf),
\end{equation*}
which completes the proof. \qed

%
%

\subsection{Proof of \Cref{thm:convergence-equilibrium}.\eqref{item-2:thm:convergence-equilibrium}}

We first give a result that strengthens the intertwining relations in Proposition \ref{prop:intertwining} and falls into the framework of the work by Miclo and Patie \cite{Patie-Miclo}. Write $\mathrm{V}_{\dpt}$ for the Markov multiplicative kernel associated to a random variable with law $\gab[\dpt]$, which, by the same arguments as in the proof of Lemma \ref{lem:Lambda-L2-bounded}, satisfies $\mathrm{V}_{\dpt} \in \Bop{\Leb^2(\gab[\dpt])}{\Leb^2(\gab[\m])}$. We write $\overline{\rm{V}}_\phi = \Lambda_{\phi_{\dz}} \mathrm{V}_{\dpt}^*$ and, for $\mo \geq 1+\h$, let $\widetilde{\rm{V}}_\phi = \mathrm{V}_{\phm}$ and otherwise let $\widetilde{\rm{V}}_\phi = \mathrm{V}_{\phm} \rm{U}_\varphi$. Recall that a function $F:\R_+ \to [0,\infty)$ is said to be completely monotone if $F \in C^\infty(\R_+)$ and $(-1)^n \frac{d^n}{dx^n} F(u) \geq 0$, for $u > 0$ and $n \in \N$. By Bernstein's theorem, any completely monotone function $F$ is the Laplace transform of a positive measure on $[0,\infty)$, and if $\lim_{u \to 0} F(u) < \infty$ (resp.~$\lim_{u \to 0} F(u) =1$) then $F$ is the Laplace transform of finite (resp.~probability) measure on $\R_+$, see e.g.~\cite[Chapter 1]{SchillingSongVondracek10}.

\begin{proposition}
\label{prop:cmir}
Under the assumptions of the theorem, we have an interweaving relationship between $\J$ and $\Q[\m]$, in the sense of \cite{Patie-Miclo}, that is for $t \geq 0$ and on the respective $\Leb^2$-spaces
\begin{equation}
\label{eq:cmir}
\J_t^\phi \overline{\rm{V}}_\phi = \overline{\rm{V}}_\phi  \Q[\m]_t \quad \text{and} \quad \widetilde{\rm{V}}_\phi \J_t^\phi  = \Q[\m]_t \widetilde{\rm{V}}_\phi \quad \text{with} \quad \widetilde{\rm{V}}_\phi \overline{\rm{V}}_\phi   = F_\phi(-\calJd_\m)
\end{equation}
where $-\log F_\phi$ is a Bernstein function with $F_\phi:[0,\infty) \to [0,\infty)$ being the completely monotone function given by
\begin{equation*}
F_\phi(u) = \frac{(\dpt)_{\rho(u)}}{(\m)_{\rho(u)}} \frac{(\lam-\m)_{\rho(u)}}{(\lam-\dpt)_{\rho(u)}} , \quad u \geq 0.
\end{equation*}
\end{proposition}

\begin{proof}
We give the proof only in the case $\mo \geq 1+\h$, so that $\dpt = \dz$, as the other case follows by similar arguments. From Proposition \ref{prop:intertwining} we get, with $\calJ = \calJd_{\dz}$,
\begin{equation*}
\Q[\m]_t \mathrm{V}_{\dz} = \mathrm{V}_{\dz} \Q[\dz]_t,
\end{equation*}
and taking the adjoint and using that both $\Q[\m]$ and $\Q[\dz]$ are self-adjoint on $\Leb^2(\gab[\m])$ and $\Leb^2(\gab[\dz])$, respectively, we get that
\begin{equation*}
\Q[\dz]_t \mathrm{V}_{\dz}^* = \mathrm{V}_{\dz}^* \Q[\m]_t.
\end{equation*}
Combining this with the first intertwining relation in Proposition \ref{prop:intertwining} then yields
\begin{equation*}
\J_t \overline{\mathrm{V}}_{\phi} = \overline{\mathrm{V}}_{\phi} \Q[\m]_t,
\end{equation*}
and, together with second intertwining relation in Proposition\ref{prop:inter_P}, we conclude that
\begin{equation}
\label{eq:Q-comm}
\Q[\m]_t \widetilde{\mathrm{V}}_{\phi}\overline{\mathrm{V}}_{\phi} =  \widetilde{\mathrm{V}}_{\phi}  \J_t \overline{\mathrm{V}}_{\phi} = \widetilde{\mathrm{V}}_{\phi}\overline{\mathrm{V}}_{\phi} \Q[\m]_t.
\end{equation}
As $\Q[\m]_t$ is self-adjoint with simple spectrum the commutation identity \eqref{eq:Q-comm} implies, by the Borel functional calculus, see e.g.~\cite{Rudin1973}, that $\widetilde{\mathrm{V}}_{\phi}\overline{\mathrm{V}}_{\phi} = F(\calJd_{\m})$ for some bounded Borel function $F$, and to identify $F$ it suffices to identify the spectrum of $\widetilde{\mathrm{V}}_{\phi}\overline{\mathrm{V}}_{\phi}$. To this end we observe that, for any $g \in \Leb^2(\gab[\dz])$,
\begin{align*}
\langle \mathrm{V}_{\dz}^* \P_n^{(\m)}, g \rangle_{\gab[\dz]} &= \langle \P_n^{(\m)}, \mathrm{V}_{\dz} g \rangle_{\gab[\m]} = \sum_{m=0}^\infty  \langle g, \P_m^{(\dz)} \rangle_{\gab[\dz]} \langle \P_n^{(\m)}, \mathrm{V}_{\dz} \P_m^{(\dz)} \rangle_{\gab[\m]} \\ & = \frac{\mathfrak{c}_n(\dz)}{\mathfrak{c}_n(\m)} \langle \P_n^{(\dz)},g \rangle_{\gab[\dz]}
\end{align*}
where we used that $(\P_n^{(\dz)})_{n \geq 0}$ forms an orthonormal basis for $\Leb^2(\gab[\dz])$ and the identity $\mathrm{V}_{\dz} \P_m^{(\dz)} = {\mathfrak{c}_m(\dz)} \P_m^{(\m)} / {\mathfrak{c}_m(\m)}$ follows by a straightforward, albeit tedious, computation. Consequently, for any $n \in \N$,
\begin{equation*}
\widetilde{\mathrm{V}}_{\phi}\overline{\mathrm{V}}_{\phi}  \P_n^{(\m)} = \frac{\mathfrak{c}_n(\dz)}{\mathfrak{c}_n(\m)} \mathrm{V}_{\phm} \Lambda_{\phi_{\dz}} \P_n^{(\dz)} = \frac{\mathfrak{c}_n^2(\dz)}{\mathfrak{c}_n(\m)} \mathrm{V}_{\phm} \P_n^\phi = \frac{\mathfrak{c}_n^2(\dz)}{\mathfrak{c}_n^2(\m)} \P_n^{(\m)},
\end{equation*}
where the second and third equalities follow from calculations that were detailed in the proof of Proposition \ref{prop:riesz}. Using the definition of $\mathfrak{c}_n$ in \eqref{eq:defn-cn} we thus get that, for $n \in \N$,
\begin{equation*}
F(\lambda_n) = \frac{\mathfrak{c}_n^2(\dz)}{\mathfrak{c}_n^2(\m)} = \frac{(\dz)_n}{(\m)_n} \frac{(\lam-\m)_n}{(\lam-\dz)_n}
\end{equation*}
recalling from \eqref{eq:def-lambda-n} that $(\lambda_n)_{n \geq 0}$ are the eigenvalues of $-\calJd_{\m}$, which proves that $F_\phi = F$. Next, one readily computes that the non-negative inverse of the mapping $n \mapsto \lambda_n$ is given by the function $\rho$ defined prior to the statement of the theorem, which was remarked to be a Bernstein function. For another short proof of this fact, observe that, for $u \geq 0$,
\begin{equation*}
\rho'(u) = \left((\lam-1)^2+4u\right)^{-\frac{1}{2}},
\end{equation*}
which is completely monotone. Since $u \mapsto F_\phi(u^2+(\lam-1)u)$ is the Laplace transform of the product convolution of the beta distributions $\gab[\dz]$ and $\gab[\m]$ we may invoke \cite[Theorem 3.7]{SchillingSongVondracek10} to conclude $F_\phi$ is completely monotone. Finally, to show that $-\log F_\phi$ is a Bernstein function we note that, for any $a, b > 0$, the function $u \mapsto \log (a+b)_u - \log (a)_u$ is a Bernstein function, see e.g.~Example 88 in \cite[Chapter 16]{SchillingSongVondracek10}. Since
\begin{equation*}
-\log F_\phi(u) = \log \frac{(\m)_{\rho(u)}}{(\dz)_{\rho(u)}} + \log \frac{(\lam-\dz)_{\rho(u)}}{(\lam-\m)_{\rho(u)}},
\end{equation*}
with $\dz < \m$, and the composition of Bernstein functions remains Bernstein together with the fact that the set of Bernstein functions is a convex cone, see e.g.~\cite[Corollary 3.8]{SchillingSongVondracek10} for both of these claims, it follows that $-\log F_\phi$ is a Bernstein function.
\end{proof}


\begin{proof}[Proof of \Cref{thm:convergence-equilibrium}.\eqref{item-2:thm:convergence-equilibrium}]
Since $\m \in (\bm{1}_{\{\mo < 1+\h\}} + \mo, \lam)$ we may apply Proposition \ref{prop:cmir} to conclude that $\widetilde{\mathrm{V}}_\phi \overline{\mathrm{V}}_\phi = F_\phi(-\calJd_{\m})$ and a straightforward substitution gives $\E\left[e^{-u\tau}\right] = F_\phi(u)$, $u \geq 0$, with $-\log F_\phi$ a Bernstein function. From the Borel functional calculus we get, since $\Q[\m]_t$ is self-adjoint on $\Leb^2(\gab[\m])$, that
\begin{equation*}
\Q[\m]_\tau = \int_0^\infty \Q[\m]_t \Prob(\tau \in dt) = \int_0^\infty e^{t\calJd_{\m}} \Prob(\tau \in dt) = F_\phi(-\calJd_{\m}) = \widetilde{\mathrm{V}}_\phi \overline{\mathrm{V}}_\phi.
\end{equation*}
Combining this identity with \eqref{eq:cmir} yields, for non-negative $f \in \Leb^2(\bpsi)$,
\begin{equation*}
\widetilde{\mathrm{V}}_\phi \overline{\mathrm{V}}_\phi \widetilde{\mathrm{V}}_\phi f = \int_0^\infty \Q[\m]_t  \widetilde{\mathrm{V}}_\phi f \:  \Prob(\tau \in dt) = \int_0^\infty  \widetilde{\mathrm{V}}_\phi \J_t f \: \Prob(\tau \in dt) =  \widetilde{\mathrm{V}}_\phi \int_0^\infty \J_t f \: \Prob(\tau \in dt),
\end{equation*}
and the general case follows by linearity and by decomposing $f$ into the difference of non-negative functions. By Proposition \ref{prop:quasi-similar}, $\widetilde{\mathrm{V}}_\phi$ has trivial kernel on $\Leb^2(\bpsi)$. So we deduce
\begin{equation}
\label{eq:J-tau}
\overline{\mathrm{V}}_\phi \widetilde{\mathrm{V}}_\phi = \int_0^\infty \J_t \Prob(\tau \in dt) = \J_\tau,
\end{equation}
and thus $\J$ satisfies an interweaving relation with $\Q[\m]$, in the sense of \cite{Patie-Miclo}. Consequently we may invoke \cite[Theorems 7, 24]{Patie-Miclo} to transfer the entropy decay and $\Phi$-entropy decay of $\Q[\m]$, reviewed in \Cref{sec:classical-Jacobi}, to the semigroup $\J$ but after a time shift of the independent random variable $\tau$. Note that, when $\lam > 2(\bm{1}_{\{\mo < 1+\h\}} + \mo)$, we may take $\m = \frac{\lam}{2}$ so that the reference semigroup is $\Q[\lam/2]$, which has optimal entropy decay rate.
\end{proof}

\subsection{Proof of \Cref{thm:contractivity}}

The proof of \Cref{thm:contractivity}\eqref{item-1:thm:contractivity} follows by using \Cref{eq:J-tau} above to invoke \cite[Theorem 8]{Patie-Miclo}. Next, by \Cref{eq:J-tau} and using Proposition \ref{prop:cmir}, we get
\begin{equation*}
||\J_{t+\tau}||_{1\to\infty} = ||\J_t \overline{\mathrm{V}}_\phi \widetilde{\mathrm{V}}_\phi||_{1\to\infty} = ||\overline{\mathrm{V}}_\phi \Q[\m]_t \widetilde{\mathrm{V}}_\phi||_{1\to\infty} \leq ||\Q[\m]_t||_{1\to\infty}
\end{equation*}
where the last inequality follows by applying Lemma \ref{lem:Lambda-L2-bounded} twice, once in the case $p=\infty$ for $\overline{\mathrm{V}}_\phi$ and once with $p=1$ for $\widetilde{\mathrm{V}}_\phi$. The claim now follows from the corresponding ultracontractivity estimate for $\Q[\m]$. \qed

\subsection{Proof of \Cref{cor:subordinated}}

The following arguments are taken from the proof of \cite[Proposition 5]{Gateway}. We denote by $\Q[\m,\bm{\tau}]$ for the classical Jacobi semigroup $\Q[\m]$ subordinated with respect to $\bm{\tau} = (\tau_t)_{t \geq 0}$. By \cite[Theorem 3]{Patie-Miclo} we obtain, from Proposition \ref{prop:cmir}, an  interweaving relationship between the subordinate semigroups, i.e.~writing $\overline{\mathrm{V}}_\phi$ and $\widetilde{\mathrm{V}}_\phi$ as above, we have, for any $t \geq 0$ and on the appropriate $\Leb^2$-spaces,
\begin{equation}
\label{eq:cmir-subordinated}
\J_t^{\bm{\tau}} \overline{\rm{V}}_\phi = \overline{\rm{V}}_\phi  \Q[\m,\bm{\tau}]_t \quad \text{and} \quad \widetilde{\rm{V}}_\phi \J_t^{\bm{\tau}}  = \Q[\m,\bm{\tau}]_t \widetilde{\rm{V}}_\phi \quad \text{with} \quad \overline{\rm{V}}_\phi\widetilde{\rm{V}}_\phi    = \J_1^{\bm{\tau}}.
\end{equation}
Using this we get, for any $f \in \Leb^2(\bpsi)$ and $t \geq 1$,
\begin{align*}
\J_t^{\bm{\tau}} f = \J_{t-1}^{\bm{\tau}} \overline{\mathrm{V}}_\phi \widetilde{\mathrm{V}}_\phi f = \overline{\mathrm{V}}_\phi  \Q[\m,\bm{\tau}]_{t-1} \widetilde{\mathrm{V}}_\phi f &= \sum_{n=0}^\infty \E\left[e^{-\lambda_n \tau_{t-1}}\right] \langle \widetilde{\mathrm{V}}_\phi f, \P_n^{(\m)} \rangle_{\gab[\m]} \overline{\mathrm{V}}_\phi \P_n^{(\m)} \\
&= \sum_{n=0}^\infty \E\left[e^{-\lambda_n \tau_{t-1}}\right]  \frac{\mathfrak{c}_n^2(\dpt)}{\mathfrak{c}_n^2(\m)} \langle f, \V_n^\phi \rangle_{\bpsi} \P_n^\phi \\
&= \sum_{n=0}^\infty \E\left[e^{-\lambda_n \tau_t}\right]  \langle f, \V_n^\phi \rangle_{\gab[\m]} \P_n^\phi
\end{align*}
where in the second equality we used the boundedness of $\overline{\mathrm{V}}_\phi$ together the expansion for the subordinated classical Jacobi semigroup which follows from \eqref{Q_sem} and standard arguments, then the properties of $\widetilde{\mathrm{V}}_\phi$ and $\overline{\mathrm{V}}_\phi$ detailed in previous sections, and finally the expression for $\E[e^{-\lambda_n \tau}]$ in \eqref{eq:def-T-Laplace}. All of the claims, save for the last one, then follow from \cite[Theorems 7, 24]{Patie-Miclo} applied to \eqref{eq:cmir-subordinated}. Next, we establish the ultracontractive bound $||\J_t^{\bm{\tau}}||_{1\to\infty} \leq c_\m(\E[\tau^{-p}]+1)$ for $t > 2$. From \eqref{eq:def-T-Laplace} we get, by applying Stirling's formula for the gamma function together with $\lim_{u \to \infty} u^{-1/2}\rho(u) = 1$, that $\lim_{u \to \infty} u^{(\m-\dpt)} \E[e^{-u \tau}] = 1$. Writing for convenience $p = \frac{\lam-\m}{\lam-\m-1} > 0$, we get by assumption on the parameters that $p < \m-\dpt$ so that the previous asymptotic yields, for $t \geq 1$,
\begin{equation*}
\E[\tau_t^{-p}] = \frac{1}{\Gamma(p)} \int_0^\infty \E[e^{-u \tau_t}] u^{p-1} du \leq \frac{1}{\Gamma(p)} \int_0^\infty \E[e^{-u \tau}] u^{p-1} du = \E[\tau^{-p}] < \infty
\end{equation*}
where the two equalities follow by applying Tonelli's theorem together with a change of variables, and the inequality follows from the fact that, for all $u\geq0$, $t \mapsto \E[e^{-u \tau_t}]$ is non-increasing, recalling the notation $\tau_1\stackrel{(d)}{=}\tau$. Hence, from the ultracontractive bound $||\Q[\m]_s||_{1 \to \infty} \leq c_\m \max(1,s^{-p})$, valid for all $s > 0$, we deduce that for $t \geq 1$
\begin{align*}
||\Q[\m,\bm{\tau}]_t||_{1 \to \infty} \leq \int_0^\infty ||\Q[\m]_s||_{1 \to \infty} \Prob(\tau_t \in ds) \leq c_\m \left(\int_0^1 s^{-p} \Prob(\tau_t \in ds) +  \int_1^\infty \Prob(\tau_t \in ds) \right) \leq c_\m(\E[\tau^{-p}]+1).
\end{align*}
Consequently from \eqref{eq:cmir-subordinated} we get that, for $t > 2$,
\begin{equation*}
||\J_t^{\bm{\tau}}||_{1\to\infty} = ||\J_{t-1}^{\bm{\tau}} \overline{\mathrm{V}}_\phi \widetilde{\mathrm{V}}_\phi||_{1\to\infty} = ||\overline{\mathrm{V}}_\phi \Q[\m,\bm{\tau}]_{t-1} \widetilde{\mathrm{V}}_\phi||_{1\to\infty} \leq ||\Q[\m,\bm{\tau}]_{t-1}||_{1 \to \infty} \leq  c_\m(\E[\tau^{-p}]+1).
\end{equation*}
Then it is easy to complete the proof of the last claim by following similar arguments as in the proof of \cite[Proposition 6.3.4]{Bakry_Book}, noting that the required variance decay estimate therein, namely
\begin{equation*}
\Var_\bpsi(\J_t^{\bm{\tau}} f) \leq \left(\frac{\m(\lam-\dpt)}{\dpt(\lam-\m)}\right)^{1-2t} \Var_\bpsi(f)
\end{equation*}
valid for all $t \geq 0$ and $f \in \Leb^2(\bpsi)$, follows trivially from \Cref{thm:convergence-equilibrium}.\eqref{item-1:thm:convergence-equilibrium} via subordination. \qed

\section{Examples} \label{sec:examples}


In this section we consider a parametric family of non-local Jacobi operators for which $h$ is a power function. More specifically, let $\mm \geq 1$ and consider the integro-differential operator $\calJ_\mm$ given by
\begin{align*}
\calJ_{\mm} f(x) &= x(1-x)f''(x) - (\lam x - \mm-1)f'(x)- x^{-(\mm+1)}\int_0^{x} f'(r)r^{\mm}dr. 
\end{align*}
Then $\calJ_\mm$ is a non-local Jacobi operator with $\mo = \mm+1$ and $h(r) = r^{-\mm-1}$, $r>1$, or one easily gets that equivalently $\overline{\Pi}(r) = e^{-\mm r}$, $r > 0$. One readily computes that $\h = \int_1^\infty h(r)dr = \mm^{-1}$ and thus the condition $\mo \geq 1+\h$ is always satisfied, which implies that $\vat = 1$. Writing $\phi_\mm$ for the Bernstein function in one-to-one correspondence with $\calJ_\mm$, we have that for $u \geq 0$,
\begin{equation}
\label{eq:defvp}
\phi_{\mm}(u) = u + \frac{\mm^2-1}{\mm} + \int_1^\infty (1-r^{-u})r^{-\mm-1}dr = \frac{(u+\mm+1)(u+\mm-1)}{u + \mm}. 
\end{equation}
From the right-hand side of \eqref{eq:defvp} we easily see that $d_{\phi_\mm} = \mm-1$. Now, we assume that $\lam  >\mm+2 > 3$ and, for sake of simplicity, take $\lam - \mm \not\in \N$. The following result characterizes all the spectral objects for these non-local Jacobi operators.

\begin{proposition} $\mbox{}$\\[-5mm]
\label{prop:small-perturbation}
\begin{enumerate}[\rm(i)]
\item \label{item-1:prop:small-perturbation} The density of the unique invariant measure of the Markov semigroup associated to $\calJ_\mm$ is given by
\begin{equation*}
\beta(x) = \frac{((\lam - \mm - 2)x+1)}{(\mm+1)(1-x)} \gab[\mm](x),  \ x \in (0,1), 
\end{equation*}
where $\gab[\mm]$ is the density of a Beta random variable of parameter $\mm$ and $\lam$, see \eqref{eq:def_beta}.
\item \label{item-2:prop:small-perturbation} We have that $\P_0^{\phi_\mm} \equiv 1$ and, for $n \geq 1$,
\begin{equation*}
\P_n^{\phi_\mm}(x)= \frac{n!}{(\mm+2)_n} \sqrt{\C_n(1)} \left( \frac{\P_n^{(\lam,\mm+2 )} (x)}{\sqrt{\C_n(\mm+2)}}  + \frac{x}{\mm} \frac{\P_{n-1}^{(\lam+1,\mm+3)}(x)}{\sqrt{\widetilde{\C}_{n-1}(\mm+3)}} \right), \ x \in [0,1],
\end{equation*}
where, on the right-hand side, we made explicit the dependence on the two parameters for the classical Jacobi polynomials, i.e.~$\P_n^{(\lam,\mo)} (x)=\P_n^{(\mo)} (x)=\sqrt{\C_n(\mo)} \sum_{k=0}^n \frac{(-1)^{n+k}}{(n-k)!}  \frac{(\lam-1)_{n+k}}{(\lam-1)_{n\phantom{+k}}} \frac{(\mo)_n}{(\mo)_k} \frac{x^k}{k!}$, see \eqref{J_pol}, and where $\widetilde{\C}_{n}(\mm+3) = n!(2n+\lam)(\lam+1)_n / (\mm+3)_n(\lam-\mm-2)_n$.
\item \label{item-3:prop:small-perturbation} For any $n \in \N$ the function $\V_n^{\phi_\mm}$ is given by
\begin{equation*}
\V_n^{\phi_\mm}(x) = \frac{w_n(x)}{\bpsi(x)}, \ x \in (0,1),
\end{equation*}
where $w_n$ has the so-called Barnes integral representation, see e.g.~\cite{barnes_integrals},  for any  $a>0$,
\begin{align*}
w_n(x)
&= -C_{\lam, \mm, n} \frac{1}{2 \pi i} \int_{-a - i \infty}^{-a + i \infty} \frac{\Gamma(\mm+2-z)\Gamma(-z)\Gamma(\mm-z)}{\Gamma(\mm+1-z)\Gamma(-n-z)\Gamma(z+\lam+n)} x^{z}dz, \\
&= C_{\lam, \mm, n} \frac{\sin(\pi(\mm-\lam))}{\pi} \sum_{k = 0}^{\infty} \frac{(\mm+1)_{k+n}}{(\mm+1)_{k\phantom{+n}}} \frac{\Gamma(k+\mm-n-\lam+1)}{k!}(k-1)  x^{k+\mm},  \quad |x|<1,
\end{align*}
and $C_{\lam, \mm, n} = \mm(\lam - 1) \Gamma(\lam+n-1) \sqrt{\C_n(1)} (-2)^n / (n!\Gamma(\mm+2))$.
\end{enumerate}
\end{proposition}

\begin{proof}
First, from~\eqref{eq:defvp} and \eqref{eq:product-Wphi} we get that, for any $n \in \N$,
\begin{equation}
\label{eq:small_pert_J}
W_{\phi_{\mm}}(n+1) = \frac{\mm}{n+\mm} (\mm+2)_n, 
\end{equation}
so that from~\eqref{eq:mom-bpn} we deduce that
\begin{equation}
\label{eq:mom-bb}
\bpsi [p_n] = \frac{W_{\phi_{\mm}}(n+1)}{(\lam)_n} = \frac{\mm}{n+\mm} \frac{(\mm+2)_n}{(\lam)_n}.
\end{equation}
The first term on the right of \eqref{eq:mom-bb} is the $n^{th}$-moment of the probability density $f_\mm(x) = \mm x^{\mm-1}$ on $[0,1]$ while the second term is the $n^{th}$-moment of a $\gab[\mm+2]$ density. Thus, by moment identification and determinacy, we conclude that $\bpsi(x) = f_\mm \diamond \gab[\mm+2](x)$ and after some easy algebra  we get, for $x \in (0,1)$, that
\begin{align*}
\beta(x) &= 
\frac{\Gamma(\lam)\mm x^{\mm-1}}{\Gamma(\mm+2)\Gamma(\lam-\mm-2)}   \int_x^1 y (1-y)^{\lambda_1 - \mm - 3} dy 
= \frac{((\lam - \mm - 2)x+1)}{(\mm+1)(1-x)} \gab[\mm](x),
\end{align*}
which completes the proof of the first item. Next, substituting \eqref{eq:small_pert_J} in~\eqref{eq:Jpolpsi}, gives $\P_0^{\mm} \equiv 1$, and for $n = 1,2,\ldots$,
\begin{align*}
\P_n^{\phi_\mm}(x) &= \sqrt{\C_n(1)}\left( \sum_{k=0}^n \frac{(-1)^{n+k}}{(n-k)!} \frac{(\lam-1)_{n+k}}{(\lam-1)_{n\phantom{+k}}} \frac{n!}{k!} \frac{x^k}{(\mm+2)_k}  + \sum_{k=0}^n \frac{(-1)^{n+k}}{(n-k)!} \frac{(\lam-1)_{n+k}}{(\lam-1)_{n\phantom{+k}}}  \frac{n!}{k!}\frac{k}{\mm} \frac{x^k}{(\mm+2)_k} \right) \\
&= \frac{n!}{(\mm+2)_n} \sqrt{\C_n(1)} \left( \frac{\P_n^{(\mm+2 )} (x)}{\sqrt{\C_n(\mm+2)}}  + \frac{x}{\mm} \frac{\P_{n-1}^{(\lam+1,\mm+3)}(x)}{\sqrt{\widetilde{\C}_{n-1}(\mm+3)}} \right)
\end{align*}
where, to compute the second equality we made a change of variables and used the recurrence relation of the gamma function, and the definition of the classical Jacobi polynomials, see Section~\ref{sec:classical-Jacobi} and also~\cite{orthogonal_pols}. This completes the proof of \eqref{item-2:prop:small-perturbation}. To prove \eqref{item-3:prop:small-perturbation} we recall from~\eqref{eq:v-n-rodrigues} that, for any $n \in \N$, $\V_n^{\phi_\mm}(x)= \frac{1}{\beta(x)}w_n(x)$, where, by~\eqref{eq:Mellin_w_n}, the Mellin transform of $w_n$ is given, for any $\Re(z)>0$, as
\begin{equation*}
\mathcal{M}_{w_n}(z) = C_{\lam, \mm, n} (z+\mm+1)\frac{\Gamma(z)}{\Gamma(z-n)} \frac{\Gamma(z+\mm)}{\Gamma(z+\lam+n)},
\end{equation*}
used twice the functional equation for the gamma function and the definition of the constant $C_{\lam, \mm, n}$ in the statement. Next, writing $z=a+ib$ for any $b \in \R$ and $a>0$, we recall from~\eqref{eq:classical-Gamma} that there exists a constant $C_a > 0$ such that
\begin{equation}
\label{eq:G-G}
\lim_{|b| \to \infty} C_a |b|^{\lam + n-1} \left|(z+\mm+1) \frac{\Gamma(z)}{\Gamma(z-n)} \frac{\Gamma(z+\mm)}{\Gamma(z+\lam+n)} \right| = 1
\end{equation}
where we recall that $\lam > \mm+2 >3$ and $n \geq 0$. Hence, since $z \mapsto \mathcal{M}_{w_n}(z)$ is analytic on the right half-plane, by Mellin's inversion formula, see e.g.~\cite[Chapter 11]{Misra-Lavoine}, one gets for any $a>0$,
\begin{equation*}
w_n(x) = C_{\lam, \mm, n} \frac{1}{2 \pi i} \int_{a - i \infty}^{a + i \infty} (z+\mm+1)\frac{\Gamma(z)}{\Gamma(z-n)} \frac{\Gamma(z+\mm)}{\Gamma(z+\lam+n)} x^{-z} dz
\end{equation*}
where the integral is absolutely convergent for any $x>0$. Note that this is a Barnes-integral since we can write, again using the functional equation for the gamma function,
\begin{align*}
w_n(x) = -C_{\lam, \mm, n}\frac{1}{2 \pi i} \int_{-a - i \infty}^{-a + i \infty} \frac{\Gamma(\mm+2-z)}{\Gamma(\mm+1-z)} \frac{\Gamma(-z)}{\Gamma(-z-n)} \frac{\Gamma(\mm-z)}{\Gamma(z+\lam+n)} x^{z}dz,
\end{align*}
see for instance~\cite{barnes_integrals}. Next, since $(z+\mm+1)\frac{\Gamma(z)}{\Gamma(z-n)} = (z+\mm+1)(z-n)\cdots(z-1)$, it follows that the function $z \mapsto (z+\mm+1)\frac{\Gamma(z)}{\Gamma(z-n)}$ does not have any poles, while the function $z \mapsto  \frac{\Gamma(z+\mm)}{\Gamma(z+\lam+n)}$ has simple poles at $z = -k-\mm$ for all $k \in \N$. Consequently, by Cauchy's residue theorem we have, for any $|x|<1$,
\begin{equation*}
w_n(x) = C_{\lam, \mm, n} \sum_{k = 0}^{\infty} \frac{(1-k)\Gamma(-k-\mm)}{\Gamma(-k-\mm-n)} \frac{(-1)^k}{k!}\frac{x^{k+\mm}}{ \Gamma(-k-\mm+\lam+n)}
\end{equation*}
where we used that the integrals along the two horizontal segments of any closed contour vanish, as by ~\eqref{eq:G-G} they go to $0$ when $|b|\rightarrow \infty$. We justify the radius of convergence of the series as follows. Since $\lam - \mm \not\in \N$, using Euler's reflection formula for the gamma function, i.e.~$\Gamma(z)\Gamma(1-z) = \frac{\pi}{\sin(\pi z)}$, $z \not\in \mathbb{Z}$, we conclude that
\begin{align*}
w_n(x) = C_{\lam, \mm, n} \frac{\sin(\pi(\mm-\lam))}{\pi} \sum_{k = 0}^{\infty} \frac{(\mm+1)_{k+n}}{(\mm+1)_{k\phantom{+n}}} \frac{\Gamma(k+\mm-n-\lam+1)}{k!}(k-1)  x^{k+\mm}
\end{align*}
where we used that $\sin(x+k\pi) = (-1)^k\sin(x)$ for $k \in \N$. Using the recurrence relation of the gamma function we deduce that the radius of convergence of this series is $1$, which completes the proof.
\end{proof}

\begin{appendix}
\section{Classical Jacobi operator and semigroup}
\label{sec:classical-Jacobi}

\subsection{Introduction and boundary classification}

Before we begin reviewing the classical Jacobi operator, semigroup and process, we clarify the notational conventions that we use for these objects throughout the paper. Namely, with $\lam$ being fixed, instead of writing $\calJd_{\lam,\mo}$ we suppress the dependency on $\lam$ and simply write $\calJd_\mo$, and similarly for the beta distribution, Jacobi semigroup and polynomials. The exception is when any of these objects depend in a not-straightforward way on $\lam$, in which case we highlight the dependency explicitly. Now, fix constants $\lam > \mo > 0$ and let $\Q=(\Q_t)_{t \geq 0}$ be the classical Jacobi semigroup whose c\`adl\`ag realization is the Jacobi process $(Y_t)_{t \geq 0}$ with values in $[0,1]$, that is, for bounded measurable functions
$f \colon [0,1] \to \R$,
\begin{equation*}
\Q_t f(x) = \E_x \left[f(Y_t)\right], \quad x \in [0,1].
\end{equation*}
Then $\Q$ is a Feller semigroup with infinitesimal generator $\calJd_\mo$, given, for $f\in C^2[0,1]$, by
\begin{equation*}
\calJd_\mo f(x) =  x(1-x)f''(x)-(\lam x -\mo)f'(x), \quad x\in [0,1].
\end{equation*}
Note that if the state space of the Jacobi process is taken to be $[-1,1]$, then the associated infinitesmal generator $\widetilde{\mathbf{J}}_\mo$ is
\begin{equation*}
\widetilde{\mathbf{J}}_\mo f(x) = (1-x^2)f''(x)+(2\mo-\lam-\lam x )f'(x),
\end{equation*}
and setting $g(x)= (x+1)/2$ yields
\begin{equation*}
\widetilde{\mathbf{J}}_\mo (f \circ g) (g^{-1}(x)) = x(1-x)f''(x)-(\lam x-\mo)f'(x) = \calJd_\mo f(x).
\end{equation*}
Since the operator $\calJd_\mo$ is degenerate at the boundaries of $[0,1]$, it is important to specify how the process behaves at these points. After some straightforward computations, as outlined in \cite[Chapter 2]{Borodin2002} and using the notation therein, we get that the boundaries are classified as follows
\begin{equation*}
    0 \text{ is }
\begin{cases}
    \text{exit-not-entrance for} & \mo \leq 0,\\
    \text{regular for} & 0< \mo < 1,\\
    \text{entrance-not-exit for} & \mo \geq 1,
\end{cases} \quad \text{and} \quad
     1 \text{ is }
\begin{cases}
    \text{exit-not-entrance for} & \lam \leq \mo,\\
    \text{regular for} & \mo< \lam < 1+\mo,\\
    \text{entrance-not-exit for} & \lam \geq 1+\mo.
\end{cases}
\end{equation*}
Therefore, our assumptions \Cref{assA} on $\lam$ and $\mo$ guarantee that both $0$ and $1$ are at least entrance, and may be regular or entrance-not-exit depending on the particular values of $\lam$ and $\mo$. Let us write $\mathcal{D}_F(\calJd_\mo)$ for the domain of the generator $\calJd_\mo$ of the Feller semigroup. To specify it we recall that the so-called scale function $s$ of $\calJd_\mo$ satisfies
\begin{equation*}
s'(x)= x^{-\mo}(1-x)^{-(\lam-\mo)}, \quad x \in (0,1).
\end{equation*}
Let $f^+$ and $f^-$ denote the right and left derivatives of a function $f$ with respect to $s$, i.e.
\begin{equation*}
f^+(x) =\lim_{h \downarrow 0}\frac{f(x+h)-f(x)}{s(x+h)-s(x)} \quad \text{and} \quad f^-(x) =\lim_{h \downarrow 0}\frac{f(x)-f(x-h)}{s(x)-s(x-h)}.
\end{equation*}
Then,
\begin{equation}
\label{eq:C-domain}
\mathcal{D}_{F}(\calJd_\mo) = \left\lbrace f \in C^2[0,1] : f^+(0^+) = f^-(1^-) = 0 \right\rbrace.
\end{equation}
In particular, $\Poly \subseteq \mathcal{D}_{F}(\calJd_\mo)$, since for any $f \in \Poly$ we have
\begin{equation*}
f^+(0^+) = \lim_{x \downarrow 0} x^{\mo} f'(x) = 0 \quad \text{and} \quad f^-(1^-) = \lim_{x \uparrow 1} (1-x)^{\lam-\mo} f'(x) = 0.
\end{equation*}
From the boundary conditions in \eqref{eq:C-domain} we get that if any point in $\{0,1\}$ is regular then it is necessarily a reflecting boundary for the Jacobi process with $\lam > \mo > 0$.

\subsection{Invariant measure and $\Leb^2$-properties} \label{subsec:inv-measure}


The classical Jacobi semigroup $\Q = (\Q_t)_{t \geq 0}$ has a unique invariant measure $\gab[\mo]$, which is the
following beta distribution on $[0,1]$:
\begin{equation} \label{eq:def_beta}
\gab[\mo](dx) = \gab[\mo](x)dx =\frac{\Gamma(\lam)}{\Gamma(\mo)\Gamma(\lam-\mo)}x^{\mo-1}(1-x)^{\lam-\mo- 1}dx, \quad x \in (0,1),
\end{equation}
and we recall that, for any $n\in \mathbb{N}$,
\begin{equation}
\label{eq:beta-moments}
\gab[\mo][p_n] = \int_0^1 x^n\gab[\mo](dx)= \frac{(\mo)_n}{(\lam)_n}.
\end{equation}
Since $\gab[\mo]$ is invariant for $\Q$, we get that $\Q$ extends to a contraction semigroup on $\Leb^2(\gab[\mo])$ and, moreover, the stochastic continuity of $Y$ ensures that this extension is strongly continuous in $\Leb^2(\gab[\mo])$. Consequently, we obtain a Markov semigroup on $\Leb^2(\gab[\mo])$, which we still denote by $\Q = (\Q_t)_{t \geq 0}$. The eigenfunctions of $\calJd_\mo$ are the Jacobi polynomials given, for any $n \in \N$ and $x \in [0,1]$, by
\begin{equation}
\label{J_pol}
\P_n^{(\mo)} (x) =  \sqrt{\C_n(\mo)} \sum_{k=0}^n \frac{(-1)^{n+k}}{(n-k)!}  \frac{(\lam-1)_{n+k}}{(\lam-1)_{n\phantom{+k}}} \frac{(\mo)_n}{(\mo)_k} \frac{x^k}{k!},
\end{equation}
where we denote
\begin{equation}
\label{eq:Cn-lam-mo}
\C_n(\mo)=(2n+\lam-1)\frac{n!(\lam)_{n-1}}{(\mo)_n(\lam-\mo)_n}.
\end{equation}
These polynomials are orthogonal with respect to the measure $\gab[\mo]$ and, by choice of $\C_n(\mo)$, satisfy the normalization condition
\begin{equation*}
\int_0^1 \P_n^{(\mo)}(x) \P_m^{(\mo)}(x) \gab[\mo](dx) = \langle \P_n^{(\mo)}, \P_m^{(\mo)}\rangle_{\gab[\mo]} = \delta_{nm}.
\end{equation*}
In particular, they form an orthonormal basis of $\Leb^2(\gab[\mo])$. They have the alternative representation
\begin{align}
\P_n^{(\mo)}(x) &= \frac{1}{n!} \sqrt{\C_n(\mo)} \frac{1}{\gab[\mo](x)} \frac{d^n}{dx^n}\left(x^n(1-x)^n \gab[\mo](x)\right) \nonumber \\
\label{eq:Jac-Rod}&= \frac{1}{\gab[\mo](x)} \frac{(\lam-\mo)_n}{(\lam)_n}  \sqrt{\C_n(\mo)} \Rc_{n} \gab[\lam+n,\mo](x)
\end{align}
where $\Rc_n$ are the Rodrigues operators given in \eqref{eq:def-Rodrigues} and
\begin{equation*}
\gab[\lam+n,\mu](x) = \frac{\Gamma(\lam+n)}{\Gamma(\mu)\Gamma(\lam+n-\mu)} x^{\mu-1}(1-x)^{\lam+n-\mu-1}.
\end{equation*}
All these relations follow by the change of variables $x \mapsto 2x-1$ and simple algebra from the corresponding relations for the polynomials $P_n^{(\mo-1,\lam-\mo-1)}$ defined in \cite[Chap.~4]{Ismail}, which are orthogonal with respect to the weight $(1-x)^{\mo-1}(1+x)^{\lam-\mo-1}$, and are also called Jacobi polynomials in the literature. Indeed, $\P_n^{(\mo)}$ and $P_n^{(\mo-1,\lam-\mo-1)}$ are related through
\begin{equation*}
\label{eq:def_pol}
\P_n^{(\mo)}(x) = (-1)^n \sqrt{\frac{(2n+\lam-1)n!(\lam)_{n-1}}{(\mo)_n(\lam-\mo)_n}} P_n^{(\mo-1,\lam-\mo-1)}(1-2x).
\end{equation*}
Next, the eigenvalue associated to the eigenfunction $\P_n^{(\mo)}(x)$ is, for $n \in \N$,
\begin{equation}
\label{Q_eigenv}
-\lambda_n =  -n^2-(\lam-1)n= -n(n-1)-\lam n.
\end{equation}
Observe that for $n=1$, \eqref{Q_eigenv} reduces to $-\lam$ and $\lambda_0=0$, so that $-\lam$ denotes the largest, non-zero eigenvalue of $\calJd_\mo$, which is also called the spectral gap. The semigroup $\Q$ then admits the spectral decomposition given, for any $f \in \Leb^2(\gab[\mo])$ and $t \geq 0$, by
\begin{equation}
\label{Q_sem}
\Q_t  = \sum_{n=0}^{\infty} e^{- \lambda_n t} \langle \: \cdot \:, \P^{(\mo)}_n \rangle_{\gab[\mo]} \P^{(\mo)}_n ,
\end{equation}
where the sum converges in the operator norm. The domain of $\calJd_\mo$, the generator of the Markov semigroup, in $\Leb^2(\gab[\mo])$ can then be identified as
\begin{equation*}
\calD_{\Leb^2}(\calJd_\mo) = \left\lbrace f \in \Leb^2(\gab[\mo]) : \sum_{n=0}^\infty n^4 \left|\langle f, \P^{(\mo)}_n \rangle_{\gab[\mo]}\right|^2 < \infty \right\rbrace.
\end{equation*}

\subsection{Variance and entropy decay, hypercontractivity and ultracontractivity} \label{subsec:variance-decay}

As mentioned in the introduction, the fact that $\Q$ has nice spectral properties and satisfies certain functional inequalities gives quantitative rates of convergence to the equilibrium measure $\gab[\mo]$. For instance, from \eqref{Q_sem} one gets the following variance decay estimate, valid for any $f \in \Leb^2(\gab[\mo])$ and $t \geq 0$,
\begin{equation*}
\Var_{\gab[\mo]}(\Q_t f ) \leq e^{-2\lam t} \Var_{\gab[\mo]}(f),
\end{equation*}
which may also be deduced directly from the Poincar\'e inequality for $\calJd_\mo$, see \cite[Chapter 4.2]{Bakry_Book}. This convergence is optimal in the sense that the decay rate does not hold for any constant greater than $2\lam$. Next, let us write $\logSob[\mo]$ for the log-Sobolev constant of $\calJd_\mo$ defined as
\begin{equation}
\label{eq:log-Sobolev-gen}
\logSob[\mo] = \inf_{f \in \calD_{\Leb^2}(\calJd_\mo)} \left\{  \frac{-4\gab[\mo] [f \calJd_\mo f]}{\Ent_{\gab[\mo]} (f^2)} :\Ent_{\gab[\mo]} (f^2) \neq 0 \right\}.
\end{equation}
Once always has $\logSob[\mo] \leq 2\lam$, and in the case of the symmetric Jacobi operator, i.e.
$\mo = \frac{\lam}{2} > 1$, one gets
\begin{equation}
\label{eq:log-Sobolev}
\logSob[\frac{\lam}{2}] = 2\lam,
\end{equation}
while otherwise $\logSob[\mo] < 2\lam$, see e.g.~\cite{fontenas:1998}, although the equality for the symmetric case goes back to \cite{pearson:1992}. As a consequence of \eqref{eq:log-Sobolev-gen} we have on the one hand, the convergence in entropy for any $t \geq 0$ and $f \in \Leb^1(\gab[\mo])$ such that $\Ent_{\gab[\mo]}(f) < \infty$,
\begin{equation}
\label{eq:entr}
\Ent_{\gab[\mo]} (\Q[\mo]_tf) \leq e^{-\logSob[\mo] t} \Ent_{\gab[\mo]} (f),
\end{equation}
and on the other hand, from Gross \cite{Gross}, the hypercontractivity estimate; that is, for all $t\geq 0$,
\begin{equation}
\label{eq:hyper}
||\Q[\mo]_t||_{2 \rightarrow q} \leq 1  \textrm{ for all $q$ satisfying }  2\leq q \leq 1+e^{ \logSob[\mo] t}.
\end{equation}
From \eqref{eq:log-Sobolev} we thus get that the symmetric Jacobi semigroup attains the optimal entropic decay and hypercontractivity rate. Moreover, if $\lam/2 = n \in \N$, there exists a homeomorphism between $\calJd_\mo$ and the radial part of the Laplace--Beltrami operator on the $n$-sphere, which leads to the curvature-dimension condition $CD(\lam-1,\lam)$; see \cite{Bakry_Book}. Thus for any function $\Phi \colon I \to \R$ satisfying the admissibility condition \eqref{eq:def-admissible},
one has
\begin{equation}
\label{eq:C-D-entropy}
\Ent_{\gab[\lam/2]}^\Phi(\Q[\lam/2]_t f) \leq e^{-(\lam-1)t} \Ent_{\gab[\lam/2]}^\Phi(f)
\end{equation}
for any $t \geq 0$ and $f:[0,1] \to I$ such that $f, \Phi(f) \in \Leb^1(\gab[\lam/2])$. If $\lam-\mo > 1$, the operator $\calJd_\mo$ also satisfies a Sobolev inequality, see e.g.~\cite{Bakry-Jacobi}, and thus one obtains from \cite[Theorem 6.3.1]{Bakry_Book} that, for $0 < t \leq 1$,
\begin{equation*}
||\Q[\mo]_t||_{1\to\infty} \leq c_{\mo} t^{-\frac{\lam-\mo}{\lam-\mo-1}},
\end{equation*}
where $c_{\mo}$ is the Sobolev constant for $\Q[\mo]$ of exponent $p = \frac{2(\lam-\mo)}{(\lam-\mo-1)}$, given by
\begin{equation*}
c_{\mo} = \inf_{f \in \mathcal{D}_{\Leb^2}(\calJd)} \left\lbrace \frac{||f||_2^2-||f||_p^2}{\gab[\mo][f \calJd_\mo f]} : \gab[\mo][f \calJd_\mo f] \neq 0 \right\rbrace.
\end{equation*}
The fact that $\Q[\mo]$ is a contraction on $\Leb^1(\gab[\mo])$ together with the above ultracontractive bound yields $||\Q[\mo]_t||_{1\to\infty} \leq c_{\mo}$ for any $t \geq 1$. Finally, we mention that $c_{\frac{\lam}{2}}=\frac{4}{\lam(\lam- 2)}$, and upper and lower bounds are known in the general case; see again~\cite{Bakry-Jacobi}.
\end{appendix}

\bibliographystyle{abbrv}

\end{document}